\documentclass[11pt,a4paper,reqno]{amsart}
\usepackage[english]{babel}
\usepackage{mathtools}
\usepackage{amsmath,amsthm,amssymb,amsfonts}
\usepackage{thmtools}
\usepackage{mathrsfs}
\usepackage{enumitem}
\usepackage{tikz}
\usepackage{tikz-cd}
\numberwithin{equation}{section}

\usepackage{graphicx}
\usepackage[left=3cm,right=3cm,top=2.8cm,bottom=2.8cm]{geometry} 

\usepackage[pagebackref=true]{hyperref}

\renewcommand*{\backref}[1]{}
\renewcommand*{\backrefalt}[4]{{\tiny%
    (\ifcase #1 Not cited.%
          \or Cited on p.~#2.%
          \else Cited on pp. #2.%
    \fi%
    )}}
\newcommand{\Z}{\mathbb{Z}}
\newcommand{\Q}{\mathbb{Q}}
\newcommand{\R}{\mathbb{R}}
\newcommand{\C}{\mathbb{C}}

\newcommand{\A}{\mathbb{A}}
\newcommand{\HH}{\mathbb{H}}
\newcommand{\PP}{\mathbb{P}}
\newcommand{\Qal}{\overline{\mathbb{Q}}}

\newcommand{\cA}{\mathcal{A}}
\newcommand{\D}{\mathcal{D}}
\newcommand{\cC}{\mathcal{C}}
\newcommand{\cO}{\mathcal{O}}
\newcommand{\cE}{\mathcal{E}}

\renewcommand{\a}{\alpha}
\renewcommand{\b}{\beta}
\newcommand{\g}{\gamma}
\renewcommand{\d}{\delta}
\newcommand{\f}{\varphi}
\newcommand{\eps}{\varepsilon}
\renewcommand{\l}{\lambda}
\newcommand{\s}{\sigma}
\newcommand{\m}{\mu}

\renewcommand{\r}{\rho}
\renewcommand{\t}{\theta}

\newcommand{\ii}{\mathbf{i}}

\newcommand{\pt}[1]{\left( #1 \right)}
\newcommand{\pq}[1]{\left[ #1 \right]}
\newcommand{\pg}[1]{\left\lbrace #1 \right\rbrace}

\newcommand{\abs}[1]{\left| #1 \right|}
\newcommand{\nm}[1]{\left\lVert #1 \right\rVert}
\newcommand{\nmi}[1]{\left\lVert #1 \right\rVert_{\infty}}

\renewcommand{\Im}{\mathrm{Im}}
\renewcommand{\Re}{\mathrm{Re}}

\newcommand{\RI}[1]{#1^{\dagger}}
\newcommand{\quotes}[1]{``#1''}

\DeclareMathOperator{\en}{End}
\renewcommand{\hom}{\mathrm{Hom}}
\DeclareMathOperator{\Gal}{Gal}
\DeclareMathOperator{\GL}{GL}
\DeclareMathOperator{\tr}{tr}
\DeclareMathOperator{\mat}{Mat}
\DeclareMathOperator{\rk}{rk}

\newtheorem{theorem}{Theorem}[section]
\newtheorem{proposition}[theorem]{Proposition}
\newtheorem{corollary}[theorem]{Corollary}

\newtheorem{lemma}[theorem]{Lemma}
\theoremstyle{remark}
\newtheorem{remark}[theorem]{Remark}
\theoremstyle{definition}
\newtheorem{example}[theorem]{Example}
\theoremstyle{definition}
\newtheorem{definition}[theorem]{Definition}

\author[Luca Ferrigno]{Luca Ferrigno}
\address{Laboratoire de Mathématiques Blaise Pascal, Université Clermont Auvergne, Campus des Cézeaux 3, place
Vasarely, 63178 Aubière, France}
\email{lucaferrigno.math@gmail.com}
\title[Unlikely intersections and canonical height bounds]{Unlikely intersections with CM abelian varieties in a family and explicit bounds for canonical heights under endomorphisms}



\begin{document}

\begin{abstract}
Let $\cA \rightarrow S$ be a non-isotrivial abelian scheme over a smooth irreducible curve and let $\cC \subseteq \cA$ be an irreducible curve, all defined over $\Qal$. In 2020, Barroero and Capuano proved that, if $\cC$ is not contained in a proper subgroup scheme, then its intersection with the union of the flat subgroup schemes of $\cA$ of codimension at least $2$ is finite. We complement this result by proving that, if $\cC$ is contained neither in a fixed fiber nor in a proper flat subgroup scheme, even after a finite base change, then only finitely many points of $\cC$ lie in a proper algebraic subgroup of a CM fiber of $\cA$. This generalizes a previous result of Barroero for fibered powers of elliptic schemes.
A key ingredient of the proof is an explicit control of canonical heights under endomorphisms: for an abelian variety $A/\Qal$, an ample symmetric divisor $D$ and $f \in \en(A)$, we bound $\widehat{h}_{A,D}(f(P))$ from above and from below in terms of $\widehat{h}_{A,D}(P)$, with optimal constants given by the largest and the smallest eigenvalue of $\RI{f}f$, where $\RI{}$ is the Rosati involution associated with $D$. 
\end{abstract}
\maketitle
\tableofcontents
\section{Introduction}
Let $S$ be a smooth, irreducible, quasi-projective curve, and let $\pi : \cA \rightarrow S$ be an abelian scheme of relative dimension $g\geq 1$, both defined over a number field $k$. For any (not necessarily closed) point $s \in S$ we denote the fiber of $\cA$ over $s$ by $\cA_s$. If $s \in S(\C)$, then $\cA_s$ is an abelian variety of dimension $g$ defined over $k(s)$. 
Let $O : S \rightarrow \cA$ be the zero section of $\cA$ and consider an irreducible curve $\cC \subseteq \cA$, also defined over $k$. 

Recall that an irreducible component of a subgroup scheme of $\cA$ is either a component of an algebraic subgroup of a fiber or it dominates the base curve $S$. We say that a subgroup scheme is \emph{flat} if all of its irreducible components are of the latter kind (for a reduced subgroup scheme this is equivalent to being flat over $S$, by \cite[Proposition III.9.7]{Har77}).

We call $\cA \rightarrow S$ \emph{isotrivial} if it becomes constant after a base change, i.e.\ $\cA \times_S S' \cong A \times_{\Qal} S'$ for a finite surjective morphism $S'\to S$ of smooth irreducible curves and some fixed abelian variety $A/\Qal$. We say that a section $\s: S \rightarrow \cA$ is \emph{constant} if there exist an abelian variety $A_1$ defined over $\Qal$, a homomorphism of abelian schemes $\psi: A_1 \times S \rightarrow \cA$ and $a \in A_1(\Qal)$ such that $\s(s)=\psi(a,s)$ for every $s \in S$. Moreover, we say that $\cC$ is contained in a translate of a proper flat subgroup scheme by a constant section \emph{after a finite base change} if there exist a smooth irreducible curve $S'$ with a finite surjective morphism $S' \rightarrow S$ and an irreducible component of the preimage of $\cC$ in $\cA \times_S S'$ which is contained in a translate of a proper flat subgroup scheme of $\cA \times_S S'$ by a constant section.
We say that $\cC$ is \emph{transverse} if it is contained neither in a fixed fiber nor in a translate of a proper flat subgroup scheme of $\cA$ by a constant section, even after a finite base change.

We are interested in understanding how the curve $\cC$ intersects the subgroup schemes of the abelian scheme $\cA$.
In \cite{BC20}, Barroero and Capuano studied the intersections of $\cC$ with flat subgroup schemes of codimension at least 2 and proved that, if $\cC$ is not contained in a proper subgroup scheme, then its intersection with the union of all such flat subgroup schemes is finite. As a matter of fact, given a flat subgroup scheme $H$ of $\cA$ of codimension at least 2, one expects, for dimensional reasons, that $\cC \cap H$ should be empty. While this may not be true in general, their result confirms the prediction, arising from the Zilber--Pink conjecture, that the union of all these intersections is finite.

In general, the Zilber--Pink conjecture, formulated independently and in various settings by Bombieri, Masser and Zannier \cite{BMZ99}, by Zilber \cite{Zil02} and by Pink \cite{Pin05}, predicts that \quotes{unlikely intersections} between a fixed algebraic subvariety $V$ of a semiabelian or Shimura variety $X$ and \quotes{special} subvarieties of $X$ of codimension at least $\dim V +1$ should be scarce, i.e.\ there should be only finitely many maximal such intersections. In particular, one expects that if $V$ is not itself contained in any proper special subvariety, then its intersection with the union of all special subvarieties of codimension at least $\dim V +1$ is not Zariski-dense. For comprehensive treatments of the conjecture and of problems of unlikely intersections, see \cite{Zan12} and \cite{Pil22}, as well as the survey \cite{Cap23}. In the relative setting of an abelian scheme $\cA \rightarrow S$, the special subvarieties of $\cA$ are the irreducible components of flat subgroup schemes and the torsion translates of abelian subvarieties of the fibers $\cA_s$ such that the Mumford--Tate group of $\cA_s$ is strictly smaller than that of a very general fiber (equivalently, if $S\subseteq\A_g$, such that $s$ lies in a special subvariety of $\A_g$ not containing $S$). As $\cA$ is not isotrivial, this includes all torsion cosets of the CM fibers.

In this paper, we complement the results of \cite{BC20} by considering the intersections of $\cC$ with the proper algebraic subgroups of the CM fibers of $\cA$, which, like the flat subgroup schemes considered there, are special subvarieties of $\cA$ of codimension at least $2$. Motivated by the same dimensional considerations and by the Zilber--Pink conjecture, we prove the following theorem.

\begin{theorem}\label{main_thm}
Let $S$ and $\cA \rightarrow S$ be as above and assume that $\cA$ is not isotrivial. Let $\cC \subseteq \cA$ be an irreducible curve defined over $\Qal$ that is contained neither in a fixed fiber nor in a proper flat subgroup scheme of $\cA$, even after a finite base change. Then, the intersection of $\cC$ with the union of all proper algebraic subgroups of the CM fibers of $\cA$ is a finite set.
\end{theorem}

Since a point of an abelian variety lies in a proper algebraic subgroup if and only if it lies in the kernel of a non-zero endomorphism, the theorem can be restated as follows: under the same assumptions as above, there are at most finitely many $P \in \cC(\C)$ such that $\cA_{\pi(P)}$ has complex multiplication and there exists a non-zero $f \in \en(\cA_{\pi(P)})$ such that $f(P)=O_{\pi(P)}$.

Theorem \ref{main_thm} generalizes a result of Barroero \cite{Bar19}, who treated the case of a fibered power of an elliptic scheme.

Theorem \ref{main_thm} also proves a special case of a strengthening of Conjecture 6.1 of \cite{Pin05}: Pink's conjecture only concerns algebraic subgroups of the fibers of codimension at least $2$, while we also allow subgroups of codimension $1$, at the cost of restricting to CM fibers. In fact, Theorem \ref{main_thm} is a special case of the Zilber--Pink conjecture for a curve in an abelian scheme, which implies Conjecture 6.1 of \cite{Pin05} for abelian schemes (see \cite[Theorem 6.3]{Pin05}).

For a curve in an abelian variety defined over $\Qal$, the Zilber--Pink conjecture was proved by Habegger and Pila \cite[Theorem 9.14]{HP16}. On the other hand, to the best of our knowledge, the Zilber--Pink conjecture for a curve in a non-isotrivial abelian scheme has been settled only for a curve in a fibered power of an elliptic scheme, by work of Barroero and Capuano \cite{BC16}, Barroero \cite{Bar19} and Barroero and Dill \cite{BD25}, building on previous works by Masser and Zannier \cite{MZ10, MZ12}. Partial results for a curve in a product of fibered powers of elliptic schemes were obtained by Masser and Zannier \cite{MZ14}, Barroero and Capuano \cite{BC17} and (under additional assumptions on the curve) by previous work of the author \cite{Fer26}.  A variation of the conjecture involving tangential intersections has also been studied by Corvaja, Demeio, Masser and Zannier \cite{CDMZ21}, by Ulmer and Urzúa \cite{UU20, UU21}, and by Ottolini \cite{Ott27}.

\begin{remark}
Before proceeding, we note that if $S \subseteq \A_g$ is not a special (i.e.\ Shimura) curve (as explained in Section \ref{sec_red}, we may always assume $S \subseteq \A_g$), then the André--Oort conjecture for $\A_g$ (proved by Tsimerman \cite{Tsi18}) guarantees that only finitely many points $s \in S(\C)$ correspond to CM fibers $\cA_s$, which in turn implies Theorem \ref{main_thm}. Hence, one may assume that $S$ is a Shimura curve, though this assumption will not be used in the rest of the paper.
\end{remark}

\begin{remark}\label{rmk:ZP_hypotheses}
The hypotheses of Theorem \ref{main_thm} cannot be dropped. For example, if $\cC\subseteq \s+\mathcal{H}$, with $\s$ a torsion section of order $N$ and $\mathcal{H}$ a proper flat subgroup scheme, then every $P\in\cC(\C)$ lies in the proper algebraic subgroup $[N]^{-1}(\mathcal{H}_{\pi(P)})$ of $\cA_{\pi(P)}$. Hence, if $\cC$ is not contained in a fiber, the set of points of $\cC$ lying in a proper algebraic subgroup of a CM fiber is infinite whenever $S$ contains infinitely many CM points. On the other hand, Theorem \ref{main_thm} allows $\cC$ to be contained in a translate of a proper flat subgroup scheme by a non-torsion constant section, a case which requires some additional work (see below).

The hypothesis that $\cC$ is not contained in a fixed fiber cannot be dropped either. If $\cC\subseteq\cA_{s_0}$ and $\cA_{s_0}$ has CM, then $\cA_{s_0}$ is a special subvariety containing $\cC$, and its algebraic subgroups of codimension $1$, if any, are no longer unlikely to meet $\cC$, as their codimension in $\cA_{s_0}$ equals $\dim \cC$. For instance, assume that there is a surjective homomorphism $\phi:\cA_{s_0}\to E$ onto an elliptic curve which is not constant on $\cC$. Then $\phi(\cC)=E$, since $\cC$ is projective, and every $P\in\cC$ with $\phi(P)\in E_{\mathrm{tors}}$ lies in the proper algebraic subgroup $\phi^{-1}(\langle\phi(P)\rangle)$, so that infinitely many points of $\cC$ lie in proper algebraic subgroups of $\cA_{s_0}$. 

Similarly, the hypothesis that $\cA$ is not isotrivial cannot be dropped: if $\cA=E\times S$, with $E$ an elliptic curve with CM, and $\cC$ is the graph of a non-constant morphism $f:S\to E$, then $\cC$ satisfies all the other hypotheses of Theorem \ref{main_thm}, but the infinitely many points $(f(s),s)$ with $f(s)\in E_{\mathrm{tors}}$ lie in proper algebraic subgroups of CM fibers.
\end{remark}

Our proof of Theorem \ref{main_thm} follows the well-established Pila--Zannier strategy, first introduced in \cite{PZ08} and later used, among others, by Masser and Zannier \cite{MZ10, MZ12}, by Barroero and Capuano \cite{BC16, BC17, BC20}, by Barroero \cite{Bar19} and in previous work of the author \cite{Fer26}.

In Section \ref{sec_red} we reduce, up to isogeny and finite base change, to the case where $S$ is a curve in the moduli space $\A_g$ of principally polarized abelian varieties of dimension $g$ with level-$3$-structure and $\cA$ is the restriction to $S$ of the universal family. A subtle point is given by the curves contained in a translate of a proper flat subgroup scheme by a non-torsion constant section, which are allowed in Theorem \ref{main_thm} but cannot be treated directly with the functional transcendence result we use (Lemma \ref{lemma:func_transc_CM}, based on \cite{Dil21}), as the latter only applies to transverse curves. We show that, up to isogeny and finite base change, every curve as in Theorem \ref{main_thm} is of the form $\cC_1\times\pg{b}$ in $\mathcal{G}\times_S(B\times S)$, where $\cC_1$ is a transverse curve in an abelian scheme $\mathcal{G}$, $B$ is an abelian variety over $\Qal$ and $b\in B(\Qal)$ is not contained in any proper algebraic subgroup of $B$ (Lemma \ref{lemma:product_decomposition}). Thus, it suffices to show (Proposition \ref{prop:reduction_transverse_curve}) that, for $\cC$ transverse, there are at most finitely many $P \in \cC(\C)$ such that $\cA_{\pi(P)}\times B$ has CM and $(P,b)$ lies in a proper algebraic subgroup of $\cA_{\pi(P)}\times B$ (Theorem \ref{thm:main_product}). We denote by $\cE_b(\cC)$ the set of such points.

By a result of Peterzil and Starchenko \cite{PS13}, the restriction of the uniformizing map of the universal family to a suitable fundamental domain is definable in the o-minimal structure $\R_{\mathrm{an,exp}}$, so that the preimage $\mathcal{Z}$ of $\cC$ in this fundamental domain is a definable set of real dimension $2$. Each point $P_0 \in \cE_b(\cC)$ corresponds to a point of $\mathcal{Z}$ lying on an algebraic subvariety defined by equations whose coefficients are algebraic numbers of bounded degree. Combining a counting theorem of Habegger and Pila \cite[Corollary 7.2]{HP16} with Lemma \ref{lemma:func_transc_CM}, we show that, for every $\eps>0$, at most $\ll_{\eps} T^{\eps}$ points of $\mathcal{Z}$ lie on such subvarieties whose coefficients have height at most $T$ (Proposition \ref{prop:main_estim_CM}).

The remaining part of the proof is arithmetic in nature. Using estimates for canonical heights on $\cC$, Faltings heights of CM abelian varieties and CM period matrices, together with the height bounds discussed below, we construct a non-zero endomorphism of $\cA_{\pi(P_0)}\times B$ vanishing at $(P_0,b)$, whose Rosati norm is bounded by a constant times a power of $d_0\coloneqq\pq{k(P_0):\Q}$. Its analytic representation gives one of the subvarieties above, with coefficients of height $\ll d_0^{O(1)}$. Since the Galois conjugates of $P_0$ over $k$ also lie in $\cE_b(\cC)$, this yields at least $\pq{k(P_0):k}=d_0/\pq{k:\Q}$ points of $\mathcal{Z}$ on subvarieties of height $\ll d_0^{O(1)}$, while the counting result above gives at most $\ll_{\eps} d_0^{\eps}$ such points. Comparing the two bounds, we conclude that $d_0$ is uniformly bounded. Finally, since the Faltings height of a CM abelian variety is polynomially bounded in terms of the degree of its field of definition, the points $\pi(P_0)$ have bounded height and bounded degree, and the finiteness of $\cE_b(\cC)$ follows from Northcott's theorem.

The construction of the endomorphism above relies on an explicit control of the behaviour of canonical heights under endomorphisms, which is the second main result of this article and which we believe has independent interest. We briefly recall the setting.

Let $A$ be an abelian variety of dimension $g$ defined over $\Qal$ and let $D$ be a symmetric divisor on $A$, i.e.\ a divisor such that $[-1]^*D \sim D$. Then, $D$ induces a canonical (or Néron--Tate) height $\widehat{h}_{A,D}$ on $A(\Qal)$ (see \cite[Section B.5]{HS13}) and, since $D$ is symmetric, we have the classical identity
$$\widehat{h}_{A,D}([n]P) = n^2 \cdot \widehat{h}_{A,D}(P)$$
for every $n \in \Z$ and every $P \in A(\Qal)$. For the proof of Theorem \ref{main_thm} we need an analogue of this identity for an arbitrary endomorphism $f \in \en(A)$. In general, no identity of this form holds (see the examples in Section \ref{sec_can_bounds}). However, if $D$ is ample, one can show that there exist constants $0 \leq \g_1 \leq \g_2$, depending on $f$ and $D$, such that
$$\g_1 \cdot \widehat{h}_{A,D}(P)\leq \widehat{h}_{A,D}(f(P))\leq \g_2 \cdot \widehat{h}_{A,D}(P)$$
for every $P \in A(\Qal)$, and that $\g_1$ can be chosen positive if and only if $f$ is an isogeny. The standard proof of this fact, which we recall in Section \ref{sec_can_bounds}, does not give explicit values for $\g_1$ and $\g_2$, while in the proof of Theorem \ref{main_thm} we need to control $\g_2$ explicitly as $f$ varies in $\en(A)$ (see Lemma \ref{lemma:small_Ros_nm}).

The main result of Section \ref{sec_can_bounds} gives the optimal values of $\g_1$ and $\g_2$. Let $\RI{}$ be the Rosati involution defined by the polarization associated with $D$ and let $\a_1, \ldots, \a_g$ be the eigenvalues (counted with multiplicities) of the analytic representation of $\RI{f}f$ (see Section \ref{sect:prelim_AVs}), which are real and non-negative by Lemma \ref{lemma:eigen_pos}. Define
$$\a^{-}_{D}(f)=\min\!\pg{\a_1, \ldots, \a_g} \quad \text{ and } \quad \a^{+}_{D}(f)=\max\!\pg{\a_1, \ldots, \a_g}.$$

\begin{restatable}{theorem}{canheightsendo}\label{thm:can_heights_endomorphisms}
Let $A$ be an abelian variety defined over $\Qal$, and let $D$ be an ample symmetric divisor on $A$. Then, for every endomorphism $f : A \rightarrow A$, we have
$$\a^{-}_{D}(f) \cdot \widehat{h}_{A,D}(P)\leq \widehat{h}_{A,D}(f(P))\leq \a^{+}_{D}(f) \cdot \widehat{h}_{A,D}(P)$$
for every $P \in A(\Qal)$. Moreover, these constants are the best possible, meaning that we cannot replace $\a^{+}_{D}(f)$ and $\a^{-}_{D}(f)$ with a smaller and a larger constant, respectively.
\end{restatable}

In Section \ref{sec_can_bounds} we also extend Theorem \ref{thm:can_heights_endomorphisms} to homomorphisms between different abelian varieties (Theorem \ref{thm:can_heights_isogenies}), which in particular gives an optimal comparison between the canonical heights associated with two ample symmetric divisors on the same abelian variety.

\subsubsection*{Notation} In this article we will use Vinogradov's $\ll$ notation: for real-valued functions $f(T)$ and $g(T)$, we write $f(T) \ll g(T)$ if there exists a constant $c > 0$ such that $f(T) \leq c g(T)$ for all sufficiently large $T$. When not explicitly stated, the implied constant is either absolute or depends only on $S, \cA, g, \cC, B, b$ and other fixed data. We use subscripts to indicate any additional dependence of the implied constant.

We also denote by $\ii$ the imaginary unit. 

\section{Preliminaries}

\subsection{Abelian varieties and their endomorphisms}\label{sect:prelim_AVs}
In this subsection, we collect the basic definitions and results on complex abelian varieties that will be used in the rest of the paper. It is not intended as a comprehensive treatment of abelian varieties, for which we refer the reader to \cite{BL04, Mil08, Mum08}.

Here, every abelian variety is defined over $\C$ and we identify it with its set of complex points. Abelian varieties defined over $\Qal$, or over a number field, are regarded as complex abelian varieties via a fixed embedding $\Qal \hookrightarrow \C$.

It is well-known that if $A$ is an abelian variety of dimension $g$ defined over $\C$, then $A(\C)$ is a complex torus, i.e.\ $A(\C) \cong V/\Lambda$ for some $g$-dimensional $\C$-vector space $V$ and some lattice $\Lambda \subseteq V$. After fixing bases of $V$ and $\Lambda$, we have that $\Lambda=\Pi \Z^{2g}$, for some matrix $\Pi \in \mat_{g \times 2g}(\C)$ called a \emph{period matrix}, i.e.\ the matrix representing the basis of $\Lambda$ in terms of the basis of $V$.

Let $A, B$ be two abelian varieties. A \emph{homomorphism} is a morphism $f: A \rightarrow B$ of group varieties (in other words, it is a morphism of algebraic varieties which is also a group homomorphism). When $B=A$ such a map is called an \emph{endomorphism}. A homomorphism $f: A \rightarrow B$ is called an \emph{isogeny} if it is surjective and it has finite kernel.

We denote by $\hom(A,B)$ the set of homomorphisms from $A$ to $B$ and we define $\en(A)\coloneqq\hom(A,A)$ to be the set of all endomorphisms. Note that $\hom(A,B)$ is an abelian group under point-wise addition and, similarly, $\en(A)$ is a ring where the multiplication is given by composition of maps. Moreover, we define
$$\hom^0(A,B)\coloneqq \hom(A,B) \otimes \Q \quad \text{and} \quad \en^0(A)\coloneqq\en(A) \otimes \Q.$$
 We will always assume that all the morphisms are defined over an algebraic closure of the ground field.

Given an endomorphism $f$ of $A=V/\Lambda$, by Proposition 1.2.1 of \cite{BL04}, there is a unique linear map $F:V \rightarrow V$ with $F(\Lambda) \subseteq \Lambda$ and inducing $f$ on $A$. The restriction $F_{\Lambda}$ of $F$ to $\Lambda$ is $\Z$-linear and completely determines both $F$ and $f$.

Fix bases of $V$ and $\Lambda$, and let $\Pi$ be the corresponding period matrix. With respect to these bases, $F$ and $F_{\Lambda}$ are given by matrices $\rho_a(f) \in \mat_g(\C)$ and $\rho_r(f) \in \mat_{2g}(\Z)$, respectively. Since $F(\Lambda) \subseteq \Lambda$, we must have
\begin{equation}\label{eqn:endom_matr_periods}
\rho_a(f) \cdot \Pi = \Pi \cdot \rho_r(f).
\end{equation}
The associations $f \mapsto \r_a(f)$ and $f \mapsto \r_r(f)$ extend to injective ring homomorphisms
\begin{align*}
\rho_a : \en^0(A) &\longrightarrow \mat_g(\C) \\
\rho_r : \en^0(A) &\longrightarrow \mat_{2g}(\Q)
\end{align*}
called the analytic representation and the rational representation of $\en^0(A)$, respectively. More generally, if $A=V/\Lambda$ and $A'=V'/\Lambda'$, every $\phi \in \hom^0(A,A')$ has an analytic and a rational representation $\r_a(\phi)$ and $\r_r(\phi)$, defined in the same way, and these are compatible with composition.

We denote by $\widehat{A} = \mathrm{Pic}^0(A)$ the dual abelian variety, i.e.\ the group of isomorphism classes of line bundles on $A$ that are algebraically equivalent to zero. Given a point $x \in A$, we denote by $T_x$ the translation-by-$x$ map.
If $L$ is an arbitrary line bundle on $A$, we have a homomorphism
\begin{equation}\label{eqn:def_Phi_L}
\setlength{\arraycolsep}{0pt}
\renewcommand{\arraystretch}{1.2}
  \begin{array}{ c c c c }
    \Phi_L:\; & {} A & {} \longrightarrow {} & \widehat{A} \\
     &{} x      & {} \longmapsto {} & T_x^*L \otimes L^{-1} 
  \end{array}
\end{equation}
and we call $K(L)$ its kernel. A \emph{polarization} is an isogeny $A \rightarrow \widehat{A}$ of the form $\Phi_L$ for some ample line bundle $L$. We say that a polarization is \emph{principal} if it is an isomorphism (i.e.\ $\deg \Phi_L = 1$). Recall that any two algebraically equivalent ample line bundles on $A$ define the same polarization.

We denote by $\chi(L)$ the Euler characteristic of $L$.

To any polarization $\Phi_L$ on $A$ corresponds a positive definite Hermitian form $H_L = c_1(L): V \times V \to \C$, given by the first Chern class of the line bundle $L$. It is worth noting that, in the literature, the term polarization may refer either to the ample line bundle $L$ (up to algebraic equivalence), the associated isogeny $\Phi_L$, or the Hermitian form $H_L$. These notions are equivalent; see, for example, Section 4.1 of \cite{BL04}. 
We denote by $E_L = \Im(H_L)$ the alternating Riemann form associated with $L$, which takes integer values on the lattice $\Lambda$.

Given an ample line bundle $L$ on $A$, there exists a basis of $\Lambda$, called a \emph{symplectic basis}, such that the alternating Riemann form $E_{L} : \Lambda \times \Lambda \rightarrow \mathbb{Z}$ is represented by the matrix
$$\begin{pmatrix}
0 & \mathbf{D} \\
-\mathbf{D} & 0
\end{pmatrix}$$ 
where $\mathbf{D}\coloneqq\mathrm{diag}(d_1, \ldots, d_g)$ is a diagonal matrix, with $d_1, \ldots, d_g$ positive integers such that $d_{i}$ divides $d_{i+1}$ for each $i=1, \ldots, g-1$. 
We call $\mathbf{D}$ the \emph{type} of the polarization $\Phi_L$ and we define the \emph{Pfaffian} of $E_{L}$ as $\mathrm{Pf}(E_{L})=\det(\mathbf{D})$ \cite[Section 3.2]{BL04}. The degree of the isogeny $\Phi_L$ is called the \emph{degree} of the polarization and it is easy to prove that it is equal to $\mathrm{Pf}(E_{L})^2=\det(E_L)$.  

Next, we define the Rosati (anti-)involution on $\en^0(A)$ with respect to the polarization $\Phi_L$ as:
\begin{equation}\label{eqn:def_rosati_inv}
\setlength{\arraycolsep}{0pt}
\renewcommand{\arraystretch}{1.2}
  \begin{array}{ c c c c }
    ^{\dagger} :& {} \en^0(A) & {} \longrightarrow {} & \en^0(A) \\
     		   &{} f      & {} \longmapsto {} & \RI{f}=\Phi_{L}^{-1} \circ \widehat{f} \circ \Phi_{L}
  \end{array}
\end{equation}
where $\widehat{f} \in \en^0(\widehat{A})$ denotes the dual of $f$ and, with a slight abuse of notation, we also denote by $\Phi_{L}$ the corresponding element of $\hom^0(A,\widehat{A})$. This map is $\Q$-linear and satisfies $\RI{(fg)}=\RI{g} \RI{f}$ for all $f,g \in \en^0(A)$. In particular, if $\Phi_L$ is a principal polarization, the Rosati involution restricts to an involution on $\en(A)$.

\subsection{Moduli spaces, universal families and their uniformizations}\label{sec_moduli}
Let $g\geq 1$ and $n\geq 3$ be integers. We denote by $\A_{g,\mathbf{1},n}$ the moduli space of principally polarized abelian varieties of dimension $g$ with symplectic level-$n$-structure, which is a fine moduli space \cite[Theorem 7.9]{MFK94}. In other words, there is a universal family $\pi:\mathfrak{A}_{g,\mathbf{1},n} \rightarrow \A_{g,\mathbf{1},n}$, which, like $\A_{g,\mathbf{1},n}$, is defined over $\Qal$. Here, a level-$n$-structure is a symplectic basis of $A[n]$ with respect to the Weil pairing, normalized with respect to a fixed primitive $n$-th root of unity; with this normalization, $\A_{g,\mathbf{1},n}$ and $\mathfrak{A}_{g,\mathbf{1},n}$ are irreducible.

We will show in Section \ref{sec_red} that we can always reduce the problem to studying principally polarized abelian varieties, and the choice of the level structure is not important for our proof of Theorem \ref{main_thm}. So, for the rest of the article, we fix $n=3$ and write $\A_g=\A_{g,\mathbf{1},3}$ and $\mathfrak{A}_g=\mathfrak{A}_{g,\mathbf{1},3}$.

It is well-known (see for example Chapter 8 of \cite{BL04}) that $\A_g^{an}$, the analytification of $\A_g$, can be realized as the quotient of
$$\HH_g\coloneqq\pg{Z \in \mat_g(\C): Z=Z^t, \, \Im(Z)>0}$$
by the principal congruence subgroup $\Gamma(3)=\ker\pt{\mathrm{Sp}_{2g}(\Z)\to\mathrm{Sp}_{2g}(\Z/3\Z)}$ of $\mathrm{Sp}_{2g}(\Z)\coloneqq\pg{M \in \mat_{2g}(\Z): M^t J M=J}$, where $J\coloneqq\begin{psmallmatrix} 0 & \mathbf{1}_g\\ -\mathbf{1}_g & 0\end{psmallmatrix}$, which acts on $\HH_g$ by
$$\begin{pmatrix}
A & B\\
C & D
\end{pmatrix} \cdot Z= (A Z+B)(C Z+D)^{-1}.$$
Note that $\HH_g$ is an open subset, in the Euclidean topology, of 
$$\mathrm{Sym}_g(\C)\coloneqq\pg{M \in \mat_g(\C): M=M^t} \cong \C^{\frac{g(g+1)}{2}}$$
and that we can see $\HH_g$ as a semialgebraic subset of $\R^{2g^2}$, by identifying a complex number with its real and imaginary parts.
Furthermore, the quotient map $u_{b}: \HH_g \rightarrow \A_g^{an}$ is holomorphic.

Similarly, we have a holomorphic uniformization map for the universal family, given by theta functions, $u: \HH_g \times \C^g \rightarrow \mathfrak{A}_g^{an}$, such that the following diagram commutes:

\begin{center}
\begin{tikzcd}
\HH_g \times \C^g \arrow[d, "p_1"'] \arrow[r, "u"] & \mathfrak{A}_g^{an} \arrow[d, "\pi"] \\
\HH_g \arrow[r, "u_b"']                     & \A_g^{an}                           
\end{tikzcd}
\end{center}

The group $\Gamma(3)$ is neat, and $u$ is the quotient map of $\HH_g \times \C^g$ by the natural action of $\Gamma(3)\ltimes\Z^{2g}$.

Now, we would like to find a semialgebraic subset of $\HH_g \times \C^g$ on which $u$ is surjective and finite-to-one.

By \cite[Section V.4]{Igu72}, there is a semialgebraic subset $\mathfrak{F}_g$ of $\HH_g$ which can be used as a fundamental domain for the action of $\mathrm{Sp}_{2g}(\Z)$ on $\HH_g$. If $\Gamma$ is a finite index subgroup of   $\mathrm{Sp}_{2g}(\Z)$ and $\s_1=\mathbf{1}_{2g}, \s_2, \ldots, \s_m$ is a complete set of representatives of its right cosets, then
\begin{equation}\label{eqn:def_F_Gamma}
\mathfrak{F}_{\Gamma}\coloneqq\bigcup\limits_{i=1}^m \s_i \cdot \mathfrak{F}_g
\end{equation} 
is called a \emph{Siegel fundamental domain for $\Gamma$} and can be used as a fundamental domain for the action of $\Gamma$ on $\HH_g$.

For a fixed $\tau \in \HH_g$ we have a principally polarized abelian variety $A_{\tau}=\C^g/(\Z^g + \tau \Z^g)$. Let $L_{\tau}\coloneqq\pg{z \in \C^g: z=p+\tau q \text{ with } p,q \in [0,1)^g}$ be the fundamental parallelogram for the lattice $\Z^g + \tau \Z^g$. Moreover, let $\Gamma=\Gamma(3)$ and define
$$\mathcal{F}_g\coloneqq\pg{(\tau, z) \in \HH_g \times \C^g: \tau \in \mathfrak{F}_{\Gamma}, z \in L_{\tau}}.$$
Then, the restriction of $u$ to $\mathcal{F}_g$ is surjective and finite-to-one. 

Finally, let $S \subseteq \A_g$ be a smooth, irreducible, locally closed curve, $\cA=\mathfrak{A}_g \times_{\A_g} S \rightarrow S$ and let $\cC\subseteq\cA$ be an irreducible curve. We set
\begin{equation}\label{def:Z_CM}
\mathcal{Z}=u^{-1}(\cC(\C)) \cap \mathcal{F}_g.
\end{equation}
 Define the constant part of $\cA \rightarrow S$ as the largest abelian subvariety $A_0$ of $\cA_{\eta}\otimes \overline{\Qal(S)}$ which can be defined over $\Qal$. By the universal property of the $\overline{\Qal(S)}/\Qal$-trace \cite[Section VIII.3]{Lan83}, $A_0$ is the image of the trace homomorphism, and we denote by $A_0(\Qal)$ the image of the $\Qal$-points of the trace.

Let $D$ be an open disc in $\cC(\C)$ containing no singular point of $\cC$ and consider $\tau$ and $z$ as holomorphic functions on $D$ such that $u(\tau(P),z(P))=P$ for every $P \in D$. The following functional transcendence result is a consequence of Theorem 7.1 of \cite{Dil21} (which is in turn based on results by Gao \cite{Gao17, Gao20}).\footnote{Theorem 7.1 of \cite{Dil21} is stated for the universal family with orthogonal level $l$-structure. However, the only ingredients of its proof that depend on the level structure are \cite[Theorem 8.1]{Gao17}, which holds for every connected mixed Shimura variety, and \cite[Proposition 5.3]{Gao20}, which is stated for the universal family of abelian varieties with level-$N$-structure for every $N\geq 3$, uniformized by a congruence subgroup of $\Z^{2g}\rtimes\mathrm{Sp}_{2g}(\Z)$ (see \cite[Sections 2.2 and 3.2]{Gao20} and the remark at the end of the proof of \cite[Theorem 7.1]{Dil21}). Hence, the same proof applies to $\mathfrak{A}_g \rightarrow \A_g=\A_{g,\mathbf{1},3}$, since $u$ is the quotient map of $\HH_g \times \C^g$ by $\Gamma(3) \ltimes \Z^{2g}$.}
\begin{lemma}\label{lemma:func_transc_CM}
Let $S$, $\cA$, $\cC$ and $D$ be as above and let $F = \C(\tau)$ be the field generated over $\C$ by the entries of $\tau$, regarded as holomorphic functions on $D$. If $\cC$ is transverse, then $\mathrm{tr.deg.}_F F(z) = g$ on $D$, where $F(z)$ is generated over $F$ by the coordinates of $z$.
\end{lemma}
\begin{proof}
By contradiction, assume that $\mathrm{tr.deg.}_F F(z) < g$ and set $t=\mathrm{tr.deg.}_{\C} \C(\tau)$. Let $\widetilde{Y}$ be the irreducible analytic component of $u^{-1}(\cC(\C))$ containing $(\tau,z)(D)$ and let $\Sigma$ be the connected component of $u_b^{-1}(S(\C))$ containing $\tau(D)$. Since $(\tau,z)(D)$ and $\tau(D)$ are open subsets of $\widetilde{Y}$ and $\Sigma$, respectively, their Zariski closures in $\mathrm{Sym}_g(\C) \times \C^g$ and $\mathrm{Sym}_g(\C)$ satisfy
$$\dim \widetilde{Y}^{\mathrm{Zar}}=t+\mathrm{tr.deg.}_F F(z)<t+g \qquad \text{and} \qquad \dim \Sigma^{\mathrm{Zar}}=t.$$
Since $\pi$ is proper and $\cC$ is not contained in a fiber, we have $\pi(\cC)=S$. Then, Theorem 7.1 of \cite{Dil21} implies that exactly one complex analytic irreducible component $\widetilde{\mathcal{W}}$ of $\widetilde{Y}^{\mathrm{Zar}}\cap \pt{\HH_g \times \C^g}$ contains $\widetilde{Y}$ (the Zariski closure in $\mat_g(\C)\times\C^g$ considered in \cite{Dil21} coincides with the one in $\mathrm{Sym}_g(\C)\times\C^g$, since $\mathrm{Sym}_g(\C)$ is a linear subspace of $\mat_g(\C)$), and that there is a subvariety $\mathcal{W}$ of $\cA$, containing $\cC$, such that $\mathcal{W}(\C)=u\pt{\widetilde{\mathcal{W}}\cap u^{-1}(\cA(\C))}$ and, over $\overline{\Qal(S)}$, every irreducible component of $\mathcal{W}_{\eta}$ is a translate of an abelian subvariety of $\cA_{\eta}$ by a point in $\pt{\cA_{\eta}}_{\mathrm{tors}}+A_0(\Qal)$.

We claim that $\mathcal{W}\neq \cA$. Otherwise, since $u^{-1}(\cA(\C))$ is the countable union of the sets $\g\Sigma \times \C^g$, with $\g \in \Gamma$, a dimension argument shows that $\g\Sigma \times \C^g \subseteq \widetilde{\mathcal{W}}$ for some $\g \in \Gamma$. As $\g$ acts on $\mathrm{Sym}_g(\C)$ by a birational map, this implies $\dim \widetilde{Y}^{\mathrm{Zar}} \geq \dim \pt{\g\Sigma}^{\mathrm{Zar}}+g=t+g$, a contradiction.

Therefore, since $\cA$ is irreducible of dimension $g+1$, every irreducible component of $\mathcal{W}$ has dimension at most $g$, so that every irreducible component of $\mathcal{W}_{\eta}$ over $\overline{\Qal(S)}$ has dimension at most $g-1$. If $P$ is a geometric generic point of $\cC$, then $P$ lies in a component of $\mathcal{W}_{\eta}$ of the form $T+a+B$, where $T$ is a torsion point of order $N$, $a \in A_0(\Qal)$ and $B$ is an abelian subvariety of $\cA_{\eta}$, which is proper by the above. Up to a finite base change $S'\to S$, we can assume that $P$, $T$, $B$ and the inclusion $A_0 \hookrightarrow \cA_{\eta}$ are defined over $\Qal(S')$, so that the latter extends to a homomorphism $A_0 \times S' \rightarrow \cA'=\cA\times_S S'$ and $a$ defines a constant section of $\cA'$. Then, the Zariski closure $\mathcal{B}$ of $B$ in $\cA'$ is a flat subgroup scheme and $\mathcal{H}=[N]^{-1}(\mathcal{B})$ is a proper flat subgroup scheme of $\cA'$. Since $N(P-a)\in B$, the irreducible component of the preimage of $\cC$ in $\cA'$ containing $P$ is contained in the translate of $\mathcal{H}$ by the constant section defined by $a$, contradicting the transversality of $\cC$.
\end{proof}

\subsection{Heights}\label{sect:prelim_heights}
Let $h$ denote the logarithmic absolute Weil height on $\PP^N$, as defined in \cite[Chapter 1]{BG06} or \cite[Part B]{HS13} and, if $\a$ is an algebraic number, define $h(\a)=h\pt{\pq{1:\a}}$. Define also the multiplicative Weil height as $H(P)=\exp(h(P))$. More generally, if $V$ is a projective variety and $D$ is a divisor, denote by $h_{V,D}$ a Weil height on $V$ associated to $D$ (see \cite[Chapter 2]{BG06} or \cite[Section B.3]{HS13}).

For a number field $K$, we denote by $M_K$ its set of places, by $M_K^{\infty}$ and $M_K^0$ the subsets of archimedean and non-archimedean places, and, for $v \in M_K$, by $d_v=[K_v:\Q_v]$ the local degree and by $\abs{\cdot}_v$ the absolute value on $K$ extending the usual (archimedean or $p$-adic) absolute value of $\Q$. 
Let $X=(x_{i,j}) \in \mat_{n}(\Qal)$. We associate to $X$ two natural heights:
\begin{itemize}
\item the \emph{affine height}, defined by
$$H_{\mathrm{aff}}(X)=\prod\limits_{v \in M_K}\max\!\pg{1, \max\limits_{1\leq i,j \leq n}\pg{\abs{x_{i,j}}_v}}^{\frac{d_v}{[K:\Q]}}$$
where $K$ is a number field containing all the entries of $X$. This coincides with the absolute multiplicative Weil height of $X$ regarded as a point of $\Qal^{n^2}$;
\item the \emph{entry-wise height}, defined by
$$H_{\mathrm{max}}(X)=\max\limits_{1\leq i,j \leq n}\pg{H(x_{i,j})}.$$
\end{itemize}
The affine and entry-wise heights enjoy many useful properties with respect to usual matrix operations, which we now collect.
\begin{proposition}\label{prop:properties_mat_heights}
Let $A, B \in \mat_n(\Qal)$. Then:
\begin{enumerate}
\item $H_{\mathrm{max}}(A) \leq H_{\mathrm{aff}}(A) \leq H_{\mathrm{max}}(A)^{n^2}$;
\item $H_{\mathrm{max}}(A+B) \leq 2 H_{\mathrm{max}}(A) H_{\mathrm{max}}(B)$;
\item $H_{\mathrm{max}}(AB) \leq n H_{\mathrm{max}}(A)^n H_{\mathrm{max}}(B)^n$;
\item $H(\det(A)) \leq n! \cdot H_{\mathrm{aff}}(A)^n$;
\item if $A$ is invertible, $H_{\mathrm{max}}(A^{-1})\leq n! \cdot (n-1)! \cdot H_{\mathrm{aff}}(A)^{2n-1}$.
\end{enumerate}
\end{proposition}
\begin{proof}
Let $A=(a_{i,j})$ and $B=(b_{i,j})$ and fix a number field $K$ containing all entries of $A$ and $B$.
\begin{enumerate}
\item Since $\max\!\pg{1, \abs{a_{i,j}}_v} \leq \max\!\pg{1, \max\limits_{1\leq i,j \leq n}\pg{\abs{a_{i,j}}_v}}$, we clearly have 
$$H(a_{i,j})=\prod\limits_{v \in M_K}\max\!\pg{1, \abs{a_{i,j}}_v}^{\frac{d_v}{[K:\Q]}} \leq H_{\mathrm{aff}}(A)$$
which implies that $H_{\mathrm{max}}(A) \leq H_{\mathrm{aff}}(A)$. Moreover, recall that
$$\max\!\pg{1, \max\limits_{1\leq i,j \leq n}\pg{\abs{a_{i,j}}_v}} \leq \prod\limits_{1 \leq i,j \leq n}\max\!\pg{1, \abs{a_{i,j}}_v}$$ 
which implies that $H_{\mathrm{aff}}(A) \leq \prod_{1 \leq i,j \leq n}H(a_{i,j}) \leq H_{\mathrm{max}}(A)^{n^2}$.
\item The claim follows from the inequality 
$$H(a_{i,j}+b_{i,j})\leq 2 H(a_{i,j}) H(b_{i,j}) \leq 2 H_{\mathrm{max}}(A) H_{\mathrm{max}}(B),$$ 
which is a direct consequence of \cite[Proposition 1.5.15]{BG06}.
\item Let $AB=(c_{i,j})$, where $c_{i,j}=\sum\limits_{k=1}^{n} a_{i,k} b_{k,j}$. Then, applying \cite[Proposition 1.5.15]{BG06} and the fact that $H(\a \b)\leq H(\a) H(\b)$ for all $\a,\b \in \Qal$, yields
$$H(c_{i,j}) \leq n \cdot \prod\limits_{k=1}^{n} H(a_{i,k}) H(b_{k,j}) \leq n H_{\mathrm{max}}(A)^n H_{\mathrm{max}}(B)^n,$$
which implies $H_{\mathrm{max}}(AB) \leq n H_{\mathrm{max}}(A)^n H_{\mathrm{max}}(B)^n$.
\item Recall that 
$$\det(A)=\sum\limits_{\s \in S_n} \mathrm{sgn}(\s)\prod_{i=1}^{n} a_{i, \s(i)}$$
where $S_n$ denotes the symmetric group on $n$ elements and $\mathrm{sgn}(\s) \in \pg{\pm 1}$ is the sign of the permutation $\s$. Hence $\det(A)$ is the sum of $n!$ monomials of degree $n$ in the entries of $A$. In particular, for every place $v \in M_K$, we have
$$\abs{\det(A)}_v \leq \begin{cases}
n! \cdot \max\limits_{1\leq i,j \leq n}\pg{\abs{a_{i,j}}_v}^n \quad &\text{if } v \text{ is archimedean}\\
\max\limits_{1\leq i,j \leq n}\pg{\abs{a_{i,j}}_v}^n \quad &\text{if } v \text{ is non-archimedean}
\end{cases}$$
Hence,
\begin{align*}
\prod\limits_{v \in M^0_K}\max\!\pg{1, \abs{\det(A)}_v}^{\frac{d_v}{[K:\Q]}} &\leq  \prod\limits_{v \in M^0_K}\!\pt{\max\!\pg{1, \max\limits_{1\leq i,j \leq n}\pg{\abs{a_{i,j}}_v}}^n}^{\frac{d_v}{[K:\Q]}}\\
&= \pt{\prod\limits_{v \in M^0_K}\max\!\pg{1, \max\limits_{1\leq i,j \leq n}\pg{\abs{a_{i,j}}_v}}^{\frac{d_v}{[K:\Q]}}}^n
\end{align*}
and 
\begin{align*}
\prod\limits_{v \in M^{\infty}_K}\max\!\pg{1, \abs{\det(A)}_v}^{\frac{d_v}{[K:\Q]}} &\leq  \prod\limits_{v \in M^{\infty}_K}\!\pt{n! \max\!\pg{1, \max\limits_{1\leq i,j \leq n}\pg{\abs{a_{i,j}}_v}}^n}^{\frac{d_v}{[K:\Q]}}\\
&= (n!)^{\sum\limits_{v \in M^{\infty}_K}\frac{d_v}{[K:\Q]}}\pt{\prod\limits_{v \in M^{\infty}_K}\max\!\pg{1, \max\limits_{1\leq i,j \leq n}\pg{\abs{a_{i,j}}_v}}^{\frac{d_v}{[K:\Q]}}}^n\\
&=n!\pt{\prod\limits_{v \in M^{\infty}_K}\max\!\pg{1, \max\limits_{1\leq i,j \leq n}\pg{\abs{a_{i,j}}_v}}^{\frac{d_v}{[K:\Q]}}}^n
\end{align*}
since $\sum_{v \in M^{\infty}_K}d_v=[K:\Q]$. So, we have
\begin{align*}
H(\det(A))&=\prod\limits_{v \in M_K}\max\!\pg{1, \abs{\det(A)}_v}^{\frac{d_v}{[K:\Q]}}\\
&=\prod\limits_{v \in M^{\infty}_K}\max\!\pg{1, \abs{\det(A)}_v}^{\frac{d_v}{[K:\Q]}} \cdot \prod\limits_{v \in M^0_K}\max\!\pg{1, \abs{\det(A)}_v}^{\frac{d_v}{[K:\Q]}}\\
&\leq n! \cdot H_{\mathrm{aff}}(A)^n.
\end{align*}
\item The case $n=1$ is trivial, so assume $n\geq 2$. Recall that $A^{-1}=\frac{1}{\det(A)} \cdot C^t$, where $C=\pt{(-1)^{i+j} \m_{i,j}}$ is the cofactor matrix and $\m_{i,j}$ is the $(i,j)$-minor\footnote{Some authors use the word \emph{minor} to denote just the matrix obtained from $A$ by removing a row and a column. In this article, by minor we mean the determinant of such a submatrix.} of $A$. Then, by part (4), $H((-1)^{i+j} \m_{i,j}) \leq (n-1)! \cdot H_{\mathrm{aff}}(A)^{n-1}$, so that $H_{\mathrm{max}}(C)\leq (n-1)! \cdot H_{\mathrm{aff}}(A)^{n-1}$. Therefore, $H_{\mathrm{max}}(A^{-1}) \leq H(\det(A)) \cdot H_{\mathrm{max}}(C) \leq n! \cdot (n-1)! \cdot H_{\mathrm{aff}}(A)^{2n-1}$.\qedhere
\end{enumerate}
\end{proof}

We will also need another definition of height (from \cite[Section 7]{HP16}).
\begin{definition}\label{def_d_height}
If $d \in \Z_{\geq 1}$ and $\a$ is a complex number, we define the $d$-height of $\a$ as
$$H_d(\a)\coloneqq\min\pg{H\!\pt{\pq{a_0 : \ldots : a_d}}: \pq{a_0: \ldots: a_d} \in \PP^{d}(\Q) \text{ s.t. } a_0+ a_1 \a + \ldots + a_d \a^d=0}$$
where we use the convention $\min \emptyset= + \infty$. 
For $(\a_1, \ldots, \a_N) \in \C^N$, we also define $H_d(\a_1, \ldots, \a_N)=\max\pg{H_d(\a_i)}$. For a matrix $M$ with complex entries, we write $H_d(M)$ for the $d$-height of the tuple of its entries.
\end{definition}
Note that $H_d(\a_1, \ldots, \a_N)$ is finite if and only if $\a_1, \ldots, \a_N$ are all algebraic numbers of degree at most $d$. Moreover, since every point of $\PP^{d}(\Q)$ has multiplicative height at least $1$, and a relation $a_0+a_1\a+\ldots+a_d\a^d=0$ is also a relation of degree at most $n$ for every $n\geq d$, with the same height of the coefficients, we have
\begin{equation}\label{eqn:prop_hpoly}
1\leq H_{n}(\a)\leq H_d(\a) \quad \text{for every } n\geq d.
\end{equation}
In particular, $H_d(0)=1$ for every $d\geq 1$.

For every $p \in \C[x]$, let $M(p)$ denote the Mahler measure of $p$, defined by
$$M(p)\coloneqq\exp\left(\int_0^{1}\log\abs{p(e^{2\pi it})}\,dt\right),$$
as in \cite[Section~1.6]{BG06}.

\begin{lemma}\label{lemma:bounds_H_hpoly}
For any $\a \in \Qal$ of degree at most $d$ we have
$$\frac{1}{\sqrt{d+1}} H(\a)^{[\Q(\a):\Q]}\leq H_d(\a) \leq 2^d H(\a)^d.$$
\end{lemma}
\begin{proof}
Let $f (x) = a_0+ a_1 x + \ldots + a_n x^n \in \Q[x]$ be a polynomial of degree $n\leq d$ such that $f(\a)=0$ and let $m_{\a}(x) \in \Z[x]$ be the minimal polynomial of $\a$ (so its coefficients are coprime). Since in the definition of $H_d(\a)$ we are considering the coefficients of $f$ as a point in a projective space, we may assume that the coefficients of $f$ are integers with $\gcd(a_0, \ldots, a_n)=1$. 

By \cite[Proposition~1.6.6]{BG06}, 
$$M(m_{\a})=H\!\pt{\a}^{[\Q(\a):\Q]}\leq H(\a)^d.$$
 Moreover, \cite[Lemma 1.6.7]{BG06} gives
$$\nmi{f}\coloneqq\max\!\pg{\abs{a_0}, \ldots, \abs{a_n}} \leq \binom{n}{\left\lfloor \frac{n}{2}\right\rfloor} M(f)\leq 2^n M(f) \leq 2^d M(f).$$
Since the coefficients of $f$ are coprime integers, $H([a_0:\ldots:a_n])=\nmi{f}$. Hence
\begin{align*}
H_d(\a)&=\min\pg{\nmi{f}: f \in \Z[x] \text{ with coprime coefficients, } \deg(f) \leq d \text{ and } f(\a)=0}\\
&\leq \nmi{m_{\a}}\leq 2^d M(m_{\a})\leq 2^d H(\a)^d
\end{align*}
which proves the upper bound.

For the lower bound, let $H_{d}(\a)=B$. Then, there exists a non-zero polynomial 
$$f(x) = a_0+ a_1 x + \ldots + a_n x^n \in \Z[x]$$
of degree $n\leq d$ such that $f(\a)=0$ and $\nmi{f}=B$. Then, by \cite[Lemma 1.6.7]{BG06}, it follows that
$$M(f)\leq \sqrt{d+1} \cdot \nmi{f} =\sqrt{d+1} \cdot B.$$
Since $f(\a)=0$, the minimal polynomial $m_{\a}$ divides $f$, so we can write $f(x)=m_{\a}(x) \cdot g(x)$, for some $g \in \Z[x]$. By multiplicativity of the Mahler measure \cite[Lemma B.7.3.1 (ii)]{HS13} and the fact that $M(g)\geq 1$ (see \cite[Proposition~1.6.5]{BG06}), we obtain 
$$M(m_{\a}) \leq M(m_{\a}) \cdot M(g) =M(f)\leq \sqrt{d+1} \cdot B.$$
Finally, recalling that $M(m_{\a})=H\!\pt{\a}^{[\Q(\a):\Q]}$ and $H_{d}(\a)=B$, this proves the lower bound.
\end{proof}

\begin{lemma}\label{lemma:bound_mod_hpoly}
For any $\a \in \Qal$ of degree at most $d$ we have $\abs{\a} \leq \sqrt{d+1} \cdot H_d(\a)$.
\end{lemma}
\begin{proof}
We continue with the notations of the previous proof. Note first that \cite[Proposition 1.6.6]{BG06} implies $\abs{\a} \leq M(f)$ for any $f \in \Z[x]$ such that $f(\a)=0$. Furthermore, by \cite[Lemma 1.6.7]{BG06}, we also have that $M(f) \leq \sqrt{\deg(f)+1} \cdot \nmi{f}$. Taking the minimum over all polynomials $f \in \Z[x]$ with coprime coefficients and $\deg(f)\leq d$ such that $f(\a)=0$ then yields the desired bound $\abs{\a} \leq \sqrt{d+1} \cdot H_d(\a)$.
\end{proof}

\begin{lemma}\label{lemma:bounds_hpoly_ReIm}
Let $d \geq 2$, $\mathfrak{d}=\frac{d(d-1)}{2}$ and let $\a \in \Qal$ of degree at most $d$. Then
$$H_{2\mathfrak{d}}(\Re(\a)) \leq 2^{3\mathfrak{d}} \cdot (d+1)^{\mathfrak{d}/2} \cdot H_d(\a)^{\mathfrak{d}} \quad \text{and} \quad H_{2\mathfrak{d}}(\Im(\a)) \leq 2^{6\mathfrak{d}} \cdot (d+1)^{\mathfrak{d}} \cdot H_d(\a)^{2\mathfrak{d}}.$$ 
\end{lemma}
\begin{proof}
If $\a \in \R$, then $\Im(\a)=0$ and $\Re(\a)=\a$, so that (\ref{eqn:prop_hpoly}) gives $H_{2\mathfrak{d}}(\Im(\a))=1$ and, as $d\leq 2\mathfrak{d}$, $H_{2\mathfrak{d}}(\Re(\a))\leq H_d(\a)\leq H_d(\a)^{\mathfrak{d}}$. Thus, assume that $\Im(\a)\neq 0$, so that $\a$ has degree $2 \leq n \leq d$ and $\overline{\a}\neq \a$ is one of its conjugates.

We begin by bounding the degrees of the real and imaginary parts. Let $\a_1, \ldots, \a_n$ be the conjugates of $\a$. Then, for every $\s \in \Gal(\Qal/\Q)$, we have
$$\s(\Re(\a))=\dfrac{\s(\a)+\s(\overline{\a})}{2} \in \pg{\frac{\a_i+ \a_j}{2}: i\neq j}$$
which implies $[\Q(\Re(\a)): \Q] \leq \frac{n(n-1)}{2}\leq \mathfrak{d}$. Similarly, since $\s(\ii)=\pm \ii$, we have
$$\s(\Im(\a))=\dfrac{\s(\a)-\s(\overline{\a})}{2\s(\ii)} \in \pg{\frac{\a_i- \a_j}{2\ii}: i\neq j}$$
from which it follows that $[\Q(\Im(\a)): \Q] \leq n(n-1)\leq 2\mathfrak{d}$.

Next, by Equation (\ref{eqn:prop_hpoly}) and Lemma \ref{lemma:bounds_H_hpoly}, we have $H_{2\mathfrak{d}}(\Re(\a)) \leq H_{\mathfrak{d}}(\Re(\a)) \leq 2^{\mathfrak{d}} H(\Re(\a))^{\mathfrak{d}}$. Using standard properties of heights (\cite[Proposition 1.5.15]{BG06} and \cite[Proposition 3.1]{Zan14}), we get 
$$H(\Re(\a))=H\!\pt{\frac{\a+\overline{\a}}{2}}\leq 4 H(\a)^2.$$
Since $[\Q(\a):\Q]=n \ge 2$, by Lemma \ref{lemma:bounds_H_hpoly} we have $H(\a)^2 \leq \sqrt{d+1} \cdot H_d(\a)$ and therefore $H(\Re(\a))\leq 4\sqrt{d+1} \cdot H_d(\a)$. Combining the above inequalities yields
$$H_{2\mathfrak{d}}(\Re(\a)) \leq 2^{\mathfrak{d}} \cdot \pt{4\sqrt{d+1} \cdot H_d(\a)}^{\mathfrak{d}}=2^{3\mathfrak{d}}  \cdot (d+1)^{\mathfrak{d}/2} \cdot H_d(\a)^{\mathfrak{d}}.$$
Similarly, $H_{2\mathfrak{d}}(\Im(\a))\leq 2^{2\mathfrak{d}} H(\Im(\a))^{2\mathfrak{d}}$ and, arguing as above, 
$$H(\Im(\a)) \leq 4H(\a)^2 \leq 4\sqrt{d+1} \cdot H_d(\a).$$
Hence, $H_{2\mathfrak{d}}(\Im(\a)) \leq 2^{6\mathfrak{d}} \cdot (d+1)^{\mathfrak{d}} \cdot H_d(\a)^{2\mathfrak{d}}$.
\end{proof}

For an abelian variety $A$ defined over a number field, we denote by $h_F(A)$ its stable Faltings height (see \cite{Fal83}), i.e.\ the Faltings height of $A$ computed over a finite extension over which $A$ has semistable reduction everywhere, which does not depend on the choice of such an extension.

Finally, for an abelian variety $A$ defined over a number field and a divisor $D$, we denote by $\widehat{h}_{A,D}$ the Néron--Tate (or canonical) height associated with $D$, as defined in \cite[Chapter 9]{BG06} or \cite[Section B.5]{HS13}.
For the reader's convenience, we recall some properties of canonical heights on abelian varieties in the following proposition.

\begin{proposition}\label{prop:propr_can_heights}
Let $A$ be an abelian variety defined over a number field $K$, and let $D \in \mathrm{Div}(A)$ be a divisor on $A$. Then, the canonical height $\widehat{h}_{A,D}$ satisfies the following properties:
\begin{enumerate}
\item $\widehat{h}_{A,D}=h_{A,D}+O(1)$ and $\widehat{h}_{A,D}(O)=0$;
\item If $D' \in \mathrm{Div}(A)$ is linearly equivalent to $D$, then $\widehat{h}_{A,D'}=\widehat{h}_{A,D}$;
\item If $D, E \in \mathrm{Div}(A)$, then $\widehat{h}_{A,D+E}=\widehat{h}_{A,D}+\widehat{h}_{A,E}$;
\item Let $B/K$ be another abelian variety, and let $\phi:B \rightarrow A$ be a morphism. Then $\widehat{h}_{B,\phi^*D}=\widehat{h}_{A,D} \circ \phi - \widehat{h}_{A,D}(\phi(O_B))$;
\item If $D$ is symmetric (i.e.\ $[-1]^*D \sim D$), then $\widehat{h}_{A,D}([n]P)=n^2 \cdot \widehat{h}_{A,D}(P)$ for every $P \in A(\overline{K})$;
\item If $D$ is nef and symmetric, then $\widehat{h}_{A,D}(P)\geq 0$ for every $P \in A(\overline{K})$. Moreover, if $D$ is ample and symmetric, $\widehat{h}_{A,D}(P)= 0$ if and only if $P$ has finite order.
\end{enumerate}
\end{proposition}
\begin{proof}
The proofs of (1)--(5) can be found in \cite[Theorems B.5.1 and B.5.6]{HS13}, and the proof of (6) for $D$ ample and symmetric is given in \cite[Proposition B.5.3]{HS13}. The case of (6) when $D$ is nef is treated in \cite{KS16}, immediately after the displayed equation (6.14). For the reader's convenience, we briefly recall the argument.

If $D$ is symmetric and nef, then for any ample symmetric divisor $H$ and any integer $n\ge 1$, the divisor $nD+H$ is ample and symmetric by Kleiman's criterion. Using the linearity established in (3), we obtain
$$n\widehat{h}_{A,D}=\widehat{h}_{A,nD+H}- \widehat{h}_{A,H}\geq -\widehat{h}_{A,H}.$$
Since $H$ is ample, we have $\widehat{h}_{A,H} \geq 0$, and as $n > 0$ is arbitrary, it follows that $\widehat{h}_{A,D} \geq 0$ as well.
\end{proof}
\begin{remark}\label{rmk:zero_can_height}
Note that for any divisor $D$ and any torsion point $P \in A(\Qal)$, one has $\widehat{h}_{A,D}(P)=0$. The converse, however, does not hold in general: as shown in \cite{KS16}, when $D$ is nef the set of points of canonical height zero may strictly contain the torsion subgroup. 

For an explicit example with a non-nef divisor, let $E$ be an elliptic curve defined over $\Qal$, set $A=E\times E$, and consider $D=\pi_1^*(O)-\pi_2^*(O)=(O\times E)- (E \times O)$. It is easy to check that $D$ is not nef. Then, the proposition above implies that (see Section \ref{sec_can_bounds} for a similar computation) $\widehat{h}_{A,D}(P_1, P_2)=\widehat{h}_{E,(O)}(P_1)- \widehat{h}_{E,(O)}(P_2)$, for every $(P_1, P_2) \in A(\Qal)$. Hence, in this case the set of points of zero canonical height contains, for instance, the $\Qal$-points of the diagonal and all its translates by torsion points. 
\end{remark}

\subsection{Complex Multiplication}
In this subsection, we recall the basic definitions and key facts about complex multiplication for abelian varieties defined over fields of characteristic 0, which will be used throughout this article. For further details on this topic, we refer to \cite{Lan83b, Shi98, Mil20}.
\begin{definition} A \emph{CM field} $K$ is a totally imaginary quadratic extension of a totally real number field. That is, $K$ has the form $K = K_0(\sqrt{\alpha})$, where $K_0$ is a totally real field, i.e.\ a number field whose embeddings into $\mathbb{C}$ are all real, and $\alpha \in K_0$ satisfies the condition that each embedding of $K_0$ into $\mathbb{C}$ maps $\alpha$ to a negative real number.
\end{definition}  
\begin{definition}
An abelian variety $A$ of dimension $g$ is said to have \emph{Complex Multiplication} (CM) if its endomorphism algebra $\en^0(A)$ contains a commutative semisimple subalgebra of dimension $2g$ over $\Q$. We say that $A$ has CM by the CM field $K$ (of degree $2g$) if there exists an embedding $K \hookrightarrow \en^0(A)$.
\end{definition}
Note that a simple abelian variety $A$ has complex multiplication if and only if $\en^0(A)$ is a CM field of degree  $2\dim(A)$. In general, an abelian variety has complex multiplication if and only if each of
its simple factors up to isogeny has complex multiplication.

If $A$ is a simple CM abelian variety of dimension $g$, then $\en^0(A)\cong K$ is a CM field and there is a set $\Phi=\pg{\phi_1, \ldots, \phi_g}$ of complex embeddings of $K$ such that $\Phi \cup \overline{\Phi}$ is the set of all complex embeddings of $K$ and $T_O(A) \cong \prod_{i=1}^{g} \C_{\phi_i}$, where $\C_{\phi_i}$ is a 1-dimensional $\C$-vector space on which $\a \in K$ acts as $\phi_i(\a)$. We call the pair $(K, \Phi)$ a \emph{CM type} of $A$. Moreover, by Proposition 3.13 of \cite{Mil20}, $(K, \Phi)$ is primitive, i.e.\ it is not induced by a CM type of a proper CM subfield of $K$.

More generally, let $\mathcal{K}=\prod_{j} K_j$ be a finite product of CM fields and let $\iota$ be the involution of $\mathcal{K}$ acting as complex conjugation on each factor, so that $\overline{\psi}=\psi\circ\iota$ for every ring homomorphism $\psi:\mathcal{K}\to\C$. A \emph{CM type} of $\mathcal{K}$ is a set $\Phi$ of ring homomorphisms $\mathcal{K}\to\C$ such that every such homomorphism lies in exactly one of $\Phi$ and $\overline{\Phi}$. The \emph{reflex field} $E(\mathcal{K},\Phi)$ is the fixed field of $\pg{\s \in \Gal(\Qal/\Q): \s\circ\Phi=\Phi}$; when $\mathcal{K}$ is a field, this is the usual reflex field (see \cite{Mil20}). Since complex conjugation commutes with every automorphism of a CM field, we have $\overline{\s\circ\psi}=\s\circ\psi\circ\iota$, so that the set $\s\circ\Phi$ is again a CM type of $\mathcal{K}$, and as there are $2^{\dim_{\Q}\mathcal{K}/2}$ CM types, we get $\pq{E(\mathcal{K},\Phi):\Q}\leq 2^{\dim_{\Q}\mathcal{K}/2}$.

For a matrix $Z \in \mat_g(\C)$ we denote by $\Q(Z)$ the field generated by its entries, and we set $\kappa_g\coloneqq2^g$.

\begin{lemma}\label{lemma:degree_CM_periods}
Let $A=\C^g/\Lambda$ be an abelian variety with complex multiplication and suppose that $(Z,\mathbf{1}_g)$ is a period matrix of $A$ with respect to some bases of $\Lambda$ and $\C^g$ (e.g.\ $A$ principally polarized, the basis of $\Lambda$ symplectic and $Z \in \HH_g$). Then there is a CM type $(\mathcal{K},\Phi)$, with $\mathcal{K}\subseteq\en^0(A)$ and $\dim_{\Q}\mathcal{K}=2g$, such that $Z\in\mat_g(E(\mathcal{K},\Phi))$. In particular, $\pq{\Q(Z):\Q}\leq\kappa_g$.
\end{lemma}
\begin{proof}
Write $A$, up to isogeny, as $\prod_{i} A_i^{e_i}$ with $A_i$ simple. Each $A_i$ has CM, so $\en^0(A_i)$ is a CM field $K_i$ of degree $2\dim A_i$, and $\en^0(A)$ contains $\mathcal{K}=\prod_i K_i^{e_i}$ (embedded diagonally), with $\dim_\Q \mathcal{K}=2g$.

Since $\r_r$ is injective and $\mathcal{K}$ is a product of fields of total degree $2g$, $\Lambda\otimes\Q$ is a free $\mathcal{K}$-module of rank one. Hence there are $\l_0 \in \Lambda$ and a lattice $\mathfrak{a}\subseteq\mathcal{K}$ such that $\Lambda=\mathfrak{a}\l_0$. In particular, $\r_r|_{\mathcal{K}}\otimes\C$ is the direct sum of the $2g$ ring homomorphisms $\mathcal{K}\to\C$, each occurring once. As $\r_r\otimes\C\cong\r_a\oplus\overline{\r_a}$ \cite[Proposition 1.2.3]{BL04}, there is a CM type $\Phi$ of $\mathcal{K}$ and $T\in\GL_g(\C)$ such that $T\r_a(x)T^{-1}=\mathrm{diag}(\psi(x))_{\psi\in\Phi}$ for every $x \in \mathcal{K}$. No coordinate of $T\l_0$ vanishes, otherwise $T\Lambda=\pg{\mathrm{diag}(\psi(a))_{\psi\in\Phi}\,T\l_0 : a\in\mathfrak{a}}$ would lie in a coordinate hyperplane. So, after multiplying $T$ on the left by a diagonal matrix, we may assume that $T\l_0=(1,\ldots,1)^t$, so that $T(a\l_0)=(\psi(a))_{\psi\in\Phi}$ for every $a\in\mathfrak{a}$.

Let $a_1\l_0,\ldots,a_{2g}\l_0$ be the basis of $\Lambda$ giving the period matrix $\Pi=(Z,\mathbf{1}_g)$, and let $\Omega=(\psi(a_k))_{\psi\in\Phi,\,1\leq k\leq 2g}=(\Omega_1,\Omega_2)$, with $\Omega_1,\Omega_2\in\mat_g(\Qal)$. Then $T\Pi=\Omega$, so that $\Omega_2=T$ and $Z=\Omega_2^{-1}\Omega_1\in\mat_g(\Qal)$. If $\s\in\Gal(\Qal/\Q)$ satisfies $\s\circ\Phi=\Phi$, then $\s$ permutes the rows of $\Omega_1$ and of $\Omega_2$ in the same way, i.e.\ $\s(\Omega_j)=P\Omega_j$ for a permutation matrix $P$, whence $\s(Z)=(P\Omega_2)^{-1}P\Omega_1=Z$. Thus $Z\in\mat_g(E(\mathcal{K},\Phi))$, and the bound follows from $\pq{E(\mathcal{K},\Phi):\Q}\leq 2^g$.
\end{proof}

\begin{remark}
In fact $\Q(Z)=E(\mathcal{K},\Phi)$. Indeed, for every ring homomorphism $\psi:\mathcal{K}\to\C$, the vector $r_{\psi}=(\psi(a_1),\ldots,\psi(a_{2g}))$ satisfies $r_{\psi}\,\r_r(x)=\psi(x)\,r_{\psi}$ for all $x\in\mathcal{K}$, so the $2g$ vectors $r_{\psi}$ are linearly independent. Since $Z$ is determined by the row space of $(Z,\mathbf{1}_g)$, which is the span of $\pg{r_{\psi}:\psi\in\Phi}$, and $\s(r_{\psi})=r_{\s\circ\psi}$, we get $\s(Z)=Z$ if and only if $\s\circ\Phi=\Phi$. In particular, $\pq{\Q(Z):\Q}$ can be equal to $2^g$: for instance, if $A=E_1\times E_2\times E_3$ is the product of elliptic curves with CM by $\Q(i)$, $\Q(\sqrt{-2})$ and $\Q(\sqrt{-3})$, then $\Q(Z)=\Q(i,\sqrt{-2},\sqrt{-3})$ has degree $8=2^g$.
\end{remark}

\subsection{Euclidean lattices}

A \emph{Euclidean lattice} is a pair $(\Lambda,\nm{\cdot})$, where $\Lambda$ is a free $\Z$-module of finite rank and $\nm{\cdot}$ is the norm induced by an inner product on the real vector space
$$\Lambda_{\R}\coloneqq\Lambda\otimes_{\Z}\R.$$
The rank of $\Lambda$ is denoted by $\rk{\Lambda}$.

\begin{example}\label{example:euclidean_lattices}
We collect here the main examples of Euclidean lattices that will be used throughout the paper. Let $A$ be an abelian variety defined over a number field $K$ endowed with an ample symmetric line bundle $L$. 
\begin{itemize}
\item The endomorphism ring $\Lambda_1 \coloneqq \mathrm{End}_{\overline{K}}(A)$ becomes a Euclidean lattice when equipped with the Rosati norm
$$\nm{f}_{Ros} \coloneqq \sqrt{\mathrm{tr}(\rho_r(\RI{f}f))},$$
where the Rosati involution is the one defined by the line bundle $L$, since by \cite[Theorem 5.1.8]{BL04}, 
\begin{equation}\label{eqn:RI_pos}
\tr(\r_r(\RI{f}f))>0
\end{equation}
for any nonzero $f \in \en^0(A)= \en(A) \otimes \Q$.
\item The Mordell--Weil group modulo torsion $\Lambda_2 \coloneqq A(K)/A(K)_{\mathrm{tors}}$
is a Euclidean lattice when equipped with the Néron--Tate norm
$$\nm{P}_{NT} \coloneqq \sqrt{\widehat{h}_{A,L}(P)}$$
since $\widehat{h}_{A,L}$ extends $\R$-linearly to a positive definite quadratic form on $A(\overline{K}) \otimes \R$ \cite[Proposition B.5.3]{HS13}.
\end{itemize}
\end{example}

The \emph{volume} of $\Lambda$ is
$$\mathrm{Vol}(\Lambda)\coloneqq\sqrt{\det\!\left(\langle e_i,e_j\rangle\right)_{1\leq i,j\leq r}},$$
where $r=\rk\Lambda$ and $(e_1,\dots,e_{r})$ is any $\Z$-basis of $\Lambda$. This quantity is independent of the choice of basis.

The \emph{first minimum} of $\Lambda$ is defined by
$$\l_1(\Lambda)\coloneqq\min\pg{\nm{x}:x\in\Lambda\setminus\pg{0}}.$$

A sublattice $\Lambda'\subseteq\Lambda$ is called \emph{saturated} if $\Lambda/\Lambda'$ is torsion-free. In this case $\Lambda/\Lambda'$ is a free $\Z$-module, $(\Lambda/\Lambda')\otimes_{\Z}\R=\Lambda_{\R}/\Lambda'_{\R}$, and we endow it with the \emph{quotient norm}
$$\nm{x+\Lambda'_{\R}}\coloneqq\min_{y\in\Lambda'_{\R}}\nm{x+y},$$
i.e.\ the norm of the orthogonal projection of $x$ onto $(\Lambda'_{\R})^{\perp}$. Then $\Lambda/\Lambda'$ is a Euclidean lattice and
\begin{equation}\label{eqn:vol_quotient}
\mathrm{Vol}(\Lambda/\Lambda')=\frac{\mathrm{Vol}(\Lambda)}{\mathrm{Vol}(\Lambda')}
\end{equation}
(see \cite[Section 2.3.1 and Equation (2.8)]{Bos20}; equivalently, apply \cite[Proposition 5]{Gau20} to the rigid adelic space over $\Q$ associated with $\Lambda$ and to its subspace $\Lambda'\otimes\Q$).

Finally, if $T:\Lambda_1\longrightarrow\Lambda_2$ is a homomorphism of Euclidean lattices, its \emph{operator norm} is
$$\nm{T}_{\mathrm{op}}\coloneqq\sup_{x\in(\Lambda_1)_{\R} \setminus\pg{0}}\frac{\nm{T(x)}_2}{\nm{x}_1},$$
where we still denote by $T$ its $\R$-linear extension.

\begin{lemma}\label{lemma:operator_norms_equal}
Let $(\Lambda_1,\nm{\cdot}_1)$ and $(\Lambda_2,\nm{\cdot}_2)$ be Euclidean lattices, and let $f : \Lambda_1 \to \Lambda_2$ be a homomorphism. Set $K = \ker(f)$. Then $K$ is saturated; endow $\Lambda_3 \coloneqq \Lambda_1/K$ with the quotient norm $\nm{\cdot}_3$ and let
$$\overline{f} : \Lambda_3 \longrightarrow \mathrm{im}(f)\subseteq \Lambda_2$$
be the induced injective homomorphism. Then $\nm{\overline{f}}_{\mathrm{op}} = \nm{f}_{\mathrm{op}}$.
\end{lemma}
\begin{proof}
Since $\Lambda_3 \cong \mathrm{im}(f)\subseteq \Lambda_2$ is torsion-free, $K$ is saturated, so $\Lambda_3$ is a Euclidean lattice. Let $q:(\Lambda_1)_{\R}\to(\Lambda_3)_{\R}=(\Lambda_1)_{\R}/K_{\R}$ be the quotient map, so that $f=\overline{f}\circ q$ on $(\Lambda_1)_{\R}$. Since $\nm{q(x)}_3\leq\nm{x}_1$ for every $x\in(\Lambda_1)_{\R}$, we have $\nm{f(x)}_2\leq\nm{\overline{f}}_{\mathrm{op}}\nm{x}_1$. Taking the supremum over $x\in (\Lambda_1)_{\R} \setminus \pg{0}$ gives $\nm{\overline{f}}_{\mathrm{op}} \geq \nm{f}_{\mathrm{op}}$. 

Conversely, $q$ restricts to an isometry $K_{\R}^{\perp}\to(\Lambda_3)_{\R}$ and $f=\overline{f}\circ q$, so that $\nm{\overline{f}}_{\mathrm{op}}$ is the supremum of $\nm{f(x)}_2/\nm{x}_1$ over $x\in K_{\R}^{\perp}\setminus\pg{0}$, which is at most $\nm{f}_{\mathrm{op}}$.
\end{proof}
\section{Reduction to transverse curves in the universal family}\label{sec_red}
In this section, we first reduce Theorem \ref{main_thm} to Theorem \ref{thm:main_product} below, which concerns transverse curves, and then we reduce the latter to the case where $\cA =\mathfrak{A}_g \times _{\A_g} S$, with $S \subseteq \A_g$ a smooth, irreducible, locally closed curve defined over $\Qal$. The results of this section are inspired by Section 2 of \cite{BC20}.
 
For an irreducible curve $\cC\subseteq\cA$, an abelian variety $B$ defined over $\Qal$ and $b\in B(\Qal)$, we denote by $\mathcal{E}_b(\cC)$ the set of points $P\in\cC(\C)$ such that $\cA_{\pi(P)}\times B$ has CM and $(P,b)$ lies in a proper algebraic subgroup of $\cA_{\pi(P)}\times B$. For $B=0$ we write $\mathcal{E}(\cC)=\mathcal{E}_0(\cC)$; this is the set of points of $\cC(\C)$ lying in a proper algebraic subgroup of a CM fiber, so that Theorem \ref{main_thm} claims that $\mathcal{E}(\cC)$ is finite. For a morphism $T \rightarrow S$ we write $\cA_T=\cA \times_S T$. We will often use that, if $T$ is a smooth curve, an abelian scheme over $T$ is the Néron model of its generic fiber \cite[Proposition 1.2.8]{BLR90}; in particular, every homomorphism from the generic fiber of a smooth group scheme over $T$ to the generic fiber of an abelian scheme over $T$ extends uniquely to a homomorphism over $T$.
 
The first result of this section allows us to perform finite base changes.
\begin{lemma}\label{lemma_base_change}
Let $\ell: S' \rightarrow S$ be a finite surjective morphism of smooth irreducible curves, let $\cA'=\cA \times_{S} S'$ and let $\r: \cA' \rightarrow \cA$ be the projection. Let $\cC' \subseteq \cA'$ be an irreducible curve and let $\cC=\r(\cC')$. Then:
\begin{enumerate}[label=(\alph*)]
\item $\cA'$ is isotrivial if and only if $\cA$ is, $\cC'$ is contained in a fiber if and only if $\cC$ is, and $\cC'$ is contained in a translate of a proper flat subgroup scheme by a constant section after a finite base change if and only if $\cC$ is;
\item for every abelian variety $B$ over $\Qal$ and every $b\in B(\Qal)$, $\mathcal{E}_b(\cC)=\r(\mathcal{E}_b(\cC'))$ and $\mathcal{E}_b(\cC')=\cC' \cap \r^{-1}(\mathcal{E}_b(\cC))$; in particular, $\mathcal{E}_b(\cC)$ is finite if and only if $\mathcal{E}_b(\cC')$ is.
\end{enumerate}
\end{lemma}
\begin{proof}
By \cite[Proposition III.9.7]{Har77}, $\ell$ is flat, so that $\r$ and all its base changes are finite and flat. Hence, $\cC$ is an irreducible curve and, by \cite[Corollary III.9.6]{Har77}, the preimage of an irreducible curve under any of these maps has pure dimension $1$. Moreover, $\r$ restricts to an isomorphism of abelian varieties $\cA'_{s'} \rightarrow \cA_{\ell(s')}$ for every $s' \in S'(\C)$. This gives (b) and the claim about fibers, since $\r^{-1}(\cA_s)$ is a finite disjoint union of fibers.
 
If $\cA$ is isotrivial, i.e.\ $\cA_T \cong A_1 \times T$ for some finite surjective $T \rightarrow S$, and $T'$ is the normalization of an irreducible component of $T \times_S S'$, then $\cA'\times_{S'}T' \cong A_1 \times T'$, so that $\cA'$ is also isotrivial; the converse is clear. 

Similarly, if $T \rightarrow S'$ is finite and surjective, then $\cA'\times_{S'}T=\cA_T$ and $T \rightarrow S$ is finite and surjective as well. Moreover, an irreducible component of the preimage of $\cC'$ in $\cA'\times_{S'}T=\cA_T$ is also an irreducible component of the preimage of $\cC$. This shows that if $\cC'$ is contained in a translate of a proper flat subgroup scheme by a constant section after a finite base change, then $\cC$ is as well.
 
Conversely, assume that, for some finite surjective $T \rightarrow S$, an irreducible component $\D$ of the preimage of $\cC$ in $\cA_T$ is contained in $\s+\mathcal{H}$, with $\s$ constant and $\mathcal{H}$ a proper flat subgroup scheme. Let $T_1, \ldots, T_m$ be the normalizations of the irreducible components of $T \times_S S'$, which are finite and surjective over $T$ and $S'$, and let $q_j: \cA_{T_j} \rightarrow \cA_T$ and $r_j: \cA_{T_j} \rightarrow \cA'$ be the projections. For every $y \in \D(\Qal)$, the image of $y$ in $\cA$ lies in $\cC=\r(\cC')$, so there is $x \in \cC'(\Qal)$ with the same image in $\cA$. As $(y,x) \in \cA_T \times_{\cA} \cA'=\cA \times_S (T \times_S S')$, there are $j$ and $\widetilde{y}\in \cA_{T_j}(\Qal)$ such that $q_j(\widetilde{y})=y$ and $r_j(\widetilde{y})=x$. Hence, for some $j$, the closed set $q_j^{-1}(\D) \cap r_j^{-1}(\cC')$ is infinite, so it contains an irreducible component $\D'$ of $r_j^{-1}(\cC')$. Then $\D' \subseteq q_j^{-1}(\s+\mathcal{H})$, which is the translate of the proper flat subgroup scheme $\mathcal{H}\times_T T_j$ by the constant section induced by $\s$.
\end{proof}
 
\begin{remark}\label{rmk_open_subset}
We can always remove finitely many points from $S$. Indeed, since $\cC$ is not contained in a fiber, this removes only finitely many points of $\cC$. Moreover, the hypotheses of Theorems \ref{main_thm} and \ref{thm:main_product} only depend on the generic fiber of $\cA \rightarrow S$ (and of its base changes), since an abelian scheme over a smooth curve is the Néron model of its generic fiber and a flat subgroup scheme is the closure of its generic fiber.
\end{remark}
 
Next, let $\cA$ and $\cA'$ be abelian schemes over the same curve $S$ and let $f_{\eta}\colon \cA'_{\eta} \rightarrow \cA_{\eta}$ be an isogeny between the generic fibers defined over $\Qal(S)$. Then, $f_{\eta}$ extends to an isogeny $f\colon \cA' \rightarrow \cA$ between the abelian schemes (see the proof of Lemma 2.2 of \cite{BC20} for a proof of this and \cite[Definition 27.176]{GW23} for the definition of isogeny between abelian schemes). If $N=\deg f$, then $[N]$ vanishes on $\ker f_{\eta}$, so that there is an isogeny $f'_{\eta}\colon \cA_{\eta} \rightarrow \cA'_{\eta}$ with $f'_{\eta} \circ f_{\eta}=[N]$, which extends to an isogeny $f'\colon \cA \rightarrow \cA'$ with $f' \circ f=[N]_{\cA'}$.
 
\begin{lemma}\label{lemma_isogenies_schemes}
Let $\cA, \cA', f$ be as above, let $\cC\subseteq\cA$ be an irreducible curve not contained in a fixed fiber and let $\cC'$ be an irreducible component of $f^{-1}(\cC)$. Then $\mathcal{E}_b(\cC) \subseteq f(\mathcal{E}_b(\cC'))$ for every abelian variety $B$ over $\Qal$ and every $b \in B(\Qal)$. Moreover, if $\cA$ is not isotrivial and $\cC$ is transverse, then $\cA'$ is not isotrivial and $\cC'$ is transverse.
\end{lemma}
\begin{proof}
Since $f$ is finite and flat, we have $f(\cC')=\cC$; in particular, $\cC'$ is not contained in a fiber. For every $s \in S(\C)$, $f_s\times\mathrm{id}_B$ is an isogeny, so $\cA_s\times B$ is CM if and only if $\cA'_s\times B$ is, and $(f_s\times\mathrm{id}_B)^{-1}$ sends proper algebraic subgroups to proper algebraic subgroups. Hence, $\mathcal{E}_b(\cC) \subseteq f(\mathcal{E}_b(\cC'))$.
 
If $\cA'_T \cong A_1\times T$ for some finite surjective $T \rightarrow S$, then $\ker f_{\eta}$ would correspond to a subgroup $G \subseteq A_1(\Qal)$, as all torsion points of $A_1$ over $\overline{\Qal(T)}$ are defined over $\Qal$. Thus, the generic fiber of $\cA_T$ would be isomorphic to that of $(A_1/G)\times T$, and $\cA$ would be isotrivial.
 
Finally, assume that, for some finite surjective $T \rightarrow S$, an irreducible component $\D'$ of the preimage of $\cC'$ in $\cA'_T$ is contained in $\s'+\mathcal{H}'$, with $\s'=\psi(a,\cdot)$ constant and $\mathcal{H}'$ proper and flat. Let $f_T$ and $f'_T$ be the base changes of $f$ and $f'$. Then $f_T(\D')$ is an irreducible component of the preimage of $\cC$ in $\cA_T$ and, since $f'_T(f_T(h))=[N]h \in \mathcal{H}'$ for every $h \in \mathcal{H}'$,
$$f_T(\D') \subseteq f_T \circ \s' + (f'_T)^{-1}(\mathcal{H}').$$
Here $f_T\circ\s'=(f_T\circ\psi)(a,\cdot)$ is constant and, as $f'_T$ is a finite flat homomorphism, $(f'_T)^{-1}(\mathcal{H}')$ is a proper flat subgroup scheme of $\cA_T$. This contradicts the hypotheses on $\cC$.
\end{proof}

\begin{lemma}\label{lemma:product_decomposition}
Let $\cA\to S$ be non-isotrivial and let $\cC \subseteq \cA$ be an irreducible curve that is contained neither in a fixed fiber nor in a proper flat subgroup scheme of $\cA$, even after a finite base change. Then, there exist a smooth irreducible curve $S'$ with a finite surjective morphism $S' \to S$, a non-isotrivial abelian scheme $\mathcal{G}\to S'$, an abelian variety $B$ over $\Qal$, a point $b \in B(\Qal)$ not contained in any proper algebraic subgroup of $B$, an isogeny $\phi:\mathcal{G}\times_{S'}(B\times S')\to\cA'=\cA\times_S S'$ and a transverse irreducible curve $\cC_1\subseteq\mathcal{G}$, such that $\cC_1\times\pg{b}$ is an irreducible component of $\phi^{-1}(\cC')$ for some irreducible component $\cC'$ of the preimage of $\cC$ in $\cA'$. Moreover, $B=0$ if $\cC$ is transverse.
\end{lemma}
\begin{proof}
Let $K=\Qal(S)$, let $A=\cA_{\eta}\otimes_K\overline{K}$, let $P\in A(\overline{K})$ be a geometric generic point of $\cC$ and let $A_0$ be the largest abelian subvariety of $A$ which can be defined over $\Qal$ (i.e.\ the image of the $\overline{K}/\Qal$-trace of $A$, as in Section \ref{sec_moduli}). As in the proof of Lemma \ref{lemma:func_transc_CM}, the hypotheses on $\cC$ mean that $P$ lies in no proper algebraic subgroup of $A$. Moreover, a curve is transverse if and only if it is not contained in a fixed fiber and its geometric generic point does not lie in $H+A_0(\Qal)+A_{\mathrm{tors}}$ for any abelian subvariety $H\subsetneq A$.

Choose an abelian subvariety $G\subseteq A$ of minimal dimension such that $P\in G+a+t$, with $a\in A_0(\Qal)$ and $t\in A_{\mathrm{tors}}$ (so $G=A$ if $\cC$ is transverse). The image of $P$ in $A/(G+A_0)$ is torsion, hence $G+A_0=A$, since $P$ is not contained in any proper algebraic subgroup of $A$. In particular, $G$ is not contained in $A_0$, since $\cA$ is not isotrivial. The constant part of $G$ is $G_0=(G\cap A_0)^0$.\footnote{Here and in the sequel, given an abelian variety $A$ and an algebraic subgroup $H$, we denote by $H^0$ the irreducible component of $H$ containing the identity.} As abelian subvarieties of $A_0$ are defined over $\Qal$, Poincaré's reducibility theorem gives an abelian subvariety $B\subseteq A_0$, defined over $\Qal$, with $G_0+B=A_0$ and $G_0\cap B$ finite. Then $G+B=A$ and $(G\cap B)^0\subseteq G_0\cap B$, so the sum map $\phi_{\eta}:G\times B\to A$ is an isogeny. Writing $a$ as the image of a point of $G_0(\Qal)\times B(\Qal)$ and $t$ as the image of a torsion point of $G\times B$, we get $P=\phi_{\eta}(P_1,b)$ with $P_1\in G(\overline{K})$ and $b\in B(\Qal)$.

If $b$ lies in a proper algebraic subgroup $H_B$ of $B$, then $P$ would lie in the proper algebraic subgroup $\phi_{\eta}(G\times H_B)$. If $P_1\in H'+a'+t'$, with $H'\subsetneq G$, $a'\in G_0(\Qal)$ and $t'$ torsion, then $P\in H'+(a'+b)+t'$, with $a'+b\in A_0(\Qal)$, contradicting the minimality of $G$.

Finally, up to a finite base change $S'\to S$, we may assume that $G$, $P_1$ and $\phi_{\eta}$ are defined over $\Qal(S')$. Since $G$ is an abelian subvariety of the generic fiber of $\cA'$, which has good reduction at every closed point of $S'$, Lemma 7.4.2 and Theorem 1.4.3 of \cite{BLR90} show that $G$ extends to an abelian scheme $\mathcal{G}\to S'$. Moreover, $\phi_{\eta}$ extends to an isogeny $\phi$, as recalled before Lemma \ref{lemma_isogenies_schemes}. Let $\cC_1\subseteq\mathcal{G}$ be the curve with generic point $P_1$ and let $\cC'$ be the irreducible component of the preimage of $\cC$ in $\cA'$ with generic point $P=\phi_{\eta}(P_1,b)$. Then $\cC_1\times\pg{b}$ is an irreducible curve contained in $\phi^{-1}(\cC')$, hence one of its components, since $\phi^{-1}(\cC')$ has pure dimension $1$. Moreover, $\mathcal{G}$ is not isotrivial, as $G \not\subseteq A_0$, and $\cC_1$ is transverse by the previous paragraph.
\end{proof}

\begin{theorem}\label{thm:main_product}
Let $\cA\to S$ be a non-isotrivial abelian scheme and let $\cC\subseteq\cA$ be a transverse irreducible curve, both defined over $\Qal$. Let $B$ be an abelian variety defined over $\Qal$ and let $b\in B(\Qal)$ be a point not contained in any proper algebraic subgroup of $B$. Then $\mathcal{E}_b(\cC)$ is finite.
\end{theorem}

In particular, for $B=0$, Theorem \ref{thm:main_product} says that $\mathcal{E}(\cC)$ is finite for every transverse curve $\cC$.

\begin{proposition}\label{prop:reduction_transverse_curve}
Theorem \ref{thm:main_product} implies Theorem \ref{main_thm}. 
\end{proposition}
\begin{proof}
Let $S'$, $\mathcal{G}$, $B$, $b$, $\phi$, $\cC'$ and $\cC_1$ be given by Lemma \ref{lemma:product_decomposition}. By Lemma \ref{lemma_base_change}(b), it suffices to prove that $\mathcal{E}(\cC')$ is finite. By the first part of Lemma \ref{lemma_isogenies_schemes}, applied to $\phi$ with the trivial abelian variety in place of $B$, we have $\mathcal{E}(\cC')\subseteq\phi\pt{\mathcal{E}(\cC_1\times\pg{b})}$, where $\mathcal{E}(\cC_1\times\pg{b})=\mathcal{E}_b(\cC_1)\times\pg{b}$ is computed in $\mathcal{G}\times_{S'}(B\times S')$. The latter set is finite by Theorem \ref{thm:main_product}.
\end{proof}

It remains to prove Theorem \ref{thm:main_product}. 

In each of the reductions below, Lemmas \ref{lemma_base_change} and \ref{lemma_isogenies_schemes} show that the new abelian scheme and the new curve still satisfy the hypotheses of Theorem \ref{thm:main_product}, and that the finiteness of $\mathcal{E}_b$ for the new curve implies the finiteness of $\mathcal{E}_b(\cC)$.
 
Now, every abelian variety over an algebraically closed field is isogenous to a principally polarized one \cite[Section 23, Corollary 1]{Mum08}. Hence, there exist a smooth irreducible curve $S'$ with a finite surjective morphism $S' \rightarrow S$, an abelian variety $A'$ over $\Qal(S')$ with an ample line bundle $M$ such that $\Phi_M$ is an isomorphism, and an isogeny $A' \rightarrow \cA_{\eta} \otimes \Qal(S')$. By Lemma \ref{lemma_base_change} we can assume that $S'=S$. Using Proposition 7.3.6, Theorem 7.4.5 and Theorem 1.4.3 of \cite{BLR90}, one can prove that $A'$ extends to an abelian scheme $\cA' \rightarrow S$ and, by Lemma \ref{lemma_isogenies_schemes}, we can replace $\cA$ by $\cA'$.
 
As $\cA$ is regular, $M$ extends to a line bundle $\mathcal{M}$ on $\cA$, and $\l=\Phi_{\mathcal{M}}: \cA \rightarrow \widehat{\cA}$ is a homomorphism of abelian schemes whose generic fiber is the isomorphism $\Phi_M$. Since $\cA$ and $\widehat{\cA}$ are the Néron models of their generic fibers, $\Phi_M^{-1}$ extends to an inverse of $\l$, which is therefore an isomorphism. Moreover, $\mathcal{M}_s$ is ample for every $s$ in a non-empty open subset of $S$ \cite[Corollaire 9.6.4]{EGAIV3}, and for such $s$ the map $\l_s=\Phi_{\mathcal{M}_s}$ is a principal polarization. By Remark \ref{rmk_open_subset}, we can then assume that $\l$ is a principal polarization of $\cA \rightarrow S$. Similarly, we may assume that $B$ is principally polarized: if $\psi_B: B'\to B$ is an isogeny, with $B'$ principally polarized, and $b'\in B'(\Qal)$ satisfies $\psi_B(b')=b$, then $b'$ is not contained in any proper algebraic subgroup of $B'$, and $\mathcal{E}_b(\cC)\subseteq\mathcal{E}_{b'}(\cC)$, since $(\mathrm{id}\times\psi_B)^{-1}$ sends proper algebraic subgroups to proper algebraic subgroups.
 
Then, by \cite[Lemma 3.2]{Ge26} (and its proof), there is a finite étale cover $\ell: S' \rightarrow S$, with $S'$ a smooth irreducible curve, such that $\cA'[3]$ is a constant group scheme over $S'$, where $\cA' \coloneqq \cA \times_S S'$. As $S'$ is connected, the Weil pairing on $\cA'[3]$ induced by $\l$ is constant too, so $\cA'$ has a symplectic level-$3$-structure. Let $\cC'$ be an irreducible component of the preimage of $\cC$ in $\cA'$. By Lemma \ref{lemma_base_change}, $\cA'$ and $\cC'$ satisfy the hypotheses of Theorem \ref{thm:main_product}, and it suffices to prove that $\mathcal{E}_b(\cC')$ is finite.
 
Hence, since $\A_g=\A_{g,\mathbf{1},3}$ is a fine moduli space, there is a unique morphism $\f: S' \rightarrow \A_g$ such that $\cA'$ is the pull-back of the universal family $\mathfrak{A}_g \rightarrow \A_g$ along $\f$. Since $\cA'$ is not isotrivial, $\f$ is not constant. Thus, up to removing finitely many points from $S'$ (Remark \ref{rmk_open_subset}), we can assume that $S''=\f(S')$ is a smooth locally closed curve in $\A_g$ and that $\f: S' \rightarrow S''$ is finite and surjective. We have the following cartesian diagram, where $\cA''=\mathfrak{A}_g \times_{\A_g} S''$:
\begin{center}
\begin{tikzcd}
\cA' \arrow[d] \arrow[r] & \mathfrak{A}_g \arrow[d] & \cA'' \arrow[d] \arrow[l] \\
S' \arrow[r, "\f"]          & \A_g                     & S'' \arrow[l, hook']                      
\end{tikzcd} 
\end{center}
Note that
$$\cA' \cong \mathfrak{A}_g \times_{\A_g} S' \cong (\mathfrak{A}_g \times_{\A_g} S'') \times_{S''} S'=\cA'' \times_{S''} S'$$
which gives a morphism $p: \cA' \rightarrow \cA''$. By Lemma \ref{lemma_base_change}, applied to $\f: S' \rightarrow S''$, the curve $\cC''=p(\cC')$ satisfies the hypotheses of Theorem \ref{thm:main_product} and, if the claim of Theorem \ref{thm:main_product} holds for $\cC''$, then it holds for $\cC'$.
 
Thus, for the remainder of the article, we assume that $S \subseteq \A_g$ is a smooth, irreducible, locally closed curve defined over $\Qal$, that $\cA =\mathfrak{A}_g \times _{\A_g} S$, that $\cC$ is transverse, and that $B$ is a principally polarized abelian variety of dimension $h\geq 0$, with principal polarization $\l_B$. We prove Theorem \ref{thm:main_product} in this setting.

\section{O-minimality and definable sets}
In this section, we recall some properties and results about o-minimal structures. For more details see \cite{vdDri98} and \cite{vdDM96}.
\begin{definition}
A \emph{structure} is a sequence $\mathcal{S}=\pt{\mathcal{S}_N}$, $N\geq 1$, where each $\mathcal{S}_N$ is a collection of subsets of $\R^N$ such that, for each $N,M\geq 1$:
\begin{itemize}
\item $\mathcal{S}_N$ is a boolean algebra (under the usual set-theoretic operations);
\item $\mathcal{S}_N$ contains every semialgebraic subset of $\R^N$;
\item if $A \in \mathcal{S}_N$ and $B \in \mathcal{S}_M$, then $A\times B \in \mathcal{S}_{N+M}$;
\item if $A\in \mathcal{S}_{N+M}$, then $\pi(A) \in \mathcal{S}_N$, where $\pi: \R^{N+M} \rightarrow \R^N$ is the projection onto the first $N$ coordinates.
\end{itemize}
If $\mathcal{S}$ is a structure and, in addition, 
\begin{itemize}
\item $\mathcal{S}_1$ consists exactly of the finite unions of open intervals and points
\end{itemize}
then $\mathcal{S}$ is called an \emph{o-minimal structure}.
\end{definition}
Given a structure $\mathcal{S}$, we say that $X \subseteq \R^N$ is a \emph{definable set} if $X \in \mathcal{S}_N$. 

Given $X\subseteq \R^N$ and a function $f\colon X\rightarrow \R^M$, we say that $f$ is a \emph{definable function} if its graph $\pg{(x,y)\in \R^N\times \R^M: x\in X, y=f(x)}$ is a definable set. One can easily prove that images and preimages of definable sets via definable functions are still definable.

Let $U\subseteq \R^{M}\times\R^{N}$. For $t_0\in \R^M$, we set $U_{t_0}=\pg{x\in \R^N:(t_0,x)\in U}$ and call $U$ a \emph{family} of subsets of $\R^N$, while $U_{t_0}$ is called the \emph{fiber} of $U$ above $t_0$. If $U$ is a definable set, then we call it a \emph{definable family} and it is easy to prove that the fibers $U_{t_0}$ are also definable.

\begin{proposition}[\cite{vdDM96}, 4.4]\label{prop_fibre_finite} 
Let $U$ be a definable family in a fixed o-minimal structure $\mathcal{S}$. Then, there exists an integer $n$ such that each fiber of $U$ has at most $n$ connected components.
\end{proposition} 

While there are many examples of o-minimal structures (see \cite{vdDM96}), in this article we will work with the structure $\R_{\mathrm{an,exp}}$, which was proved to be o-minimal by van den Dries and Miller \cite{vdDM94}. From now on, we use the term \quotes{definable} to mean definable in $\R_{\mathrm{an,exp}}$. A complex set or function is said to be definable if it is definable as a real object, considering its real and imaginary parts separately. 

For a set $Z \subseteq \R^M \times \R^N= \R^{M+N}$, a positive integer $d$ and a positive real number $T$ define
\begin{equation*}
Z^{\sim}(d,T)\coloneqq\pg{(y,z)\in Z: H_d(y)\leq T}
\end{equation*}
where $H_d(y)$ is the $d$-height given by Definition \ref{def_d_height}. Let also $ \pi_1, \pi_2 $ be the projections of $ Z $ to the first $ M $ and last $ N $ coordinates, respectively.

We will use the following counting result of Habegger and Pila.
\begin{proposition}[{\cite[Corollary 7.2]{HP16}}]\label{prop:thm_HP_CM}
Let $Z \subseteq \R^{M}\times\R^{N}$ be a definable set. For every positive integer $d$ and every $\eps>0$ there exists a constant $c=c(Z,d,\eps)>0$ with the following property. If $T\geq 1$ and $\#\pi_2\pt{Z^{\sim}(d,T)}>cT^{\eps}$, then there exists a continuous definable function $\beta:[0,1]\to Z$ such that:
\begin{enumerate}
\item $\pi_1\circ\beta:[0,1]\to\R^{M}$ is semialgebraic and its restriction to $(0,1)$       is real analytic;
\item $\pi_2\circ\beta:[0,1]\to\R^{N}$ is non-constant;
\item $\beta|_{(0,1)}$ is real analytic.
\end{enumerate}
\end{proposition}
\begin{remark}
This is the special case of \cite[Corollary 7.2]{HP16} for the constant family $F=\R\times Z$ and $\Sigma=Z^{\sim}(d,T)$; we omit conclusion (iii) there, which we do not need. Part (3) is conclusion (iv) of \cite[Corollary 7.2]{HP16}, which applies because $\R_{\mathrm{an,exp}}$ admits analytic cell decomposition \cite[Section 8]{vdDM94}.
\end{remark}

We conclude this section by showing that the set $\mathcal{Z}$ defined in (\ref{def:Z_CM}) is definable in $\R_{\mathrm{an,exp}}$. 

Fix an embedding of $\mathfrak{A}_g$ as a locally closed subvariety of some projective space, which we identify with a semialgebraic set via the standard affine charts. By Theorem 1.2 of \cite{PS13}, there is an open subset $U$ of $\HH_g \times \C^g$ containing $\mathcal{F}_g$ such that the restriction of the uniformizing map $u$ to $U$ is definable. Since $\mathcal{F}_g$ is a semialgebraic subset of $\HH_g \times \C^g$, it follows that $u$ is definable when restricted to $\mathcal{F}_g$. Consequently, as $\cC(\C)$ is a semialgebraic subset of this projective space, we conclude that $\mathcal{Z}$ is definable.

\section{Matrix bounds for endomorphisms of abelian varieties}\label{sect:matrix_bounds_endos}
Let $A$ be an abelian variety of dimension $g$ defined over $\C$, so that $A \cong \C^g/\Lambda$ for some lattice $\Lambda$. Fix a polarization $\mathcal{L}$ of type $\mathbf{D}=\mathrm{diag}(d_1, \ldots, d_g)$ and let $d=d_1 \cdot \ldots \cdot d_g=\mathrm{Pf}(E_{\mathcal{L}})$, so that $\mathcal{L}$ has degree $d^2$. Fix also a symplectic basis $\l_1, \ldots, \l_{2g}$ of $\Lambda$ and a basis $e_1, \ldots, e_g$ of $\C^g$ such that the period matrix of $A$ with respect to these bases is $\pt{\tau, \mathbf{D}}$, where $\tau \in \HH_g$ (see \cite[Section 8.1]{BL04}). 

As in Section \ref{sec_moduli}, denote by $\mathfrak{F}_g$ the fundamental domain for the action of $\mathrm{Sp}_{2g}(\Z)$ on $\HH_g$, as described in \cite[Section V.4]{Igu72}. For a finite index subgroup $\Gamma$ of $\mathrm{Sp}_{2g}(\Z)$, recall that the Siegel fundamental domain for $\Gamma$ was defined in (\ref{eqn:def_F_Gamma}) as $\mathfrak{F}_{\Gamma}=\bigcup\limits_{i=1}^m \s_i \cdot \mathfrak{F}_g$, where $\s_1=\mathbf{1}_{2g}, \s_2, \ldots, \s_m \in \mathrm{Sp}_{2g}(\Z)$ is a complete set of representatives for the right cosets of $\Gamma$ in $\mathrm{Sp}_{2g}(\Z)$.

Let $\Delta_{\mathbf{D}}\coloneqq\mathrm{diag}(\mathbf{1}_g,\mathbf{D})$ and $J_{\mathbf{D}}\coloneqq\begin{psmallmatrix} 0 & \mathbf{D}\\ -\mathbf{D} & 0\end{psmallmatrix}$, so that $J=J_{\mathbf{1}_g}$ as in Section \ref{sec_moduli}, and set
$$G_{\mathbf{D}}\coloneqq\pg{\Delta_{\mathbf{D}}^{-1}R^t\Delta_{\mathbf{D}} : R \in \GL_{2g}(\Z),\ R^tJ_{\mathbf{D}}R=J_{\mathbf{D}}}.$$
We will see in Lemma \ref{lemma:basis_in_F_Gamma} that $G_{\mathbf{D}}$ is a subgroup of $\mathrm{Sp}_{2g}(\Q)\coloneqq\pg{M \in \mat_{2g}(\Q): M^t J M=J}$, which acts on $\HH_g$ by the same formula as $\mathrm{Sp}_{2g}(\Z)$, and that a change of symplectic basis modifies $\tau$ by the action of an element of $G_{\mathbf{D}}$. Note that $G_{\mathbf{1}_g}=\mathrm{Sp}_{2g}(\Z)$, while for general $\mathbf{D}$ the group $G_{\mathbf{D}}$ is not contained in $\mathrm{Sp}_{2g}(\Z)$. Hence, we cannot always assume that $\tau \in \mathfrak{F}_{\Gamma}$ for an arbitrary $\Gamma$. The following lemma shows that this is possible for a suitable choice of $\Gamma$.

\begin{lemma}\label{lemma:basis_in_F_Gamma}
Let $G_{\mathbf{D}}$ be as above. Then:
\begin{enumerate}[label=(\alph*)]
\item $G_{\mathbf{D}}$ is a subgroup of $\mathrm{Sp}_{2g}(\Q)$ and, if $(\tau,\mathbf{D})$ is the period matrix of $A$ with respect to a symplectic basis of $\Lambda$ and a basis of $\C^g$, then for every $M \in G_{\mathbf{D}}$ the matrix $(M\cdot\tau,\mathbf{D})$ is the period matrix with respect to another such pair of bases;
\item $G_{\mathbf{D}}$ contains $\Gamma(d_g)=\ker\pt{\mathrm{Sp}_{2g}(\Z)\to\mathrm{Sp}_{2g}(\Z/d_g\Z)}$;
\item if $\Gamma \subseteq G_{\mathbf{D}}$ is a finite index subgroup of $\mathrm{Sp}_{2g}(\Z)$, the bases can be chosen so that $\tau \in \mathfrak{F}_{\Gamma}$.
\end{enumerate}
\end{lemma}
\begin{proof}
The set $G_{\mathbf{D}}$ is a group, being the image of the group $\pg{R \in \GL_{2g}(\Z): R^tJ_{\mathbf{D}}R=J_{\mathbf{D}}}$ under the anti-homomorphism $R \mapsto \Delta_{\mathbf{D}}^{-1}R^t\Delta_{\mathbf{D}}$. Since $\Delta_{\mathbf{D}}$ is symmetric, $J_{\mathbf{D}}=\Delta_{\mathbf{D}}J\Delta_{\mathbf{D}}$. Let $M=\Delta_{\mathbf{D}}^{-1}R^t\Delta_{\mathbf{D}}\in G_{\mathbf{D}}$ and $N=M^t=\Delta_{\mathbf{D}}R\Delta_{\mathbf{D}}^{-1}$. Then $R^tJ_{\mathbf{D}}R=J_{\mathbf{D}}$ gives $N^tJN=J$, so that $N$, and hence $M$, lies in $\mathrm{Sp}_{2g}(\Q)$. If $M=\begin{psmallmatrix} A & B\\ C & D\end{psmallmatrix}$, the period matrix with respect to the symplectic basis obtained by applying $R$ is
$$(\tau,\mathbf{D})R=(\tau,\mathbf{1}_g)N\Delta_{\mathbf{D}}=\pt{(A\tau+B)^t,(C\tau+D)^t}\Delta_{\mathbf{D}}.$$
Since $M \in \mathrm{Sp}_{2g}(\R)$, the matrix $C\tau+D$ is invertible and $M\cdot \tau \in \HH_g$ is symmetric, so multiplying on the left by $\pt{(C\tau+D)^t}^{-1}$ gives $(M\cdot\tau,\mathbf{D})$, proving (a).

If $M=\begin{psmallmatrix} M_1 & M_2\\ M_3 & M_4\end{psmallmatrix} \in \Gamma(d_g)$, set $R^t=\Delta_{\mathbf{D}}M\Delta_{\mathbf{D}}^{-1}=\begin{psmallmatrix} M_1 & M_2\mathbf{D}^{-1}\\ \mathbf{D}M_3 & \mathbf{D}M_4\mathbf{D}^{-1}\end{psmallmatrix}$. Since $M\equiv \mathbf{1}_{2g} \pmod{d_g}$ and $d_j \mid d_g$ for every $j$, $R$ has integer entries; moreover $\det R=1$ and $R^tJ_{\mathbf{D}}R=\Delta_{\mathbf{D}}MJM^t\Delta_{\mathbf{D}}=J_{\mathbf{D}}$. Hence $M \in G_{\mathbf{D}}$, proving (b).

Finally, as $\mathrm{Sp}_{2g}(\Z)=\bigcup_i \Gamma\s_i$, we have $\HH_g=\mathrm{Sp}_{2g}(\Z)\cdot\mathfrak{F}_g=\Gamma\cdot\mathfrak{F}_{\Gamma}$. Thus, there is $\g \in \Gamma$ with $\g\cdot\tau\in\mathfrak{F}_{\Gamma}$, and (c) follows from (a), since $\Gamma \subseteq G_{\mathbf{D}}$.
\end{proof}

In this section, $\Gamma$ is an arbitrary finite index subgroup of $\mathrm{Sp}_{2g}(\Z)$ (not necessarily $\Gamma(3)$ as in Section \ref{sec_moduli}), with the representatives $\s_1, \ldots, \s_m$ fixed as above. When we assume that $\tau \in \mathfrak{F}_{\Gamma}$ (Corollary \ref{bound_rat_rep}), we also assume that $\Gamma \subseteq G_{\mathbf{D}}$ (for instance $\Gamma=\Gamma(d_g)$, or any $\Gamma$ if $\mathbf{D}=\mathbf{1}_g$), so that this is possible by Lemma \ref{lemma:basis_in_F_Gamma}(c).

In order to state and prove the result of this section, we introduce some matrix norms.
\begin{definition}
For a matrix $M=(m_{i,j})_{1\leq i,j\leq n} \in \mat_n(\C)$ we define the following norms:
\begin{itemize}
\item $\nmi{M}\coloneqq\max\limits_{i,j} \abs{m_{i,j}}$;
\item (Frobenius norm) $\nm{M}_F\coloneqq\sqrt{\tr\pt{\overline{M}^t M}}=\sqrt{\sum\limits_{i,j=1}^{n} \abs{m_{i,j}}^2}$;
\item (Spectral norm) $\nm{M}_2\coloneqq\sqrt{\m}$, where $\m$ is the largest eigenvalue of the positive semidefinite Hermitian matrix $\overline{M}^t M$.
\end{itemize}
\end{definition}

Recall that the polarization $\mathcal{L}$ defines a Rosati involution $\RI{}$ (see Equation (\ref{eqn:def_rosati_inv})). Throughout this section, once a symplectic basis of $\Lambda$ is fixed, we identify rational representations of endomorphisms of $A$ with their matrices in that basis.

Then, $\nmi{\r_r(f)}$ and $\nm{f}_{Ros}= \sqrt{\mathrm{tr}(\rho_r(\RI{f}f))}$ (see Example \ref{example:euclidean_lattices}) are two norms on $\en^0(A)$. Since both extend to norms on the finite-dimensional real vector space $\en^0(A)\otimes_{\Q}\R$, they are equivalent. Thus, there exist two constants $c_1, c_2>0$ such that 
$$c_1 \cdot \nm{f}_{Ros} \leq \nmi{\r_r(f)} \leq c_2 \cdot \nm{f}_{Ros}$$ 
for every $f \in \en^0(A)$. 

The aim of this section is to make the constants $c_1, c_2$ effective. We first prove a bound valid for every choice of the symplectic basis, and then deduce a bound that does not depend on this choice, by taking $\tau\in\mathfrak{F}_{\Gamma}$.

\begin{proposition}\label{prop:bound_rat_rep_general}
Let $A$ be an abelian variety of dimension $g$ defined over $\C$ and fix a polarization $\mathcal{L}$ of type $\mathbf{D}$. Let $(\tau,\mathbf{D})$, with $\tau=X+\ii Y\in\HH_g$, be the period matrix of $A$ with respect to a symplectic basis of $\Lambda$ and a basis of $\C^g$, and set $N_{\tau}=\max\!\pg{1,\nmi{X},\nmi{Y},\nmi{Y^{-1}}}$. Then, for every $f \in \en^0(A)$, we have
$$c_1(A) \cdot \nm{f}_{Ros}\leq \nmi{\r_r(f)} \leq c_2(A) \cdot \nm{f}_{Ros}$$
where $c_1(A)=\dfrac{1}{2g} \sqrt{\dfrac{d_1}{d_g}}$, $c_2(A)=2^{3g+4} \cdot g^{4g+5} \cdot \dfrac{\nmi{\mathbf{D}}^{2g+2}}{d} \cdot N_{\tau}^{3g+3}$ and the Rosati involution is the one defined by $\mathcal{L}$.
\end{proposition}

Note that the constant $c_2(A)$ depends on the choice of the symplectic basis, and $N_{\tau}$ is unbounded on the $G_{\mathbf{D}}$-orbit of $\tau$. Choosing $\tau \in \mathfrak{F}_{\Gamma}$, we obtain the following bound, which only depends on $A$, $\mathcal{L}$ and $\Gamma$. We will not use it in the proof of Theorem \ref{main_thm}, where only the upper bound of Proposition \ref{prop:bound_rat_rep_general} is needed (see Lemma \ref{lemma_2g-ht_anal_rep}).

\begin{corollary}\label{bound_rat_rep}
Under the assumptions of Proposition \ref{prop:bound_rat_rep_general}, assume moreover that $\Gamma \subseteq G_{\mathbf{D}}$ and $\tau \in \mathfrak{F}_{\Gamma}$. Then, for every $f \in \en^0(A)$, we have
$$\nmi{\r_r(f)} \leq c_3(A) \cdot \nm{f}_{Ros}$$
where $c_3(A)=\d(g, \mathfrak{F}_{\Gamma}) \cdot \dfrac{\nmi{\mathbf{D}}^{2g+2}}{d} \cdot \max\!\pg{1, \nmi{\Im(Z_{\tau})}}^{2g^3+3g^2+2g+1}$. Here, $\d(g, \mathfrak{F}_{\Gamma})$ is an effective positive constant depending only on $g$ and the choice of the representatives of the right cosets of $\Gamma$ in $\mathrm{Sp}_{2g}(\Z)$ and $Z_{\tau} \in \mathfrak{F}_g$ is such that $\tau=\s_i\cdot Z_{\tau}$ for some $i \in \pg{1, \ldots, m}$.
\end{corollary}

The rest of this section is devoted to the proofs of Proposition \ref{prop:bound_rat_rep_general} and Corollary \ref{bound_rat_rep}.

Let $H$ be the Hermitian form associated with the polarization $\mathcal{L}$, and let $E=\Im(H)$ be the associated alternating form, which satisfies $E(\Lambda \times \Lambda) \subseteq \Z$. According to \cite[Lemma 2.1.7]{BL04}, the form $H$ can be expressed as: 
\begin{equation*}
H(v, w)=E(\ii v, w)+ \ii E(v,w)
\end{equation*}
for every $v, w \in \C^g$, with $S(v, w)=E(\ii v, w)=\Re(H(v,w))$ positive definite. 

Let $\RI{}$ be the Rosati involution defined by the polarization $\mathcal{L}$. By Proposition 5.1.1 of \cite{BL04}, we have
$$H(\r_a(f)(v),w) = H(v,\r_a(\RI{f})(w))$$ 
for any $f \in \en^0(A)$ and for all $v, w \in \C^g$. As in \cite{MW94}, evaluating this expression at $\l_1, \ldots, \l_{2g}$ and taking real and imaginary parts yields
$$\r_r(\RI{f})=S^{-1} \cdot \r_r(f)^t \cdot S= E^{-1} \cdot \r_r(f)^t \cdot E$$
where, with a slight abuse of notation, we denote by $S$ and $E$ the matrices representing the bilinear forms $S(v, w)$ and $E(v, w)$ with respect to the basis $\l_1, \ldots, \l_{2g}$ of $\Lambda$.
If we denote 
$R=\r_r(f),$
then
\begin{equation}\label{eqn:tr_RI_E/S}
\tr\pt{\r_r(\RI{f}f)}=\tr\pt{E^{-1} R^t\, E R}=\tr\pt{S^{-1} R^t\, S R}.
\end{equation}
Since $S$ is positive definite, there is a unique upper triangular matrix $U \in \mat_{2g}(\R)$ with positive diagonal entries such that $S=U^t U$. Substituting this decomposition, we have:
\begin{align*}
S^{-1} R^t \, S R&=U^{-1} (U^{-1})^t \, R^t \, U^t U R\\
&=U^{-1}\cdot \pt{(U^{-1})^t \, R^t \, U^t} \cdot \pt{U R \, U^{-1}} \cdot U
\end{align*}
and the invariance of the trace under conjugation implies:
$$\nm{f}_{Ros}=\sqrt{\tr\pt{\r_r(\RI{f}f)}}=\sqrt{\tr\pt{Q^{t} \, Q}}=\nm{Q}_F$$
where $Q=U R \, U^{-1}$. 

Furthermore, by the triangle inequality, for any $M_1, M_2 \in \mat_n(\C)$ we have 
\begin{equation}\label{eqn:submult_norm_infty}
\nmi{M_1 M_2}\leq n \nmi{M_1} \cdot \nmi{M_2}.
\end{equation}
Therefore, since $R=U^{-1} Q  U$, we get
$$\nmi{R}=\nmi{U^{-1} Q  U} \leq \pt{(2g)^2 \cdot \nmi{U^{-1}} \cdot \nmi{U}} \cdot \nmi{Q}.$$
We now prove a few general results about matrices.
\begin{lemma}\label{lemma_Chol}
If $M \in \mat_n(\R)$ is positive definite and $T \in \mat_n(\R)$ is an upper triangular matrix with positive diagonal entries such that $M=T^t \cdot T$, then $\nmi{T} \leq \sqrt{n \nmi{M}}$. 
\end{lemma}
\begin{proof}
We clearly have $\nm{M}_2=\nm{T}_2^2$. Moreover, $\nmi{N}\leq \nm{N}_2 \leq n \nmi{N}$ for every $N \in \mat_n(\R)$ \cite[Eq. (2.3.8)]{GVL13}. 
Thus, $\nmi{T} \leq \nm{T}_2=\sqrt{\nm{M}_2} \leq \sqrt{n \nmi{M}}$.
\end{proof}
The following result is well known but we include it for completeness.
\begin{lemma}\label{lemma:Had_ineq}
For any matrix $M \in \mat_n(\C)$, we have $\abs{\det(M)}\leq n^{n/2} \cdot \nmi{M}^{n}$.
\end{lemma}
\begin{proof}
This follows easily from Hadamard's inequality \cite{Had93}.
\end{proof}
\begin{lemma}\label{lemma_inv}
Let $M \in \mat_n(\C)$ be an invertible matrix. Then 
$$\nmi{M^{-1}} \leq \dfrac{n^{n/2}}{\abs{\det(M)}}\cdot \nmi{M}^{n-1}.$$ 
\end{lemma}
\begin{proof}
The case $n=1$ is trivial, so assume $n\geq 2$. Recall that $M^{-1}=\frac{1}{\det(M)} C^t$, where $C$ is the cofactor matrix (see also the proof of part (5) of Proposition \ref{prop:properties_mat_heights} for details). By Lemma \ref{lemma:Had_ineq}, applied to the $(n-1)\times(n-1)$ submatrices of $M$, we have 
$$\nmi{C}\leq (n-1)^{\frac{n-1}{2}} \cdot \nmi{M}^{n-1}\leq n^{n/2} \cdot \nmi{M}^{n-1},$$
and the claim follows from $\nmi{M^{-1}}=\abs{\det(M)}^{-1} \cdot \nmi{C}$.
\end{proof}
Next, we compute $S(v, w)$. Write $\tau=X + \ii Y$, with $X=(x_{j,k})_{1\leq j,k \leq g}$ and $Y=(y_{j,k})_{1\leq j,k \leq g}$ real matrices. Recall that for the bases $(\l_1, \ldots, \l_{2g})$ and $(e_1, \ldots, e_g)$ that we fixed at the start we have 
$$\l_j=\begin{cases}
\sum\limits_{k=1}^{g} x_{j,k} \cdot e_k + y_{j,k}\cdot \ii e_k &j=1, \ldots, g\\
d_{j-g} e_{j-g} &j=g+1, \ldots, 2g 
\end{cases}.$$ 
So, by doing the computations with the basis $(e_1, \ldots, e_g, \ii e_1, \ldots , \ii e_g)$ of $W= \Lambda \otimes \R$, the multiplication by $\ii$ on $W$ is represented in the basis $(\l_1, \ldots, \l_{2g})$ by the matrix
$$\begin{pmatrix}
X & \mathbf{D}\\
Y & 0
\end{pmatrix}^{-1} \begin{pmatrix}
0 & -\mathbf{1}_g \\ 
\mathbf{1}_g & 0
\end{pmatrix} \begin{pmatrix}
X & \mathbf{D}\\
Y & 0
\end{pmatrix}= \begin{pmatrix}
Y^{-1}X & Y^{-1}\mathbf{D}\\
-\mathbf{D}^{-1}Y-\mathbf{D}^{-1}XY^{-1}X & -\mathbf{D}^{-1}XY^{-1}\mathbf{D}
\end{pmatrix}.$$
Hence, the matrix representing $S(v, w)=E(\ii v, w)$ in the basis $(\l_1, \ldots, \l_{2g})$ is given by
\begin{align*}
S&=\begin{pmatrix}
Y^{-1}X & Y^{-1}\mathbf{D}\\
-\mathbf{D}^{-1}Y-\mathbf{D}^{-1}XY^{-1}X & -\mathbf{D}^{-1}XY^{-1}\mathbf{D}
\end{pmatrix}^t \begin{pmatrix}
0 & \mathbf{D} \\ 
-\mathbf{D} & 0
\end{pmatrix}\\
&= \begin{pmatrix}
XY^{-1}X+Y & XY^{-1}\mathbf{D} \\ 
\mathbf{D}Y^{-1}X & \mathbf{D}Y^{-1}\mathbf{D}
\end{pmatrix}.
\end{align*}
Furthermore, note that by \cite[Ex. 5.30]{AM05}
\begin{align*}
\det(S)&=\det(\mathbf{D}Y^{-1}\mathbf{D}) \det\pt{(XY^{-1}X+Y)-(XY^{-1}\mathbf{D})(\mathbf{D}Y^{-1}\mathbf{D})^{-1}(\mathbf{D}Y^{-1}X)}\\
&=\det(\mathbf{D})^2 \cdot \det(Y^{-1}) \cdot \det(Y)=\det(\mathbf{D})^2=d^2
\end{align*}
which also implies that $\det(U)=d$, since $S=U^t U$ and $U$ has positive diagonal entries.

Then, by Lemma \ref{lemma_Chol}, we have that $\nmi{U} \leq \sqrt{2g \nmi{S}}$ and using Lemma \ref{lemma_inv} we get
\begin{equation}\label{eqn:supnorm_inv_U}
\begin{aligned}
\nmi{U^{-1}}&\leq \dfrac{(2g)^g}{d} \cdot \nmi{U}^{2g-1} \\
&\leq \dfrac{(2g)^g}{d} \cdot \pt{2g \max\!\pg{1, \nmi{S}}}^g\\
&=\dfrac{(2g)^{2g}}{d} \cdot \max\!\pg{1, \nmi{S}}^g.
\end{aligned}
\end{equation}
Finally, in preparation for the proof of Corollary \ref{bound_rat_rep}, we establish some bounds for matrices in $\mathfrak{F}_{\Gamma}$. To this end, we first recall a few classical properties of the Siegel fundamental domain $\mathfrak{F}_g$.

\begin{lemma}\label{lemma:prop_F_g}
Let $\tau =X+ \ii Y \in \mathfrak{F}_g$. Then, we have:
\begin{enumerate}[label=(\alph*)]
\item $\nmi{X}\leq \frac{1}{2}$;
\item $\det(Y)\geq \!\pt{\frac{\sqrt{3}}{2}}^{g^2}$;
\item $\abs{\det(C\tau+D)}\geq 1$, for every $\begin{psmallmatrix} A & B\\ C & D \end{psmallmatrix} \in \mathrm{Sp}_{2g}(\Z)$.
\end{enumerate}
\end{lemma}
\begin{proof}
Parts (a) and (c) are true by definition of $\mathfrak{F}_g$ (see \cite[p. 194]{Igu72}).

Moreover, by Lemmas V.13 and V.15 of \cite{Igu72}
$$\det(Y)\geq \pt{\dfrac{3}{4}}^{\frac{g(g-1)}{2}} \cdot (y_{1,1})^g \geq \pt{\dfrac{3}{4}}^{\frac{g(g-1)}{2}} \cdot \pt{\dfrac{\sqrt{3}}{2}}^g=\pt{\dfrac{\sqrt{3}}{2}}^{g^2}$$
where $y_{1,1}$ is the $(1,1)$-entry of $Y$, which proves part (b).
\end{proof}

\begin{proposition}\label{prop:prop_F_Gamma}
Let $\tau =X+ \ii Y \in \mathfrak{F}_{\Gamma}$ and let $Z_{\tau} \in \mathfrak{F}_g$ be such that $\tau=\s\cdot Z_{\tau}$ for some $\s \in \pg{\s_1, \ldots, \s_m}$ (see (\ref{eqn:def_F_Gamma})). Then, there are effective positive constants $\d_1, \d_2, \d_3, \d_4$, depending only on $g$ and the choices of the representatives for the right cosets of $\Gamma$ in $\mathrm{Sp}_{2g}(\Z)$, such that:
\begin{enumerate}[label=(\alph*)]
\item $\nmi{Y} \leq \d_1 \cdot \max\!\pg{1, \nmi{\Im(Z_{\tau})}}^{2g-1}$;
\item $\nmi{X} \leq \d_2 \cdot \max\!\pg{1, \nmi{\Im(Z_{\tau})}}^{g}$;
\item $\det(Y)\geq \dfrac{\d_3}{\max\!\pg{1, \nmi{\Im(Z_{\tau})}}^{2g}}$;
\item $\nmi{Y^{-1}}\leq \d_4 \cdot \max\!\pg{1, \nmi{\Im(Z_{\tau})}}^{2g^2-g+1}$.
\end{enumerate}
\end{proposition}
\begin{proof}
Write $\s=\begin{psmallmatrix} A & B\\ C & D \end{psmallmatrix}$. Since $\s$ is one of the chosen representatives $\s_1, \ldots, \s_m$ for the right cosets of $\Gamma$ in $\mathrm{Sp}_{2g}(\Z)$, all the constants that appear will depend only on $g$ and on the choice of such representatives.
\begin{enumerate}[label=(\alph*)]
\item It is well-known that
\begin{equation}\label{eqn:Im_transf_form}
Y=\Im(\tau)=\Im(\s \cdot Z_{\tau})=\pq{(C Z_{\tau} + D)^t}^{-1}  \Im(Z_{\tau})  \pt{C \overline{Z_{\tau}} + D}^{-1}
\end{equation}
(see for example \cite[Section I.6]{Igu72}). So (\ref{eqn:submult_norm_infty}) implies that
$$\nmi{Y} \leq g^2 \cdot \nmi{(C Z_{\tau}+D)^{-1}}^2 \cdot \nmi{\Im(Z_{\tau})}$$
since $\nmi{\, \cdot \,}$ is invariant under transposition and complex conjugation. Then, as $Z_{\tau} \in \mathfrak{F}_g$, Lemma \ref{lemma_inv} and Lemma \ref{lemma:prop_F_g}(c) imply 
$$\nmi{(C Z_{\tau}+D)^{-1}}\leq \frac{g^{g/2}}{\abs{\det(C Z_{\tau}+D)}} \cdot \nmi{CZ_{\tau}+D}^{g-1}\leq g^{g/2} \cdot \nmi{CZ_{\tau}+D}^{g-1}.$$
Moreover, 
\begin{equation}\label{eqn:bound_nmi_CZ+D}
\begin{aligned}
\nmi{CZ_{\tau}+D}&\leq g \nmi{C} \nmi{Z_{\tau}}+ \nmi{D}\\
& \leq \dfrac{5g}{2} \cdot \max\!\pg{\nmi{C}, \nmi{D}} \cdot \max\!\pg{1, \nmi{\Im(Z_{\tau})}}
\end{aligned}
\end{equation}
since Lemma \ref{lemma:prop_F_g}(a) implies 
\begin{equation}\label{eqn:bound_nmi_Z_tau}
\begin{aligned}
\nmi{Z_{\tau}} +1 &\leq \nmi{\Re(Z_{\tau})}+ \nmi{\Im(Z_{\tau})}+1 \\
&\leq \frac{3}{2} + \nmi{\Im(Z_{\tau})} \leq \frac{5}{2} \cdot \max\!\pg{1, \nmi{\Im(Z_{\tau})}}.
\end{aligned}
\end{equation}
Combining the inequalities above yields 
\begin{align*}
\nmi{Y} &\leq g^2 \cdot \nmi{(C Z_{\tau}+D)^{-1}}^2 \cdot \nmi{\Im(Z_{\tau})}\\
&\leq g^{g+2} \cdot \nmi{CZ_{\tau}+D}^{2g-2} \cdot \nmi{\Im(Z_{\tau})}\\
&\leq \pt{\frac{5}{2}}^{2g-2} \cdot g^{3g} \cdot \max\!\pg{\nmi{C}, \nmi{D}}^{2g-2} \cdot \max\!\pg{1, \nmi{\Im(Z_{\tau})}}^{2g-1}.
\end{align*}
Hence we can take $\d_1=\pt{\frac{5}{2}}^{2g-2} \cdot g^{3g} \cdot \max\limits_{\s \in \pg{\s_1, \ldots, \s_m}}\!\pg{\max\!\pg{\nmi{C}, \nmi{D}}}^{2g-2}$.
\item We have that
\begin{align*}
\nmi{X}=\nmi{\Re(\s \cdot Z_{\tau})} &\leq \nmi{\s \cdot Z_{\tau}} = \nmi{(A Z_{\tau} +B)(C Z_{\tau} +D)^{-1}}\\ &\leq g \cdot \nmi{A Z_{\tau} +B} \cdot \nmi{(C Z_{\tau} +D)^{-1}} \\
&\leq g \cdot \pt{g \nmi{A} \nmi{Z_{\tau}}+ \nmi{B}} \cdot \nmi{(C Z_{\tau} +D)^{-1}}.
\end{align*}
From the computations above we also have that
\begin{align*}
\nmi{(C Z_{\tau} +D)^{-1}} &\leq  g^{g/2} \cdot \nmi{C Z_{\tau}+D}^{g-1}\\
&\leq \pt{\frac{5}{2}}^{g-1} g^{\frac{3}{2}g-1} \max\!\pg{\nmi{C}, \nmi{D}}^{g-1} \max\!\pg{1, \nmi{\Im(Z_{\tau})}}^{g-1}.
\end{align*} 
This implies that
\begin{align*}
\nmi{X}\leq \pt{\frac{5}{2}}^{g} g^{\frac{3}{2}g+1} \nmi{\s}^{g} \max\!\pg{1, \nmi{\Im(Z_{\tau})}}^{g}.
\end{align*}
Hence, we can take $\d_2=\pt{\frac{5}{2}}^{g} g^{\frac{3}{2}g+1} \max\limits_{\s \in \pg{\s_1, \ldots, \s_m}}\!\pg{\nmi{\s}}^{g}$.
\item Taking the determinant of the first and last part of Equation (\ref{eqn:Im_transf_form}) yields
$$\det(Y)=\det(C Z_{\tau}+ D)^{-1} \cdot \det(\Im(Z_{\tau})) \cdot \det(C \overline{Z_{\tau}}+D)^{-1}=\dfrac{\det(\Im(Z_{\tau}))}{\abs{\det(C Z_{\tau}+ D)}^2}.$$
Furthermore, it follows from Lemma \ref{lemma:Had_ineq} and Equation (\ref{eqn:bound_nmi_CZ+D}) that
\begin{align*}
\abs{\det(C Z_{\tau} + D)} &\leq g^{g/2} \cdot \nmi{C Z_{\tau} + D}^g\\
&\leq \pt{\dfrac{5}{2}}^g \cdot g^{3g/2} \cdot \max\!\pg{\nmi{C}, \nmi{D}}^g \cdot \max\!\pg{1, \nmi{\Im(Z_{\tau})}}^g.
\end{align*}
Therefore, using Lemma \ref{lemma:prop_F_g}(b), we get
$$\det(Y)=\dfrac{\det(\Im(Z_{\tau}))}{\abs{\det(C Z_{\tau}+ D)}^2}\geq \pt{\dfrac{\sqrt{3}}{2}}^{g^2} \cdot \pt{\dfrac{2}{5}}^{2g} \cdot g^{-3g} \cdot \max\!\pg{\nmi{C}, \nmi{D}}^{-2g} \cdot \max\!\pg{1, \nmi{\Im(Z_{\tau})}}^{-2g}$$
so that we can take $\d_3=\pt{\frac{\sqrt{3}}{2}}^{g^2} \pt{\frac{2}{5}}^{2g} \cdot g^{-3g} \cdot \max\limits_{\s \in \pg{\s_1, \ldots, \s_m}}\!\pg{\max\!\pg{\nmi{C}, \nmi{D}}}^{-2g}$.
\item Applying Lemma \ref{lemma_inv} and parts (a) and (c) yields
$$\nmi{Y^{-1}} \leq \dfrac{g^{g/2}}{\abs{\det(Y)}} \cdot \nmi{Y}^{g-1} \leq g^{g/2} \cdot \frac{\d_1^{g-1}}{\d_3} \cdot \max\!\pg{1, \nmi{\Im(Z_{\tau})}}^{2g^2-g+1}.$$
Thus, we can take $\d_4=g^{g/2} \cdot \d_1^{g-1}\cdot \d_3^{-1}$. \qedhere
\end{enumerate}
\end{proof}
We are now ready to prove the main result of this section.
\begin{proof}[Proof of Proposition \ref{prop:bound_rat_rep_general}]
We start from the lower bound. By Equation (\ref{eqn:tr_RI_E/S}), we have that $\tr\pt{\r_r(\RI{f}f)}=\tr\pt{E^{-1} \r_r(f)^t\, E \, \r_r(f)}$. As recalled in Section \ref{sect:prelim_AVs}, in the symplectic basis that we fixed at the start of this section, the alternating Riemann form $E$ is represented by the matrix
$$\begin{pmatrix}
0 & \mathbf{D} \\
-\mathbf{D} & 0
\end{pmatrix}$$ 
where $\mathbf{D}\coloneqq\mathrm{diag}(d_1, \ldots, d_g)$ is the type of $\mathcal{L}$. So, it is easy to see that 
$$\tr\pt{\r_r(\RI{f}f)}=2 \sum \limits_{i,j=1}^{g}\dfrac{d_i}{d_j}\pt{m_{i,j} m_{i+g,j+g} - m_{i,j+g} m_{i+g,j}}$$
where $m_{i,j}$ ($i,j=1, \ldots, 2g$) are the entries of the matrix $\r_r(f)$ with respect to the symplectic basis. Moreover, by Equation (\ref{eqn:RI_pos}), we also obtain
\begin{align*}
\tr\pt{\r_r(\RI{f}f)}=\abs{\tr\pt{\r_r(\RI{f}f)}} &\leq 2\sum \limits_{i,j=1}^{g}\dfrac{d_i}{d_j}\pt{\abs{m_{i,j}}\abs{m_{i+g,j+g}} + \abs{m_{i,j+g}} \abs{m_{i+g,j}}}\\
&\leq 4 \nmi{\r_r(f)}^2 \sum\limits_{i,j=1}^{g}\dfrac{d_i}{d_j} \leq 4g^2 \dfrac{d_g}{d_1} \nmi{\r_r(f)}^2
\end{align*}
since $1 \leq d_1 \leq d_2 \leq \ldots \leq d_g$ (see Section \ref{sect:prelim_AVs}). This yields the lower bound.
  
Next, we prove the upper bound. We already proved that 
$$\nmi{R} \leq \pt{(2g)^2 \cdot \nmi{U^{-1}} \cdot \nmi{U}} \cdot \nmi{Q}$$
and, since $\nmi{Q}\leq\nm{Q}_F=\nm{f}_{Ros}$, Equation (\ref{eqn:supnorm_inv_U}) gives
\begin{equation}\label{eqn:bound_rhor_S}
\nmi{\r_r(f)}\leq \dfrac{(2g)^{2g+3}}{d} \cdot \max\!\pg{1, \nmi{S}}^{g+1} \cdot \nm{f}_{Ros}
\end{equation}
for every choice of the bases. Moreover, the block formula for $S$ and Equation (\ref{eqn:submult_norm_infty}) give $\max\!\pg{1,\nmi{S}}\leq 2g^2\nmi{\mathbf{D}}^2 N_{\tau}^3$, since $\nmi{\mathbf{D}}\geq 1$, which proves the upper bound.  
\end{proof}
  
\begin{proof}[Proof of Corollary \ref{bound_rat_rep}]
By Equation (\ref{eqn:bound_rhor_S}), it suffices to bound $\nmi{S}$. 
Let $\tau = X + \ii Y$. Then, using Equation (\ref{eqn:submult_norm_infty}), we get
\begin{align*}
\nmi{S}&=\max\!\pg{\nmi{XY^{-1}X+Y}, \nmi{XY^{-1}\mathbf{D}}, \nmi{\mathbf{D}Y^{-1}X}, \nmi{\mathbf{D}Y^{-1}\mathbf{D}}}\\
&\leq \max\!\pg{\nmi{Y}+g^2\nmi{X}^2\nmi{Y^{-1}}, g^2\nmi{X}\nmi{Y^{-1}}\nmi{\mathbf{D}}, g^2 \nmi{\mathbf{D}}^2\nmi{Y^{-1}}}.
\end{align*}
Moreover, by Proposition \ref{prop:prop_F_Gamma} (observe that the explicit values above give $\d_1,\d_2,\d_4\geq 1$ and $\d_1\leq\d_4$), we also obtain:
$$\nmi{Y}+g^2\nmi{X}^2\nmi{Y^{-1}} \leq 2 g^2 \d_2^2 \d_4 \cdot \max\!\pg{1, \nmi{\Im(Z_{\tau})}}^{2g^2+g+1},$$
$$g^2\nmi{X}\nmi{Y^{-1}}\nmi{\mathbf{D}} \leq g^2 \d_2 \d_4\nmi{\mathbf{D}} \cdot \max\!\pg{1, \nmi{\Im(Z_{\tau})}}^{2g^2+1},$$
$$g^2 \nmi{\mathbf{D}}^2\nmi{Y^{-1}} \leq g^2 \d_4 \nmi{\mathbf{D}}^2\cdot \max\!\pg{1, \nmi{\Im(Z_{\tau})}}^{2g^2-g+1}.$$
So, we get
\begin{align*}
\nmi{S}&\leq \max\!\pg{\nmi{Y}+g^2\nmi{X}^2\nmi{Y^{-1}}, g^2\nmi{X}\nmi{Y^{-1}}\nmi{\mathbf{D}}, g^2 \nmi{\mathbf{D}}^2\nmi{Y^{-1}}}\\
&\leq 2 g^2 \d_2^2 \d_4 \nmi{\mathbf{D}}^2 \cdot \max\!\pg{1, \nmi{\Im(Z_{\tau})}}^{2g^2+g+1}.
\end{align*}
Thus,
$$(2g)^2 \cdot \nmi{U^{-1}} \cdot \nmi{U}\leq \dfrac{(2g)^{2g+3}}{d} \cdot \max\!\pg{1, \nmi{S}}^{g+1}\leq c_3(A)$$
where 
$$c_3(A)=2^{3g+4} \cdot g^{4g+5} \cdot (\d_2^2 \d_4)^{g+1} \cdot \dfrac{\nmi{\mathbf{D}}^{2g+2}}{d} \cdot \max\!\pg{1, \nmi{\Im(Z_{\tau})}}^{2g^3+3g^2+2g+1}.$$ 
Note that $\d(g, \mathfrak{F}_{\Gamma})=2^{3g+4} \cdot g^{4g+5} \cdot (\d_2^2 \d_4)^{g+1}$ is an effective positive constant that depends only on $g$ and the choice of the representatives for the right cosets of $\Gamma$ in $\mathrm{Sp}_{2g}(\Z)$.
\end{proof}

\begin{remark}
If $\Gamma=\mathrm{Sp}_{2g}(\Z)$ (e.g.\ if $\mathbf D=\mathbf 1_g$), so that $\mathfrak{F}_{\Gamma}=\mathfrak{F}_g$ and $Z_{\tau}=\tau$, one can obtain a better value for the constant $c_3(A)$, namely
$$c_3(A)=2^{3g+4} \cdot g^{\frac{g^2+9g}{2}+5} \cdot \pt{\dfrac{2\sqrt{3}}{3}}^{g^2(g+1)} \cdot \dfrac{\nmi{\mathbf{D}}^{2g+2}}{d} \cdot \max\!\pg{1, \nmi{\Im(\tau)}}^{(g+1)\max\pg{1,g-1}}.$$
Indeed, writing $\tau=X+\ii Y$, Lemmas \ref{lemma_inv} and \ref{lemma:prop_F_g} give $\nmi{X}\leq \frac{1}{2}$ and $\nmi{Y^{-1}}\leq g^{g/2}\pt{\frac{2\sqrt{3}}{3}}^{g^2}\max\!\pg{1,\nmi{Y}}^{g-1}$, so that
$$\max\!\pg{1,\nmi{S}}\leq 2g^{2} \cdot g^{g/2}\pt{\dfrac{2\sqrt{3}}{3}}^{g^2} \nmi{\mathbf{D}}^2 \cdot \max\!\pg{1,\nmi{Y}}^{\max\pg{1,g-1}},$$
and one concludes as in the proof above.
\end{remark}

\section{The main estimate}\label{sect:main_estim_CM}
Recall the setting introduced at the end of Section \ref{sec_red}. Fix a symplectic basis of the period lattice of $B$ (with respect to $\l_B$) and a basis of $\C^h$ such that the period matrix of $B$ is $(\tau_B,\mathbf{1}_h)$, with $\tau_B\in\HH_h$. Fix also $w\in\C^h$ whose image in $B(\C)=\C^h/(\Z^h+\tau_B\Z^h)$ is $b$. For $(\tau,z)\in\HH_g\times\C^g$, set
$$\widetilde{\tau}=\begin{pmatrix}\tau & 0\\ 0 & \tau_B\end{pmatrix}\in\HH_{g+h},\qquad \widetilde{z}=\begin{pmatrix} z\\ w\end{pmatrix}\in\C^{g+h},$$
and $\kappa=\kappa_{g+h}=2^{g+h}$ (if $h=0$, then $\widetilde{\tau}=\tau$ and $\widetilde{z}=z$). For every $T\geq 1$ we define
\begin{align*}
\mathcal{Z}(T)=\Big\lbrace&\pt{\tau, z} \in \mathcal{Z}: \exists M \in \mat_{g+h}(\C) \text{ whose first $g$ columns are not all zero, s.t. } \\
& M \widetilde{z} \in \Z^{g+h}+\widetilde{\tau} \Z^{g+h}, H_{\kappa}(\tau),H_{\kappa}(M) \leq T \text{ and } \det(\Im(\tau))\geq \frac{1}{T}\Big\rbrace
\end{align*}
where $H_{\kappa}$ is the height defined in Definition \ref{def_d_height}. 

We want to prove the following upper bound for the cardinality of $\mathcal{Z}(T)$.
\begin{proposition}\label{prop:main_estim_CM}
Under the hypotheses of Theorem \ref{thm:main_product}, for all $\eps>0$, we have $\# \mathcal{Z}(T)\ll_{\eps} T^{\eps}$, for all $T\geq 1$.
\end{proposition}
In order to prove this, consider the definable set $W$ whose elements are tuples of the form
\begin{align*}
(\a_{1,1}, \ldots, \a_{g+h,g+h}, &\b_{1,1}, \ldots, \b_{g+h,g+h}, \m_{1,1}, \ldots, \m_{1,g+h}, \m_{2,1}, \ldots, \m_{2,g+h},\\
  &x_{1,1}, \ldots, x_{g,g}, y_{1,1}, \ldots, y_{g,g}, r_1, \ldots, r_g, s_1, \ldots, s_g)
\end{align*}
in $\R^{2(g+h)^2+2(g+h)+2g^2}\times\R^{2g}$, satisfying the following relations:
$$\text{the first } g \text{ columns of } M \text{ are not all zero}, \qquad \pt{\tau, z} \in \mathcal{Z}, \qquad M\widetilde{z}=\m_1+\widetilde{\tau}\m_2$$
where $$M=\pt{\a_{i,j}+\ii \b_{i,j}}_{i,j=1, \ldots, g+h}, \quad \m_1=\pt{\m_{1,1}, \ldots, \m_{1,g+h}}^t, \quad \m_2=\pt{\m_{2,1}, \ldots, \m_{2,g+h}}^t,$$
$$\tau=\pt{x_{i,j}+\ii y_{i,j}}_{i,j=1, \ldots, g}, \quad z=\pt{z_1, \ldots, z_g}^t=\pt{r_1 + \ii s_1, \ldots, r_g + \ii s_g}^t.$$
Moreover, for any integer $n\geq 1$ and for any $T\geq 1$, let
$$W^{\sim}(n,T)\coloneqq\pg{(\a_{1,1}, \ldots, s_g) \in W: H_{n}(\a_{1,1}, \ldots, y_{g,g})\leq T}.$$
Recall that $H_{n}(\a_{1,1}, \ldots, y_{g,g})$ is finite if and only if $\a_{1,1}, \ldots, y_{g,g}$ are all algebraic numbers of degree at most $n$. 

Now, let $\pi_1, \pi_2$ be the projections onto the first $2(g+h)^2+2(g+h)+2g^2$ and the last $2g$ coordinates, respectively.

The proofs of Lemmas \ref{lemma:bound_ll_Teps_CM} and \ref{lemma:bound_preimages_CM} rely on the fact that an algebraic relation between $\tau$ and $z$ holding along a real arc in $\mathcal{Z}$ persists on an open disc in $\cC(\C)$, where Lemma \ref{lemma:func_transc_CM} can be applied. We prove this first.
\begin{lemma}\label{lemma:arc_to_disc}
Let $I \subseteq \R$ be an open interval and let $\theta: I \to \mathcal{Z}$ be a continuous map which is not constant on any open subinterval of $I$. Then, there exist an open subinterval $I_1\subseteq I$, an open disc $D\subseteq\cC(\C)$ containing no singular point of $\cC$ and holomorphic functions $\tau,z$ on $D$ with $u(\tau(P),z(P))=P$ for every $P\in D$, such that $u(\theta(I_1))\subseteq D$ and $\theta(t)=(\tau,z)(u(\theta(t)))$ for every $t\in I_1$. Moreover, a polynomial function $Q$ on $\mathrm{Sym}_g(\C)\times\C^g$ vanishes on $\theta(I_1)$ if and only if $Q(\tau,z)=0$ on $D$.
\end{lemma}
\begin{proof}
Since $\theta$ is continuous and not constant, $\theta(I)$ is infinite, and so is $u(\theta(I))$, as $u$ is finite-to-one on $\mathcal{F}_g$. As $\cC$ has finitely many singular points, we can choose $t_1\in I$ such that $P_1=u(\theta(t_1))$ is a smooth point of $\cC(\C)$. Since $\Gamma(3)\ltimes\Z^{2g}$ is torsion-free and acts properly discontinuously on $\HH_g\times\C^g$, it acts freely. Hence, $u$ is a local biholomorphism and there are open neighbourhoods $V$ of $\theta(t_1)$ and $U$ of $P_1$ such that $u|_V:V\to U$ is biholomorphic. Let $D\subseteq\cC(\C)\cap U$ be an open disc containing $P_1$ and no singular point of $\cC$, set $(\tau,z)=(u|_V)^{-1}|_D$ and, by continuity, choose an open interval $I_1\ni t_1$ with $\theta(I_1)\subseteq V$ and $u(\theta(I_1))\subseteq D$. Then $\theta(t)=(\tau,z)(u(\theta(t)))$ for every $t\in I_1$; in particular, if $Q(\tau,z)=0$ on $D$, then $Q$ vanishes on $\theta(I_1)$. Conversely, if $Q$ vanishes on $\theta(I_1)$, the holomorphic function $Q(\tau,z)$ vanishes on $u(\theta(I_1))$. This set is connected, contains $P_1$ and is infinite (otherwise $\theta$ would be constant on $I_1$), so it accumulates at $P_1$, and the identity theorem gives $Q(\tau,z)=0$ on $D$.
\end{proof}

\begin{lemma}\label{lemma:bound_ll_Teps_CM}
Under the hypotheses of Theorem \ref{thm:main_product}, for any integer $n\geq 1$ and for every $\eps>0$, we have 
$$\# \pi_2\pt{W^{\sim}(n,T)}\ll_{\eps, n} T^{\eps}$$ for all $T\geq 1$.
\end{lemma}
\begin{proof}
Fix $\eps>0$ and, by contradiction, suppose that for some $T_0\geq 1$,
$$\# \pi_2\pt{W^{\sim}(n,T_0)}> c  \cdot T_0^{\eps}$$
where $c=c(W, n, \eps)$ is the constant given by Proposition \ref{prop:thm_HP_CM}. 

Then, by Proposition \ref{prop:thm_HP_CM}, there exists a continuous definable function $\beta\colon \pq{0,1} \rightarrow W$ such that $\beta_1=\pi_1 \circ \beta \colon\pq{0,1} \rightarrow \R^{2(g+h)^2+2(g+h)+2g^2}$ is semialgebraic, $\beta_2=\pi_2 \circ \beta \colon \pq{0,1} \rightarrow \R^{2g}$ is non-constant and $\beta|_{(0,1)}$ is real analytic. In particular, $\beta_1(\pq{0,1})$ is contained in a real algebraic curve and, being real analytic and non-constant on $(0,1)$, $\beta_2$ is not constant on any open subinterval of $(0,1)$.

Let $M, \tau, \m_1, \m_2, z=(z_1, \ldots, z_g)^t$ be as above and consider the coordinates 
\begin{align*}
\a_{1,1}, \ldots, \a_{g+h,g+h}, &\b_{1,1}, \ldots, \b_{g+h,g+h}, \m_{1,1}, \ldots, \m_{1,g+h}, \m_{2,1}, \ldots, \m_{2,g+h},\\
 &x_{1,1}, \ldots, x_{g,g}, y_{1,1}, \ldots, y_{g,g}, r_1, \ldots, r_g, s_1, \ldots, s_g
\end{align*}
as functions on $(0,1)$. They lie in the ring $\mathcal{O}$ of complex-valued real analytic functions on $(0,1)$, which is an integral domain by the identity theorem, and all the fields below are contained in its field of fractions.

Note that $\tau$ cannot be constant on $(0,1)$, otherwise there would be infinitely many points on $\cC$ (since $\beta_2((0,1))$ is infinite and $u(\tau,\cdot)$ is injective on $L_{\tau}$) that lie on the same fiber, which contradicts the assumption that $\cC$ is not contained in any fiber. As a non-constant element of $\mathcal{O}$ is transcendental over $\C$, we get $\mathrm{tr.deg.}_{\C}\C(\tau)\geq 1$.

On the other hand, since $\beta_1(\pq{0,1})$ is contained in an algebraic curve, the field $E$ generated over $\C$ by $\a_{1,1}, \ldots, y_{g,g}$ has transcendence degree at most $1$ over $\C$. As $\C(\tau)\subseteq E$, this gives $\mathrm{tr.deg.}_{\C}\C(\tau)=1$ and shows that $E$ is algebraic over $\C(\tau)$. Hence $E(z)$ is algebraic over $\C(\tau)(z)$, and we have
$$\mathrm{tr.deg.}_{\C(\tau)}\C(\tau)(z)=\mathrm{tr.deg.}_{E}E(z).$$
Since the first $g$ columns of $M$ are not all zero, some entry $m_{i,j}\in E$ of $M$ with $j\leq g$ is non-zero. As the entries of $w$ and $\tau_B$ are constant, the $i$-th row of $M\widetilde{z}=\m_1+\widetilde{\tau}\m_2$ gives $z_j \in E(z_k: k\neq j)$. Thus $E(z)$ is generated over $E$ by at most $g-1$ elements, so that $\mathrm{tr.deg.}_{\C(\tau)}\C(\tau)(z)\leq g-1$ on $(0,1)$.

Finally, let $I_1$, $D$ and the holomorphic functions $\tau,z$ on $D$ be given by Lemma \ref{lemma:arc_to_disc} applied to $\theta=(\tau,z):(0,1)\to\mathcal{Z}$. If $Q$ is a polynomial function on $\mathrm{Sym}_g(\C)\times\C^g$, then $Q(\tau,z)\in\mathcal{O}$ vanishes on $(0,1)$ if and only if it vanishes on $I_1$, i.e.\ if and only if $Q(\tau,z)=0$ on $D$. Hence $\tau$ and $z$ satisfy the same algebraic relations over $\C$ on $(0,1)$ and on $D$, so the transcendence degrees above are the same on $D$. Thus $\mathrm{tr.deg.}_{\C(\tau)}\C(\tau)(z)\leq g-1$ on $D$, contradicting Lemma \ref{lemma:func_transc_CM}.
\end{proof}

\begin{lemma}\label{lemma:bound_preimages_CM}
Under the hypotheses of Theorem \ref{thm:main_product}, there exists a positive constant $c'=c'(\mathcal{Z})$ such that, for every $z \in \C^{g}$, there are at most $c'$ elements $\tau \in \HH_g$ such that $(\tau, z) \in \mathcal{Z}$. In particular, the same holds with $\mathcal{Z}(T)$ in place of $\mathcal{Z}$, for every $T\geq 1$.
\end{lemma}
\begin{proof}
Let $\widetilde{\pi}:\mathcal{Z}\to\C^g$ be the projection $(\tau,z)\mapsto z$. Its fibers form a definable family, so by Proposition \ref{prop_fibre_finite} they have a uniformly bounded number of connected components, and it suffices to show that every fiber of $\widetilde{\pi}$ is finite.

Suppose that $\widetilde{\pi}^{-1}(z_0)$ is infinite for some $z_0\in\C^g$. Then it has positive dimension, so it contains $\theta((0,1))$ for some continuous injective map $\theta:(0,1)\to\mathcal{Z}$. Let $I_1$, $D$, $\tau$, $z$ be given by Lemma \ref{lemma:arc_to_disc} applied to $\theta$. Since the coordinates of $z-z_0$ vanish on $\theta(I_1)$, $z$ is constant on $D$, so that $\mathrm{tr.deg.}_{\C(\tau)} \C(\tau)(z)=0<g$, contradicting Lemma \ref{lemma:func_transc_CM}.
\end{proof}

\begin{proof}[Proof of Proposition \ref{prop:main_estim_CM}]
If $(\tau, z) \in \mathcal{Z}(T)$, then there exist a matrix $M \in \mat_{g+h}(\Qal)$, whose first $g$ columns are not all zero, with $H_{\kappa}(M)\leq T$, and vectors $\m_1, \m_2 \in \Z^{g+h}$ such that $M \widetilde{z} = \m_1+ \widetilde{\tau} \m_2$. By Lemma \ref{lemma:bound_mod_hpoly}, the entries of $M$ and $\tau$ have absolute value $\ll T$. Writing $z=p+\tau q$ with $p,q\in[0,1)^g$, we also get $\abs{z_i}\leq 1+\sum_{j}\abs{\tau_{i,j}}\ll T$. Since $w$ and $\tau_B$ are fixed, this gives $\nmi{M\widetilde{z}}\ll T^2$. Moreover, $\Im(\widetilde{\tau})\m_2=\Im(M\widetilde{z})$ and $\Im(\widetilde{\tau})^{-1}=\mathrm{diag}\pt{\Im(\tau)^{-1},\Im(\tau_B)^{-1}}$, where $\nmi{\Im(\tau)^{-1}}\ll T^g$ by Lemma \ref{lemma_inv}, since $\det(\Im(\tau))\geq \frac{1}{T}$. Hence
$$\nmi{\m_2}\ll T^g\nmi{M\widetilde{z}}\ll T^{g+2} \qquad \text{and} \qquad \nmi{\m_1}\leq\nmi{M\widetilde{z}}+(g+h)\nmi{\widetilde{\tau}}\nmi{\m_2}\ll T^{g+3}.$$
Additionally, Lemma \ref{lemma:bounds_hpoly_ReIm} with $d=\kappa$ implies that 
$$H_{2N}(\Re(M)), H_{2N}(\Re(\tau)) \ll T^{N} \qquad \text{and} \qquad H_{2N}(\Im(M)), H_{2N}(\Im(\tau)) \ll T^{2N}$$
where $N=\frac{\kappa(\kappa-1)}{2}$. Moreover, since $\m_1, \m_2$ have integer entries, (\ref{eqn:prop_hpoly}) gives 
$$H_{2N}(\m_1, \m_2)\leq H_1(\m_1, \m_2)=\max\!\pg{1, \nmi{\m_1}, \nmi{\m_2}} \ll T^{g+3}.$$
This allows us to deduce that 
$$\pt{\Re(M),\Im(M), \m_1, \m_2, \Re(\tau), \Im(\tau), \Re(z), \Im(z)} \in W^{\sim}(2N, \nu T^{\max\!\pg{2N, g+3}})$$
for some positive constant $\nu$. Then, by Lemma \ref{lemma:bound_preimages_CM}, each element of 
$$\pi_2(W^{\sim} (2N, \nu T^{\max\!\pg{2N, g+3}}))$$
corresponds to at most $c'$ distinct elements of $\mathcal{Z}(T)$. Finally, the proof follows from Lemma \ref{lemma:bound_ll_Teps_CM}.
\end{proof}

\section{Canonical height bounds under endomorphisms}\label{sec_can_bounds}
Let $A$ be an abelian variety of dimension $g$ defined over $\Qal$, let $D$ be a symmetric divisor on $A$, and let $\widehat{h}_{A,D}$ be the canonical height on $A(\Qal)$ associated with $D$ (see Section \ref{sect:prelim_heights}).

As mentioned in the introduction, our aim is to generalize the usual identity $\widehat{h}_{A,D}([n]P)=n^2 \cdot \widehat{h}_{A,D}(P)$ to general endomorphisms of $A$.

It was noted by Naumann (see the remark at the end of Section 2 of \cite{Nau04}) that, if $\en^0(A)$ is $\Q$, an imaginary quadratic field or a definite quaternion algebra over $\Q$, and if $D$ is a symmetric divisor, then
$$\widehat{h}_{A,D}(f(P)) = \pt{\deg f}^{1/g} \cdot \widehat{h}_{A,D}(P)$$
for any $f \in \en(A)$ and any $P\in A(\Qal)$, recovering a well-known fact for elliptic curves.

In general, however, we cannot expect an identity of this form, as illustrated by the following examples.

\begin{example}
Consider $A=E \times E$, where $E$ is any elliptic curve with identity element $O$, and let $D=\pt{O \times E} + \pt{E \times O}$. Define the endomorphism $f: A \rightarrow A$ by 
$$f(P_1,P_2)=(P_1,2P_2).$$ 
Since we can write $D=\pi_1^*(O)+ \pi_2^*(O)$, where $\pi_1$ and $\pi_2$ are the projections onto the two factors, we obtain
$$\widehat{h}_{A,D}(P_1,P_2)=\widehat{h}_{E, (O)}(\pi_1(P_1,P_2))+\widehat{h}_{E, (O)}(\pi_2(P_1,P_2))=\widehat{h}_{E, (O)}(P_1)+\widehat{h}_{E, (O)}(P_2)$$
by applying Proposition \ref{prop:propr_can_heights}. Choosing either $P_1=O$ or $P_2=O$, we conclude that there is no constant $\g$ such that
$$\widehat{h}_{A,D}(f(P))=\widehat{h}_{E, (O)}(P_1)+4\widehat{h}_{E, (O)}(P_2)=\g\cdot \pt{\widehat{h}_{E, (O)}(P_1)+\widehat{h}_{E, (O)}(P_2)}=\g \cdot \widehat{h}_{A,D}(P)$$
for every $P=(P_1,P_2) \in A(\Qal)$.
Nonetheless, since the divisor $(O)$ is ample, it follows that 
$$\widehat{h}_{A,D}(P)\leq \widehat{h}_{A,D}(f(P))\leq 4 \widehat{h}_{A,D}(P).$$
\end{example}

\begin{example}
Let $A/\Qal$ be a simple abelian variety such that $\en^0_{\Qal}(A)$ is a quadratic field with discriminant $N>0$. Let $f = a + b\sqrt{N} \in \en(A)$. Naumann proved (see \cite[Section 3]{Nau04}) that there exists a constant $c(f)$ such that $\widehat{h}_{A,D}(f(P))=c(f) \cdot \widehat{h}_{A,D}(P)$ for every $P \in A(\Qal)$ and every ample symmetric divisor $D$ if and only if $ab = 0$. However, for every $f = a + b\sqrt{N} \in \en(A)$, using \cite[Theorem 3]{Nau04} one can deduce that, for every ample symmetric divisor $D$ and every $P \in A(\Qal)$, one has
$$\a \cdot \widehat{h}_{A,D}(P)\leq \widehat{h}_{A,D}(f(P))\leq \b \cdot \widehat{h}_{A,D}(P)$$
where $\a=\min\!\pg{(a+b\sqrt{N})^2, (a-b\sqrt{N})^2}$ and $\b=\max\!\pg{(a+b\sqrt{N})^2, (a-b\sqrt{N})^2}$.
\end{example}

More generally, if $D$ is ample and symmetric and $f \in \en(A)$, there exist constants $0 \leq \g_1 \leq \g_2$, depending on $f$ and $D$, such that
$$\g_1 \cdot \widehat{h}_{A,D}(P)\leq \widehat{h}_{A,D}(f(P))\leq \g_2 \cdot \widehat{h}_{A,D}(P)$$
for every $P \in A(\Qal)$.
In particular, as anticipated in the introduction, one must take $\gamma_1=0$ if $f$ is not an isogeny, while $\gamma_1$ can be chosen strictly positive when $f$ is an isogeny. We now provide the proof of this fact.

To prove the upper bound, recall that since $D$ is ample, there exists an integer $N_2>0$ such that $nD - f^* D$ is ample for all $n\geq N_2$, see for instance \cite[Example 1.2.10]{Laz04}.
This implies
$$N_2 \cdot \widehat{h}_{A,D}(P)-\widehat{h}_{A,D}(f(P))=\widehat{h}_{A,N_2 D- f^* D}(P)\geq 0$$
giving the upper bound with $\g_2=N_2$.

For the lower bound, first observe that if $f$ is not finite, then the dimension of $\ker(f)$ is positive and, in particular, there is a non-torsion point $P \in A(\Qal)$ for which $f(P)=O$. Therefore, we must have $\g_1=0$ in this case. On the other hand, if $f$ is finite then $f^* D$ is ample. Thus, as before, there exists an integer $N_1>0$ such that $n f^* D - D$ is ample for any $n\geq N_1$. This means that
$$N_1 \cdot \widehat{h}_{A,D}(f(P))-\widehat{h}_{A,D}(P)=\widehat{h}_{A,N_1 f^* D- D}(P)\geq 0$$
from which we deduce the lower bound, with $\g_1 =\frac{1}{N_1}>0$.

If $f$ is an isogeny, the existence of these bounds also follows from Theorem B of \cite{Lee16}.

Unfortunately, this method does not provide effective values for $\g_1$ and $\g_2$, although explicit computations may be possible for specific choices of $D$ and $f$.

However, as mentioned in the introduction, we are able to provide general explicit values for $\g_1$ and $\g_2$. Let $\RI{}$ be the Rosati involution associated with the divisor $D$ (or, more formally, with the line bundle $\cO(D)$) and let $\a_1, \ldots, \a_g$ be the eigenvalues (counted with multiplicities) of $\r_a(\RI{f}f)$. We will prove in Lemma \ref{lemma:eigen_pos} that these eigenvalues are real and non-negative. As above, we set
$$\a^{-}_{D}(f)=\min\!\pg{\a_1, \ldots, \a_g} \quad \text{ and } \quad \a^{+}_{D}(f)=\max\!\pg{\a_1, \ldots, \a_g}.$$
We recall the statement of the main result of this section, already stated in the introduction.

\canheightsendo*

\begin{remark}
Note that the ampleness hypothesis for $D$ is necessary, as shown by the following example. As before, take $A=E \times E$, where $E$ is any elliptic curve with identity element $O$, and $D=\pi_1^*(O)$. Observe that $D$ is nef, symmetric but not ample. By Proposition \ref{prop:propr_can_heights}, we get
$$\widehat{h}_{A,D}(P_1,P_2)=\widehat{h}_{E, (O)}(\pi_1(P_1,P_2))=\widehat{h}_{E, (O)}(P_1).$$
If $\f\in \en(A)$ is given by $\f(P_1,P_2)=(P_2, P_1)$, then we get that
$$\widehat{h}_{A,D}(\f(P))=\widehat{h}_{E, (O)}(P_2)$$
for every $P=(P_1,P_2) \in A(\Qal)$, and it is easy to see that there is no positive constant $\g$ such that 
$$\widehat{h}_{A,D}(\f(P)) \leq \g \cdot \widehat{h}_{A,D}(P)$$ 
for every $P=(P_1,P_2) \in A(\Qal)$.
\end{remark}

\subsection{Properties of endomorphisms and line bundles of abelian varieties}
Fix an ample divisor $D$ on $A=\C^g/\Lambda$ and let $L = \cO_A(D)$ be the associated line bundle. In the following, $\RI{}$ denotes the Rosati involution induced by the polarization $\Phi_L$ corresponding to $L$.
 
We start with a classical result about the eigenvalues of $\rho_a(\RI{f}f)$. The following proof is inspired by an argument by Masser and W\"ustholz \cite{MW94}.
\begin{lemma}\label{lemma:eigen_pos}
Let $f \in \en^0(A)$. Then all the eigenvalues of $\rho_a(\RI{f}f)$ are real and non-negative. If $f \neq 0$, then at least one eigenvalue is positive.
\end{lemma}
\begin{proof}
By Proposition 5.1.1 of \cite{BL04}, we have that $H_L(\r_a(f)v, w)=H_L(v, \r_a(\RI{f})w)$, for every $v,w \in \C^g$, where $H_L: \C^g \times \C^g \to \C$ is the Hermitian form associated with the ample line bundle $L$. Thus, if $\mathcal{H}_L$ is the matrix representing $H_L$, we have 
$$\r_a(\RI{f})=\pt{\overline{\mathcal{H}_L}}^{-1} \cdot \overline{\r_a(f)}^t \cdot \overline{\mathcal{H}_L}$$
where $\overline{M}^t$ is the conjugate transpose of the matrix $M$.

Since $L$ is ample, $H_{L}$ is positive definite, and therefore there is an invertible matrix $S$ such that $\overline{\mathcal{H}_{L}}=\overline{S}^t S$. Thus, we have
$$\rho_a(\RI{f}f)=\pt{\overline{\mathcal{H}_L}}^{-1} \cdot \overline{\r_a(f)}^t \cdot \overline{\mathcal{H}_L} \cdot \rho_a\pt{f}=S^{-1} (\overline{S}^t)^{-1} \overline{\r_a(f)}^t \overline{S}^t S \r_a(f).$$
By setting $X=S \cdot \r_a(f) \cdot S^{-1}$, we have that 
$$\rho_a(\RI{f}f)=S^{-1} (\overline{S}^t)^{-1} \cdot \overline{\rho_a(f)}^t  \cdot \overline{S}^t S \cdot \rho_a(f) \cdot S^{-1} S=S^{-1} \overline{X}^t X S,$$
proving that every eigenvalue of $\rho_a(\RI{f}f)$ is real and non-negative, since $\overline{X}^t X$ is a positive semidefinite matrix and eigenvalues are invariant under change of basis. In particular, as Hermitian matrices are diagonalizable, this also implies that $\overline{X}^t X$ cannot have all zero eigenvalues unless it is the zero matrix. Finally, if $X$ has entries $x_{i,j} \in \C$ and $\overline{X}^t X=\mathbf{0}$, then $0=\tr(\overline{X}^t X)=\sum_{i,j=1}^g \abs{x_{i,j}}^2$, which implies that $X=S \cdot \rho_a(f) \cdot S^{-1}=\mathbf{0}$ and thus $\r_a(f)=\mathbf{0}$, i.e.\ $f=0$, since $\r_a$ is injective.
\end{proof}
Notice that for $f \in \en(A)$, the matrix $\rho_a(\RI{f}f)$ has only positive eigenvalues if and only if $X$ is invertible, which is the case precisely when $\rho_a(f)$ is invertible, i.e.\ when $f$ is an isogeny.

Denote by $P_{\RI{f}f}^a(x)=\det\pt{x\mathbf{1}_g-\r_a(\RI{f}f)}$ and $P_{\RI{f}f}^r(x)=\det\pt{x\mathbf{1}_{2g}-\r_r(\RI{f}f)}$ the characteristic polynomials of $\RI{f}f$ with respect to the analytic and the rational representations, respectively. 
Using \cite[Proposition 5.1.2]{BL04} and the previous lemma, $P_{\RI{f}f}^a$ and $P_{\RI{f}f}^r$ are real polynomials and we have
\begin{equation}\label{eq:poly_char_RI}
P_{\RI{f}f}^r(x) = \pt{P_{\RI{f}f}^a(x)}^2.
\end{equation}

\begin{remark}
The equation above implies, in particular, that the eigenvalues of $\r_r(\RI{f} f)$ coincide with those of $\r_a(\RI{f} f)$, each occurring with twice the multiplicity. Since the entries of $\r_r(\RI{f} f)$ are rational, it follows that its eigenvalues (and hence also those of $\r_a(\RI{f} f)$) are (real) algebraic numbers. Moreover, if the polarization $\Phi_L$ is principal and $f \in \en(A)$, then $\RI{f} \in \en(A)$ as well, so that $\r_r(\RI{f} f)$ has integer entries. Consequently, all its eigenvalues are (real) algebraic integers.
\end{remark}

With these notations, we have the following generalization of Lemma 2.1 of \cite{Lan88} (see also \cite[Proposition 5.1.6]{BL04}).
\begin{lemma}\label{Lemma_Lange}
Let $L$ be an ample line bundle, $f \in \en(A)$ and $a, b \in \Z$, with $b>0$. Then, 
$$\chi(f^*L^{-b} \otimes L^a) = \chi(L) \cdot b^{g} \cdot P_{\RI{f}f}^a\!\pt{\dfrac{a}{b}}.$$
\end{lemma}
\begin{proof}
By Corollary 3.6.2 of \cite{BL04}, we have
$$\chi(f^*L^{-b} \otimes L^a)^2 = \deg(\Phi_{f^*L^{-b} \otimes L^a})$$
where the map $\Phi_{L}$ was defined in (\ref{eqn:def_Phi_L}). By \cite[Corollary 2.4.6]{BL04} we have 
$$\Phi_{f^*L^{-b} \otimes L^a}=-[b] \Phi_{f^*L} + [a] \Phi_{L} \quad \text{and} \quad \Phi_{f^*L}=\widehat{f} \Phi_L f=\Phi_L \RI{f} f.$$
Then, recalling that $\deg(\f)=\det(\r_r(\f))$ for every $\f \in \en(A)$ \cite[eq. (1.2)]{BL04}, and hence for every $\f \in \en^0(A)$ if we set $\deg(\f/n)=n^{-2g}\deg(\f)$ (note that $\RI{f}f$ need not lie in $\en(A)$ unless $\Phi_L$ is principal), we get
\begin{align*}
\chi(f^*L^{-b} \otimes L^a)^2 &= \deg\pt{-[b]\Phi_L \RI{f} f + [a]\Phi_{L}} \\
&= \deg\Phi_{L} \cdot \deg\pt{-[b] \RI{f} f  + [a]} \\
&= \deg\Phi_{L} \cdot \det\pt{\r_r\!\pt{-[b] \cdot \RI{f} f  + [a]}} \\
&= \deg\Phi_{L} \cdot \det\pt{-b \cdot \r_r(\RI{f} f)  + a \cdot \mathbf{1}_{2g}} \\
&= \deg\Phi_{L} \cdot b^{2g} \cdot \det\pt{-\r_r(\RI{f} f) + \dfrac{a}{b}\cdot \mathbf{1}_{2g}} \\
&= \chi(L)^2  \cdot b^{2g} \cdot P_{\RI{f}f}^r\pt{\dfrac{a}{b}}=\chi(L)^2  \cdot b^{2g} \cdot \pt{P_{\RI{f}f}^a\!\pt{\dfrac{a}{b}}}^2
\end{align*}
by Equation (\ref{eq:poly_char_RI}). Here $\mathbf{1}_{2g}$ is the $2g \times 2g$ identity matrix. It follows that
$$\chi(f^*L^{-b} \otimes L^a) = \pm \chi(L) \cdot b^{g} \cdot P_{\RI{f}f}^a\!\pt{\dfrac{a}{b}}.$$
For fixed $b$, both sides are polynomials in $a$. Since $L$ is ample, we have $\chi(L) > 0$. Moreover, for all sufficiently large $a > 0$, the line bundle $f^*L^{-b} \otimes L^a$ is ample by Kleiman's criterion, hence $\chi(f^*L^{-b} \otimes L^a) > 0$. Finally, since $P_{\RI{f}f}^a$ is monic, $P_{\RI{f}f}^a\!\pt{\frac{a}{b}}$ is also positive for all sufficiently large $a>0$, completing the proof.
\end{proof}
For the reader's convenience, we also recall the following theorem, which combines results by Kempf \cite[Theorem 2]{Kem70} and by Mumford \cite[Section 16]{Mum08}. Here, given a line bundle $M$ on $A$, we denote by $H^i(A,M)$ the $i$-th cohomology group of $M$. Recall also that we denote by $K(M)$ the kernel of the homomorphism $\Phi_M: A \rightarrow \widehat{A}$.
\begin{theorem} \label{Teo_Kempf}
Let $M$ and $M'$ be line bundles on an abelian variety $A$ of dimension $g$, with $M$ ample. Consider the polynomial $P_{M,M'}(x) \in \Q[x]$ (of degree $g$) such that
$$P_{M,M'}(n)=\chi(M^n \otimes M')$$
for every $n \in \Z$. Then:
\begin{enumerate}[label=(\roman*)]
\item All roots of $P_{M,M'}$ are real and $\dim K(M')$ is equal to the multiplicity of 0 as a root,
\item (Mumford's vanishing theorem) If $K(M')$ is finite, there is a unique integer $i=i(M')$, with $0 \leq i(M') \leq g$, such that $H^{k}(A,M')=0$ for $k\neq i$ and $H^{i}(A,M')\neq 0$. Moreover, $K(M'^{-1})$ is finite\footnote{This follows from \cite[Lemma 2.4.7 (c)]{BL04}.} and $i(M'^{-1})=g-i(M')$.
\item Counting roots with multiplicities, assume that $P_{M,M'}$ has $N_{-}$ negative roots and $N_{+}$ positive roots, then:
\begin{align*}
  H^{k}(A,M')=0, \quad &\text{ if } 0 \leq k <  N_{+}\\
H^{g-k}(A,M')=0, \quad &\text{ if } 0 \leq k <  N_{-}.
\end{align*}
\end{enumerate}
\end{theorem} 
Finally, we have the following characterization of ample line bundles. 
\begin{proposition}[{\cite[Proposition 4.5.2]{BL04}}]\label{Prop_L_ampio}
A line bundle $M$ on $A$ is ample if and only if $K(M)$ is finite and $H^0(A,M)\neq 0$.
\end{proposition}

\subsection{Proof of Theorem \ref{thm:can_heights_endomorphisms}}\label{sect:proof_thm_heights}

Given an abelian variety $A$ of dimension $g$ defined over $\Qal$, an ample symmetric divisor $D$ and $f \in \en(A)$, let $\a_1, \ldots, \a_g$ be the eigenvalues (counted with multiplicities) of $\rho_a(\RI{f}f)$, where the Rosati involution is defined with respect to the polarization $L=\cO_A(D)$. 

Define $\a^{-}_{D}(f)=\min\!\pg{\a_1, \ldots, \a_g}$ and $\a^{+}_{D}(f)=\max\!\pg{\a_1, \ldots, \a_g}$, as before. Notice that, by Lemma \ref{lemma:eigen_pos}, $\a^{-}_{D}(f)$ is non-negative and it is positive if and only if $f$ is surjective, which is compatible with what we said before. Moreover, $\a^{+}_{D}(f)>0$ for every $f\neq 0$.

\begin{proof}[Proof of Theorem \ref{thm:can_heights_endomorphisms}] 
The claim is trivially true for $f=0$, so we will assume that $f\neq 0$ for the rest of the proof. Let $\l =\frac{a}{b}$ be a rational number, with $b>0$. 

We start by proving the upper bound. 
Consider the line bundle $M= f^* L^{-b} \otimes L^a$. Then, for every $n \in \Z$,
$$P_{L,M}(n)=\chi(L^n \otimes M)= \chi(f^* L^{-b} \otimes L^{n+a})=\chi(L) \cdot b^{g} \cdot P_{\RI{f}f}^a\!\pt{\dfrac{n+a}{b}}$$
by Lemma \ref{Lemma_Lange}. Thus, we have
\begin{align*}
P_{L,M}(x)&= \chi(L) \cdot b^{g} \cdot P_{\RI{f}f}^a\!\pt{\dfrac{x+a}{b}}\\
		  &= \chi(L) \cdot b^g \cdot \prod\limits_{i=1}^g \pt{\dfrac{x+a}{b}-\a_i}\\
		  &= \chi(L) \cdot \prod\limits_{i=1}^g \pt{x-(b\a_i-a)}.
\end{align*}
Combining Proposition \ref{Prop_L_ampio} and Theorem \ref{Teo_Kempf}, we obtain that $M$ is ample if and only if all the roots of $P_{L,M}$ are negative, which is equivalent to saying that $\frac{a}{b}> \a_i$ for every $i=1, \ldots, g$.

This implies that if $\l =\frac{a}{b} > \a^{+}_D(f)$, then the divisor $a D - b f^* D$ is ample and symmetric and therefore, by Proposition \ref{prop:propr_can_heights}, we have
$$a\cdot\widehat{h}_{A,D}(P)-b \cdot \widehat{h}_{A,D}(f(P))=a\cdot \widehat{h}_{A,D}(P)-b \cdot \widehat{h}_{A,f^* D}(P)=\widehat{h}_{A,a D - b f^* D}(P)\geq 0$$
for every $P \in A(\Qal)$, which is equivalent to $\widehat{h}_{A,D}(f(P)) \leq \l \cdot \widehat{h}_{A,D}(P)$. Since this is true for every $\l \in \Q$ such that $\l > \a^{+}_D(f)$, this implies that $\widehat{h}_{A,D}(f(P))\leq \a^{+}_D(f) \cdot \widehat{h}_{A,D}(P)$.

In order to prove the lower bound, we consider the line bundle $M=f^* L^{b} \otimes L^{-a}$. By Theorem \ref{Teo_Kempf} and Proposition \ref{Prop_L_ampio}, $M$ is ample if and only if $K(M)$ is finite and $H^g(A,M^{-1})\neq 0$. Using Lemma \ref{Lemma_Lange} as before, we get that 
$$P_{L,M^{-1}}(x)= \chi(L) \cdot b^{g} \cdot P_{\RI{f}f}^a\!\pt{\dfrac{x+a}{b}}= \chi(L) \cdot \prod\limits_{i=1}^g \pt{x-(b\a_i-a)}.$$
By \cite[Lemma 2.4.7(c)]{BL04}, $K(M)=K(M^{-1})$, and Theorem \ref{Teo_Kempf} implies that $K(M^{-1})$ is finite and $H^g(A,M^{-1})\neq 0$ if and only if all the roots of $P_{L,M^{-1}}$ are positive, that is, if and only if $\frac{a}{b} < \a_i$ for every $i=1, \ldots, g$.

Again, this means that for every $\l =\frac{a}{b} < \a^{-}_D(f)$, the divisor $b f^* D - a D$ is ample and symmetric and thus we have
$$b \cdot \widehat{h}_{A,D}(f(P))-a\cdot\widehat{h}_{A,D}(P)=b \cdot \widehat{h}_{A,f^* D}(P)-a\cdot\widehat{h}_{A,D}(P)=\widehat{h}_{A,b f^* D - a D}(P)\geq 0$$
for every $P \in A(\Qal)$, which is equivalent to $\widehat{h}_{A,D}(f(P)) \geq \l \cdot \widehat{h}_{A,D}(P)$. Since this is true for every $\l \in \Q$ such that $\l < \a^{-}_D(f)$, this implies that $\widehat{h}_{A,D}(f(P))\geq \a^{-}_D(f) \cdot \widehat{h}_{A,D}(P)$.

We now prove that the constants $\a^{-}_D(f), \a^{+}_D(f)$ are optimal. 

Consider the $\Q$-divisor $\l D- f^* D$. Observe that the proof above shows that $\l D- f^* D$ is ample if and only if $\l> \a^+_D(f)$. 
From this we deduce that, if $\l \in \Q$ and $\l < \a^+_D(f)$, then $\l D- f^* D$ is not nef. Otherwise, $(\l + \eps) D- f^* D$ would be ample for every $\eps>0$ \cite[Corollary 1.4.10]{Laz04}, which is impossible for $\eps$ small enough.

Then, assume that $0 \leq \widetilde{\a}< \a^{+}_D(f)$ is such that $\widehat{h}_{A,D}(f(P))\leq \widetilde{\a} \cdot \widehat{h}_{A,D}(P)$ for every $P \in A(\Qal)$. Without loss of generality we can assume that $\widetilde{\a}$ is rational. Then, since $D$ is ample, $f^* D$ is nef and, thus, $f^* D + D$ is ample. Therefore, we have that 
$$\widehat{h}_{A,f^* D+ D}(P)\leq \widehat{h}_{A,(\widetilde{\a}+1)D}(P)$$
from which we can deduce, using \cite[Lemma 4.1]{Lee16}, that 
$$(\widetilde{\a}+1)D-\pt{f^* D+ D}=\widetilde{\a}D-f^* D$$
is nef, which is impossible. 

A similar argument, using the $\Q$-divisor $f^* D -\l D$, shows that one cannot have 
$$\widehat{h}_{A,D}(f(P))\geq \widetilde{\a} \cdot \widehat{h}_{A,D}(P)$$
for some $\widetilde{\a}> \a^{-}_D(f)$ and every $P \in A(\Qal)$.
\end{proof}

\begin{example}
Let $A=E \times E$, where $E$ is an elliptic curve with identity element $O$, and let $D=\pi_1^*(O)+ \pi_2^*(O)$, where $\pi_1$ and $\pi_2$ denote the projections onto the two factors. As above, one has
$$\widehat{h}_{A,D}(P_1,P_2)=\widehat{h}_{E, (O)}(P_1)+\widehat{h}_{E, (O)}(P_2).$$
Consider first the endomorphism $\f\colon A \rightarrow A$ defined by
$$\f(P_1,P_2)=(P_1+P_2,P_1-P_2).$$ 
Then $\r_a(\f)=\begin{psmallmatrix} 1 & 1\\ 1 & -1 \end{psmallmatrix}=\r_a(\RI{\f})$, and hence $\r_a(\RI{\f} \f)=\begin{psmallmatrix} 2 & 0\\ 0 & 2 \end{psmallmatrix}$. Applying Theorem \ref{thm:can_heights_endomorphisms}, we recover the parallelogram identity \cite[Theorem B.5.1(c)]{HS13}:
$$\widehat{h}_{E, (O)}(P_1+P_2)+\widehat{h}_{E, (O)}(P_1-P_2)=\widehat{h}_{A,D}(\f(P_1,P_2))=2\pt{\widehat{h}_{E, (O)}(P_1)+\widehat{h}_{E, (O)}(P_2)}.$$
On the other hand, if we consider the endomorphism $\psi(P_1,P_2)=(P_1+P_2, P_2)$, we get $\r_a(\psi)=\begin{psmallmatrix} 1 & 1\\ 0 & 1 \end{psmallmatrix}=\r_a(\RI{\psi})^{t}$, so that $\r_a(\RI{\psi} \psi)=\begin{psmallmatrix} 1 & 1\\ 1 & 2 \end{psmallmatrix}$, whose eigenvalues are $\frac{3 \pm \sqrt{5}}{2}$. Therefore, by Theorem \ref{thm:can_heights_endomorphisms}, we obtain
$$\widehat{h}_{A,D}(\psi(P_1,P_2))=\widehat{h}_{E, (O)}(P_1+P_2)+\widehat{h}_{E, (O)}(P_2)\leq \frac{3+\sqrt{5}}{2}\pt{\widehat{h}_{E, (O)}(P_1)+\widehat{h}_{E, (O)}(P_2)}.$$
It follows that
$$\widehat{h}_{E, (O)}(P_1+P_2)\leq \frac{3+\sqrt{5}}{2} \cdot \widehat{h}_{E, (O)}(P_1)+ \frac{1+\sqrt{5}}{2}\cdot \widehat{h}_{E, (O)}(P_2).$$
Of course, this inequality can also be deduced from the Cauchy--Schwarz inequality for the Néron--Tate pairing; the point of this example is to illustrate how the constants of Theorem \ref{thm:can_heights_endomorphisms} can be computed and how they depend on the divisor $D$. Indeed, let $D_{a,b}=a\pi_1^*(O)+b\pi_2^*(O)$, with $a,b$ positive integers, and set $s=\frac{a}{b}$. Since the Hermitian form associated with $D_{a,b}$ is a multiple of $\mathrm{diag}(a,b)$, the Rosati involution defined by $D_{a,b}$ gives $\r_a(\RI{\psi})=\begin{psmallmatrix} 1 & 0\\ s & 1 \end{psmallmatrix}$, so that $\r_a(\RI{\psi}\psi)=\begin{psmallmatrix} 1 & 1\\ s & s+1 \end{psmallmatrix}$, whose eigenvalues
$$\frac{s+2\pm\sqrt{s^2+4s}}{2}$$
depend on $s$ (note that $A=E\times E$ is not simple, so Remark \ref{rmk:Albert} does not apply).
\end{example}

\begin{remark}\label{rmk:Albert}
Assume that $A$ is simple. If the endomorphism algebra $\en^0(A)$ is a totally real number field, a totally definite quaternion algebra or a CM field, then the Albert classification \cite[Theorem 2 (p.186)]{Mum08} implies that there is a unique positive involution on $\en^0(A)$. Thus, the Rosati involution associated with any line bundle must be equal to this unique positive involution. Hence, this proves that in those cases the constants $\a^{-}_D(f), \a^{+}_D(f)$ do not depend on $D$, generalizing the above-mentioned result by Naumann.
\end{remark}

Since all the eigenvalues of $\rho_a(\RI{f}f)$ are real and non-negative, 
$$\tr(\r_a(\RI{f} f))=\a_1+ \ldots+ \a_g\geq \max\!\pg{\a_1, \ldots, \a_g}=\a^+_D(f)$$
so we also have the following consequence.

\begin{corollary}\label{cor_height_trace}
Fix an abelian variety $A$ defined over $\Qal$ with an ample symmetric divisor $D$. Then, for every endomorphism $f : A \rightarrow A$, we have that
$$\widehat{h}_{A,D}(f(P))\leq \tr(\r_a(\RI{f} f)) \cdot \widehat{h}_{A,D}(P)$$
for every $P \in A(\Qal)$.
\end{corollary}
\subsection{Height bounds for homomorphisms between abelian varieties}
We can now generalize Theorem \ref{thm:can_heights_endomorphisms} to homomorphisms between different abelian varieties. 

Let $A, B$ be two abelian varieties defined over $\Qal$, let $D_1, D_2$ be two ample symmetric divisors on $A$ and $B$, respectively, with associated line bundles $L_1$ and $L_2$, and let $\phi: A \rightarrow B$ be a homomorphism. As before, it is straightforward to see that the ratio $\widehat{h}_{B,D_2}(\phi(P))/\widehat{h}_{A,D_1}(P)$ must be bounded for non-torsion points $P \in A(\Qal)$ (see, for example, \cite[Lemma 16]{Mas84} for the upper bound). However, if $\ker \phi$ is not finite, then there exists a non-torsion $P \in A(\Qal)$ such that $\phi(P)=O_B$, showing that there is no positive constant $\g_1$ such that $\widehat{h}_{B,D_2}(\phi(P)) \geq \g_1 \cdot \widehat{h}_{A,D_1}(P)$.

Generalizing the definition of Rosati involution (Equation (\ref{eqn:def_rosati_inv})), we set
$$\RI{\phi}\coloneqq\Phi_{L_1}^{-1}\circ\widehat{\phi}\circ\Phi_{L_2} \in \hom^0(B,A),$$
so that $\RI{\phi}\phi \in \en^0(A)$ and $\Phi_{\phi^*L_2}=\widehat{\phi}\,\Phi_{L_2}\phi=\Phi_{L_1}\RI{\phi}\phi$; when $A=B$ and $D_1=D_2$, this is the Rosati involution. We denote by $\a^{-}_{D_1,D_2}(\phi)$ and $\a^{+}_{D_1,D_2}(\phi)$ the minimum and the maximum of the eigenvalues of $\r_a(\RI{\phi}\phi)$ (where, as usual, $\r_a$ is extended to homomorphisms between different abelian varieties).

\begin{theorem}\label{thm:can_heights_isogenies}
Let $A,B$, $D_1, D_2$ be as above and let $\phi: A \rightarrow B$ be a non-zero homomorphism. Then the eigenvalues of $\r_a(\RI{\phi}\phi)$ are real and non-negative, $\a^{-}_{D_1,D_2}(\phi)>0$ if and only if $\ker \phi$ is finite, and
$$\a^{-}_{D_1,D_2}(\phi) \cdot \widehat{h}_{A,D_1}(P)\leq \widehat{h}_{B,D_2}(\phi(P))\leq \a^{+}_{D_1,D_2}(\phi) \cdot \widehat{h}_{A,D_1}(P)$$
for every $P \in A(\Qal)$. Moreover, these constants are the best possible.
\end{theorem}
\begin{proof}
If $\pi_1, \pi_2$ are the projections of $A\times B$ onto $A$ and $B$ respectively, the divisor $D=\pi_1^* D_1 + \pi_2^* D_2$ on $A \times B$ is ample and symmetric, and $\widehat{h}_{A\times B, D}(P,Q)=\widehat{h}_{A,D_1}(P)+ \widehat{h}_{B,D_2}(Q)$ for every $(P,Q) \in (A \times B)(\Qal)$. Let $f \in \en(A\times B)$ be given by $f(P,Q)=(O_A, \phi(P))$. Identifying the dual of $A\times B$ with $\widehat{A}\times\widehat{B}$, we have $\Phi_{\cO(D)}=\Phi_{L_1}\times \Phi_{L_2}$ and $\widehat{f}(\xi,\eta)=(\widehat{\phi}(\eta),O_{\widehat{B}})$, so that $\RI{f}(P,Q)=(\RI{\phi}(Q),O_B)$ and $\RI{f}f(P,Q)=(\RI{\phi}\phi(P),O_B)$. Hence
$$\r_a(\RI{f}f)=\begin{pmatrix} \r_a(\RI{\phi}\phi) & 0\\ 0 & \mathbf{0}\end{pmatrix}.$$
By Lemma \ref{lemma:eigen_pos}, the eigenvalues of $\r_a(\RI{\phi}\phi)$ are real and non-negative and $\a^{+}_D(f)=\a^{+}_{D_1,D_2}(\phi)$. Moreover, the proof of Lemma \ref{lemma:eigen_pos} shows that $\ker \r_a(\RI{f}f)=\ker\r_a(f)$, hence $\ker\r_a(\RI{\phi}\phi)=\ker\r_a(\phi)$. Thus $\a^{-}_{D_1,D_2}(\phi)>0$ if and only if $\r_a(\phi)$ is injective, i.e.\ if and only if $\ker \phi$ is finite.

Applying Theorem \ref{thm:can_heights_endomorphisms} to $f$ and $(P,O_B)$ gives
$$\widehat{h}_{B,D_2}(\phi(P))=\widehat{h}_{A\times B, D}(f(P,O_B))\leq \a^{+}_{D}(f) \cdot \widehat{h}_{A\times B, D}(P,O_B)=\a^{+}_{D_1,D_2}(\phi) \cdot \widehat{h}_{A,D_1}(P).$$
If $0 \leq \g< \a^{+}_{D_1,D_2}(\phi)$, the optimality in Theorem \ref{thm:can_heights_endomorphisms} gives $(P,Q)$ with
$$\widehat{h}_{B,D_2}(\phi(P))> \g\pt{\widehat{h}_{A,D_1}(P)+\widehat{h}_{B,D_2}(Q)}\geq \g \cdot \widehat{h}_{A,D_1}(P),$$
where the last inequality holds since $\g\geq 0$ and $\widehat{h}_{B,D_2}(Q)\geq 0$ by Proposition \ref{prop:propr_can_heights}. Hence, the upper bound is optimal.

We now prove the lower bound. If $\ker \phi$ is infinite, then $\a^{-}_{D_1,D_2}(\phi)=0$, which is optimal since there is a non-torsion point $P\in A(\Qal)$ with $\phi(P)=O_B$. Thus, assume that $\ker\phi$ is finite, let $\iota: B' \hookrightarrow B$ be the inclusion of $B'=\phi(A)$ and write $\phi=\iota\circ\phi'$, with $\phi':A\rightarrow B'$ an isogeny. Since $\widehat{h}_{B',\iota^*D_2}=\widehat{h}_{B,D_2}\circ\iota$ and $\Phi_{\iota^*L_2}=\widehat{\iota}\,\Phi_{L_2}\iota$, we have $\widehat{h}_{B',\iota^*D_2}(\phi'(P))=\widehat{h}_{B,D_2}(\phi(P))$ and $\RI{\phi'}\phi'=\RI{\phi}\phi$, so we may assume that $\phi$ is an isogeny.

Let $e(\phi)$ be the exponent of $\ker \phi$. By \cite[Proposition 1.2.6]{BL04}, there is an isogeny $\psi: B \rightarrow A$ with $\psi \circ \phi= [e(\phi)]_A$ and $\phi \circ \psi =[e(\phi)]_B$. By the upper bound applied to $\psi$ (exchanging the roles of $D_1$ and $D_2$), and since $D_1$ is symmetric,
$$e(\phi)^2 \cdot \widehat{h}_{A,D_1}(P) = \widehat{h}_{A,D_1}(\psi(\phi(P))) \leq \a^{+}_{D_2,D_1}(\psi) \cdot \widehat{h}_{B, D_2}(\phi(P))$$
for every $P \in A(\Qal)$. Now $\RI{\phi}\RI{\psi}=\RI{(\psi\phi)}=[e(\phi)]$ and $\psi\phi=[e(\phi)]$, so that $\r_a(\RI{\phi}\phi)=e(\phi)^2\pt{\r_a(\psi)\r_a(\RI{\psi})}^{-1}$. As $\r_a(\psi)\r_a(\RI{\psi})$ and $\r_a(\RI{\psi}\psi)=\r_a(\RI{\psi})\r_a(\psi)$ have the same eigenvalues, $\a^{-}_{D_1,D_2}(\phi)=e(\phi)^2/\a^{+}_{D_2,D_1}(\psi)$, which gives the lower bound. Finally, if $\g>\a^{-}_{D_1,D_2}(\phi)$, then $e(\phi)^2/\g<\a^{+}_{D_2,D_1}(\psi)$, so by the optimality of the upper bound for $\psi$ there is $Q\in B(\Qal)$ with $\widehat{h}_{A,D_1}(\psi(Q))>\frac{e(\phi)^2}{\g}\,\widehat{h}_{B,D_2}(Q)$. Writing $Q=\phi(P)$, with $P \in A(\Qal)$, we get $\widehat{h}_{B,D_2}(\phi(P))<\g \cdot \widehat{h}_{A,D_1}(P)$.
\end{proof}

Applying this theorem with $B=A$ and $\phi=[1]$, for which $\RI{[1]}=\Phi_{L_1}^{-1}\Phi_{L_2}$, we obtain the following optimal comparison of canonical heights defined by different divisors (see also \cite[Exercise B.3]{HS13} for a slightly more general but ineffective statement).

\begin{corollary}
Let $A$ be an abelian variety defined over $\Qal$ and let $D_1, D_2$ be two ample symmetric divisors on $A$, with associated line bundles $L_1, L_2$. Let $\a^-$ and $\a^+$ be the smallest and the largest eigenvalue of $\r_a(\Phi_{L_1}^{-1}\Phi_{L_2})$. Then $\a^->0$ and
$$\a^- \cdot \widehat{h}_{A,D_1}(P)\leq \widehat{h}_{A,D_2}(P)\leq \a^+ \cdot \widehat{h}_{A,D_1}(P)$$
for every $P \in A(\Qal)$. Moreover, these constants are the best possible.
\end{corollary} 

Lastly, we consider the special case of elliptic curves. Given an elliptic curve $E$, a symmetric ample divisor $D$ and an endomorphism $f \in \en(E)$, we clearly have $\a_D^-(f)=\a_D^+(f)=\deg f$, since $\RI{f}=\widehat{f}$. Thus, Theorem \ref{thm:can_heights_endomorphisms} reduces to the well-known identity $\widehat{h}_{E,D}(f(P))=\deg f \cdot \widehat{h}_{E,D}(P)$ (see for example Section 3.6 of \cite{Ser97}).

However, for elliptic curves Theorem \ref{thm:can_heights_isogenies} gives again an identity instead of an inequality, since in this case $\RI{f}f$ is a constant. We give a more elementary proof of this fact.

\begin{proposition}\label{prop:can_height_isog_EC}
Let $E_1, E_2$ be two elliptic curves defined over $\Qal$, $D_1, D_2$ be two ample symmetric divisors on $E_1, E_2$, respectively, and $f: E_1 \rightarrow E_2$ be an isogeny. Then, we have
$$\widehat{h}_{E_2,D_2}(f(P))=\dfrac{\deg D_2}{\deg D_1}\cdot \deg f \cdot \widehat{h}_{E_1,D_1}(P)$$
for every $P \in E_1(\Qal)$.
\end{proposition}
\begin{proof}
Let $a=\deg D_2 \cdot \deg f$ and $b=\deg D_1$. Then, we have 
$$\deg(a D_1 - b f^* D_2)=a \cdot \deg D_1 - b \cdot \deg f \cdot \deg D_2=0.$$
So the divisor $a D_1 - b f^* D_2$ on $E_1$ is nef. As shown in Proposition \ref{prop:propr_can_heights}, the canonical height associated to a nef symmetric divisor is non-negative, therefore
\begin{align*}
a \cdot \widehat{h}_{E_1,D_1}(P)- b \cdot \widehat{h}_{E_2,D_2}(f(P))&=
\widehat{h}_{E_1,a D_1}(P)- \widehat{h}_{E_1,b f^* D_2}(P)\\
&=\widehat{h}_{E_1, a D_1 - b f^* D_2}(P)\geq 0
\end{align*}
implying that $$\widehat{h}_{E_2,D_2}(f(P))\leq \dfrac{\deg D_2}{\deg D_1}\cdot \deg f \cdot \widehat{h}_{E_1,D_1}(P)$$ since ample divisors on curves have positive degree \cite[Corollary IV.3.3]{Har77}.

Similarly, $\deg( b f^* D_2 - a D_1)=0$, so that the same argument gives
$$\widehat{h}_{E_2,D_2}(f(P))\geq \dfrac{\deg D_2}{\deg D_1}\cdot \deg f \cdot \widehat{h}_{E_1,D_1}(P)$$
concluding the proof.
\end{proof}

\begin{remark}
Since any ample symmetric divisor on an elliptic curve is linearly equivalent to $n (O) + (T)$, where $O$ is the identity element, $n\geq 0$ is an integer and $T$ is a 2-torsion point, one can also prove Proposition \ref{prop:can_height_isog_EC} more directly, by explicitly computing the pull-back $f^*(n (O) + (T))$ (see for example \cite[Proposition 2.3]{Fer26} for the special case $D_1=3(O_1)$ and $D_2=3(O_2)$).
\end{remark}

\section{A height inequality}\label{sect:height_ineq}
The aim of this section is to give a bound on the canonical height of the points $P \in \cC(\Qal)$ in terms of the Faltings height $h_F(\cA_{\pi(P)})$ of the corresponding fiber. To do so, we recall the setting fixed at the end of Section \ref{sec_red} and define some height functions that will be used to prove this bound.

Let $S \subseteq \A_g=\A_{g, \mathbf{1},3}$ be a smooth, irreducible, locally closed curve defined over $\Qal$, let $\cA=\mathfrak{A}_g \times_{\A_g} S$, with $\pi: \cA \rightarrow S$ being the structural morphism, and let $\cC \subseteq \cA$ be an irreducible curve as in Theorem \ref{thm:main_product}. Recall that $\cA$ has a level-3-structure and that there is a principal polarization $\l: \cA \rightarrow \widehat{\cA}$, where $\widehat{\cA}$ denotes the dual abelian scheme of $\cA$.

By \cite[Proposition 27.284]{GW23}, the pullback $(\mathrm{id}_{\cA}, \l)^*\mathscr{P}$ of the Poincaré bundle $\mathscr{P}$ on $\cA\times_S\widehat{\cA}$ is relatively ample, and it is symmetric since $([-1]\times[-1])^*\mathscr{P}\cong\mathscr{P}$. Thus, the line bundle
$$\mathcal{L}=\pt{(\mathrm{id}_{\cA}, \l)^*\mathscr{P}}^{\otimes 3}$$
is relatively very ample (see \cite[Theorem 27.279]{GW23}), symmetric, and $\Phi_{\mathcal{L}}=6\l$. After removing finitely many points of $S$ (Remark \ref{rmk_open_subset}), we may assume that $\pi_*\mathcal{L}$ is free, so that $\mathcal{L}$ gives an embedding $\cA \hookrightarrow \PP^n_{S} \cong \PP^n_{\Qal} \times S$. Moreover, for every fiber $\cA_s$ of $\cA \rightarrow S$, the induced closed immersion $\cA_s \rightarrow \PP^n_{\Qal}$ comes from the restriction $\mathcal{L}_s=\mathcal{L}\vert_{\cA_s}$.

The minimal compactification $\overline{\A_{g,\mathbf{1},3}}$ of $\A_{g, \mathbf{1},3}$ can be realized as a closed subvariety of some projective space $\PP^m_{\Qal}$ and we define $\mathcal{M}=\cO_{\PP^{m}}(1)\vert_{\overline{\A_{g,\mathbf{1},3}}}$. Thus, we obtain an embedding $\A_{g, \mathbf{1},3} \hookrightarrow \PP^m_{\Qal}$ and we denote by $\overline{S}$ the Zariski closure of $S$ in $\overline{\A_{g,\mathbf{1},3}}\subseteq \PP^m_{\Qal}$.

We then denote by $\overline{\cA}$ the Zariski closure of $\cA$ inside $\PP^n_{\Qal} \times \overline{S} \subseteq \PP^n_{\Qal} \times \PP^m_{\Qal}$ and let $\overline{\mathcal{L}}=\cO(1,1)\vert_{\overline{\cA}}$, whose restriction to $\cA$ is $\mathcal{L} \otimes \pi^*\!\pt{\mathcal{M}\vert_{S}}$. For $P \in \cA(\Qal)$, write $P=(x_P,\pi(P))$, with $x_P \in \PP^n(\Qal)$, via the embedding $\cA \hookrightarrow \PP^n_{\Qal}\times S$ fixed above. We set
$$h_{\cA_{\pi(P)}, \mathcal{L}_{\pi(P)}}(P)=h(x_P) \qquad \text{and} \qquad h_{\overline{S}, \mathcal{M}\vert_{\overline{S}}}(s)=h(s)$$
for every $P \in \cA(\Qal)$ and every $s \in \overline{S}(\Qal) \subseteq \PP^m(\Qal)$, where $h$ is the projective logarithmic Weil height defined in Section \ref{sect:prelim_heights}. Note that $h_{\cA_{\pi(P)}, \mathcal{L}_{\pi(P)}}$ is a Weil height on $\cA_{\pi(P)}$ associated with $\mathcal{L}_{\pi(P)}$, given by the same embedding for every fiber. We then define the \emph{total height} on $\cA(\Qal)$ as
$$h_{\overline{\cA}, \overline{\mathcal{L}}}(P)=h_{\cA_{\pi(P)}, \mathcal{L}_{\pi(P)}}(P)+ h_{\overline{S}, \mathcal{M}\vert_{\overline{S}}}(\pi(P)),$$
which is the restriction to $\cA(\Qal)$ of the Weil height $(x,y)\mapsto h(x)+h(y)$ on $\overline{\cA}(\Qal)\subseteq \PP^n(\Qal)\times\PP^m(\Qal)$ associated with $\overline{\mathcal{L}}$ (see e.g.\ \cite[Theorem B.3.6]{HS13}). We still denote the latter by $h_{\overline{\cA}, \overline{\mathcal{L}}}$.
Moreover, as $\mathcal{L}$ is symmetric, we can also define a fiberwise canonical height $\widehat{h}_{\cA_{\pi(P)}, \mathcal{L}_{\pi(P)}}(P)$ as in Section \ref{sect:prelim_heights}.

Furthermore, recall that the coarse moduli space $\A_{g,\mathbf{1}}$ of principally polarized abelian varieties of dimension $g$ is a quasi-projective variety. Assume that $g\geq 2$ (the case $g=1$ is treated in the proof of Proposition \ref{prop:can_height_points}). By \cite[Section II.3]{FW12}, there are an integer $r>0$ and a line bundle $L$ on $\overline{\A_{g,\mathbf{1}}}$, whose restriction to $\A_{g,\mathbf{1}}$ is the $r$-th power of the Hodge bundle, whose sections define an embedding of the minimal compactification $\overline{\A_{g,\mathbf{1}}}$ into some projective space $\PP^{\ell}_{\Qal}$ with $L=\cO_{\PP^{\ell}}(1)\vert_{\overline{\A_{g,\mathbf{1}}}}$. We fix this embedding. Then, $L$ has a natural Hermitian metric on $\A_{g,\mathbf{1}}$ with logarithmic singularities along $\overline{\A_{g,\mathbf{1}}}\setminus \A_{g,\mathbf{1}}$ (see \cite[Section II.3]{FW12} and \cite[Chapter I, Corollary 6.2]{FW12}).

 Hence, we can define two height functions: $h_{L}$ on $\A_{g,\mathbf{1}}$ using this metric; and $\widetilde{h}_{L}$ on $\overline{\A_{g,\mathbf{1}}}$ given by the Hermitian metric which at the archimedean places is the standard Fubini--Study metric coming from the embedding of $\overline{\A_{g, \mathbf{1}}}$ into $\PP^{\ell}_{\Qal}$ and at the non-archimedean places is the usual metric. Note that $\widetilde{h}_L$ differs from a fixed Weil height $h_{\overline{\A_{g,\mathbf{1}}}, L}$ by a bounded function on $\PP^{\ell}(\Qal)$ (see \cite[Remark 2.8.3]{BG06} or \cite[Example B.10.5]{HS13}).

In the proof of the next proposition we compare Weil heights along non-constant morphisms of curves, using the following elementary lemma.

\begin{lemma}\label{lemma:heights_curves}
Let $X, Y$ be irreducible projective curves defined over $\Qal$, let $L_X$ and $L_Y$ be ample line bundles on $X$ and $Y$, respectively, and fix Weil heights $h_{X,L_X}$ and $h_{Y,L_Y}$. Let $U \subseteq X$ be a non-empty open subset and let $\f: U \rightarrow Y$ be a non-constant morphism. Then, there exist positive constants $c_1, c_2$ such that
$$h_{X,L_X}(P)\leq c_1 \cdot h_{Y,L_Y}(\f(P))+c_2$$
for every $P \in U(\Qal)$.
\end{lemma}
\begin{proof}
Let $\nu: \widetilde{X} \rightarrow X$ be the normalization of $X$. By \cite[Proposition I.6.8]{Har77}, $\f\circ\nu: \nu^{-1}(U)\rightarrow Y$ extends to a morphism $\widetilde{\f}: \widetilde{X}\rightarrow Y$, which is non-constant and hence finite \cite[Proposition II.6.8]{Har77}. As $\nu$ and $\widetilde{\f}$ are finite, $\nu^*L_X$ and $\widetilde{\f}^*L_Y$ are ample \cite[Exercise III.5.7(d)]{Har77}, so they have positive degree. Choosing an integer $c_1$ with $c_1\deg(\widetilde{\f}^*L_Y)>\deg(\nu^*L_X)$, the line bundle $\widetilde{\f}^*L_Y^{\otimes c_1}\otimes \nu^*L_X^{-1}$ has positive degree, hence it is ample \cite[Corollary IV.3.3]{Har77}, and its Weil height is bounded from below on $\widetilde{X}(\Qal)$ by $-c_2$, for some $c_2 \geq 0$. By the functoriality and additivity of Weil heights \cite[Theorem B.3.2]{HS13}, up to enlarging $c_2$, we get $h_{X,L_X}(\nu(Q))\leq c_1 \cdot h_{Y,L_Y}(\widetilde{\f}(Q))+c_2$ for every $Q \in \widetilde{X}(\Qal)$. If $P\in U(\Qal)$ and $Q\in\nu^{-1}(P)$, then $\widetilde{\f}(Q)=\f(P)$, which proves the claim.
\end{proof}

From this point forward, $\xi_1, \xi_2, \ldots$ will be positive constants depending only on $g$, $S$, $\cA$, $\cC$ and the choices of the various Weil heights, unless otherwise specified.

\begin{proposition}\label{prop:can_height_points}
There exist positive constants $\xi_1, \xi_2$ such that
$$\widehat{h}_{\cA_{\pi(P)},\mathcal{L}_{\pi(P)}}(P)\leq \xi_1 \cdot h_F(\cA_{\pi(P)})+\xi_2$$
for every $P \in \cC(\Qal)$.
\end{proposition}
\begin{proof}
Assume first that $g\geq 2$. By \cite[Theorem II.3.1]{FW12}, there exist positive constants $\xi_3, \xi_4$, depending only on $g$ and on the choices made above, such that
$$\abs{h_{L}\pt{\pq{A}}- \xi_3 \cdot h_F(A)} \leq \xi_4$$
for every principally polarized abelian variety $A/\Qal$ of dimension $g$. Indeed, \cite[Theorem II.3.1]{FW12} is stated for $A$ with semistable reduction over a number field $K$, but both sides are invariant under finite extensions of $K$, and the height $h(A)$ of \cite[Section II.3]{FW12} is the stable Faltings height of \cite{Fal83}. Moreover, by the last displayed inequality in the proof of \cite[Theorem II.1.2]{FW12}, there are positive constants $\xi_5, \xi_6$ such that
$$\abs{h_{L}\pt{\pq{A}}- \widetilde{h}_L\pt{\pq{A}}} \leq \xi_5+ \xi_6 \log \max \pg{1, \widetilde{h}_L\pt{\pq{A}}}$$
for each $[A] \in \A_{g, \mathbf{1}}(\Qal)$.\footnote{The proof in \cite{FW12} bounds each archimedean place separately, which a priori gives constants depending on $\log[K:\Q]$, where $K$ is a field of definition of $[A]$. The bound above, with $\xi_5,\xi_6$ independent of $K$, follows by averaging: if $f_1, \ldots, f_m$ are sections of $\cO(1)$ over $\Z$ cutting out the boundary, normalized so that $\nm{f_i}_1\leq 1$, and $t_{\s}=-\log\max_i\nm{f_i}_1(\s([A]))$ for $\s:K\hookrightarrow\C$, then $\sum_{\s}t_{\s}\leq [K:\Q]\cdot\widetilde{h}_L\pt{\pq{A}}$, and the concavity of $t\mapsto\log(1+t)$ gives $\frac{1}{[K:\Q]}\sum_{\s}\log(1+t_{\s})\leq\log\!\pt{1+\widetilde{h}_L\pt{\pq{A}}}$.} In particular, the last inequality implies that $\widetilde{h}_L\pt{\pq{A}} \ll h_{L}\pt{\pq{A}}+1$, which combined with the first inequality yields $\widetilde{h}_L\pt{\pq{A}} \ll h_{F}\pt{A}+1$. As noted above, $\widetilde{h}_L$ differs from $h_{\overline{\A_{g,\mathbf{1}}}, L}$ by a bounded function, so we get
\begin{equation}\label{eqn:bound_WH_FH}
h_{\overline{\A_{g,\mathbf{1}}}, L}\pt{\pq{A}} \leq \xi_7 \cdot h_F(A)+ \xi_8
\end{equation}
for every principally polarized $A/\Qal$ of dimension $g$, where $\xi_7, \xi_8$ depend only on $g$ and the choice of the Weil height $h_{\overline{\A_{g,\mathbf{1}}}, L}$. If $g=1$, we instead identify $\overline{\A_{1,\mathbf{1}}}$ with $\PP^1$ via the $j$-invariant and take $L=\cO_{\PP^1}(1)$. In this case, \cite[Proposition 2.1]{Sil86} gives
$$\abs{h(j_E)-12 \cdot h_F(E)}\leq 6\log(1+h(j_E))+O(1)$$
for every elliptic curve $E/\Qal$ with $j$-invariant $j_E$ (as before, we may assume that $E$ has semistable reduction), where the implied constant is absolute. Since $6\log(1+t)\leq \frac{t}{2}+O(1)$ for $t\geq 0$, this gives (\ref{eqn:bound_WH_FH}) also for $g=1$.

Let $\r: \A_{g, \mathbf{1},3} \rightarrow \A_{g,\mathbf{1}}$ be the natural morphism which forgets the level structure. Since $\r$ has finite fibers, the Zariski closure $S'$ of $\r(S)$ in $\overline{\A_{g,\mathbf{1}}}$ is a curve, and the restriction of $h_{\overline{\A_{g,\mathbf{1}}}, L}$ to $S'$ is a Weil height for $L\vert_{S'}$. Thus, Lemma \ref{lemma:heights_curves} applied to $X=\overline{S}$, $U=S$, $Y=S'$ and $\f=\r\vert_{S}$ yields
$$h_{\overline{S}, \mathcal{M}\vert_{\overline{S}}}(s) \leq \xi_{9} \cdot h_{\overline{\A_{g,\mathbf{1}}}, L}(\r(s))+ \xi_{10}$$
for every $s \in S(\Qal)$.

 Combining this with (\ref{eqn:bound_WH_FH}) gives
\begin{equation}\label{eqn:bound_hSbar_FalH}
h_{\overline{S}, \mathcal{M}\vert_{\overline{S}}}(s) \leq \xi_{11} \cdot h_{F}(\cA_s)+ \xi_{12}
\end{equation}
for every $s \in S(\Qal)$ and for some positive constants $\xi_{11}, \xi_{12}$. Here we used that $\r(s)$ is the point of $\A_{g,\mathbf{1}}$ corresponding to the principally polarized abelian variety $\cA_s$.

Now, let $\overline{\cC}$ be the Zariski closure of $\cC$ inside $\overline{\cA} \subseteq \PP^n_{\Qal} \times \PP^m_{\Qal}$, so that the restriction of the total height $h_{\overline{\cA}, \overline{\mathcal{L}}}$ to $\overline{\cC}$ is a Weil height for the ample line bundle $\overline{\mathcal{L}}\vert_{\overline{\cC}}$. As $\cC$ is not contained in any fixed fiber of $\cA$, the morphism $\pi\vert_{\cC}: \cC \rightarrow \overline{S}$ is non-constant, and Lemma \ref{lemma:heights_curves} applied to $X=\overline{\cC}$, $U=\cC$, $Y=\overline{S}$ and $\f=\pi\vert_{\cC}$ gives
\begin{equation}\label{eqn:bound_hC_hSbar}
h_{\overline{\cA}, \overline{\mathcal{L}}}(P) \leq \xi_{13} \cdot h_{\overline{S}, \mathcal{M}\vert_{\overline{S}}}(\pi(P))+ \xi_{14}
\end{equation}
for every $P \in \cC(\Qal)$.

Finally, by Theorem A.1 of \cite{DGH21}, there exists a positive constant $\xi_{15}$ such that
$$\widehat{h}_{\cA_{\pi(P)}, \mathcal{L}_{\pi(P)}}(P) \leq h_{\overline{\cA}, \overline{\mathcal{L}}}(P) + \xi_{15} \cdot \max\!\pg{1, h_{\overline{S}, \mathcal{M}\vert_{\overline{S}}}(\pi(P))}$$
for every $P \in \cA(\Qal)$. Combining this with (\ref{eqn:bound_hSbar_FalH}) and (\ref{eqn:bound_hC_hSbar}) we get
$$\widehat{h}_{\cA_{\pi(P)}, \mathcal{L}_{\pi(P)}}(P) \leq \xi_{16} \cdot h_F(\cA_{\pi(P)})+\xi_{17}$$
for some positive constants $\xi_{16}, \xi_{17}$ and for every $P \in \cC(\Qal)$.
\end{proof}

\begin{remark}
In view of a possible effective version of Proposition \ref{prop:can_height_points}, we mention that the constants $\xi_7, \xi_8$ appearing in Equation (\ref{eqn:bound_WH_FH}) can be made effective using Corollary 1.3 of \cite{Paz12}. Indeed, for an abelian variety $A$ with a principal polarization given by an ample symmetric line bundle $M$, the theta height $h_{\Theta}(A,M)$ defined in \cite{Paz12} is equivalent to the Weil height $\widetilde{h}_L$ of the point $[(A, \Phi_M)] \in \overline{\A_{g,\mathbf{1}}}$.
\end{remark}

\section{Arithmetic bounds}\label{sect:arithmetic_estims_CM}
Recall the setting fixed at the end of Section \ref{sec_red}: let $S \subseteq \A_g=\A_{g, \mathbf{1},3}$ be a smooth, irreducible, locally closed curve, and let $\pi: \cA=\mathfrak{A}_g \times_{\A_g} S \rightarrow S$.  Let $\cC$, $B$ and $b$ be as in Theorem \ref{thm:main_product}, with $B$ principally polarized as at the end of Section \ref{sec_red}, and recall the set $\cE_b(\cC)$ from Section \ref{sec_red}. Let $M$ be an ample line bundle on $B$, defined over $\Qal$, such that $\l_B=\Phi_M$, and set $L_B=\pt{M\otimes[-1]^*M}^{\otimes 3}$. Then $L_B$ is ample and symmetric and, since $\Phi_{[-1]^*M}=\widehat{[-1]}\circ\Phi_M\circ[-1]=\Phi_M$, we have $\Phi_{L_B}=6\l_B$.

Assume that $S, \cA, B, b, M$ and $\cC$, as well as the morphism $\r: \A_{g, \mathbf{1},3} \rightarrow \A_{g,\mathbf{1}}$ which forgets the level structure (see the proof of Proposition \ref{prop:can_height_points}), are defined over the same number field $k$, so that $\mathcal{L}$ and $L_B$ are defined over $k$ as well. Notice that, since $\pi\vert_{\cC}$ is non-constant, its fibers are finite. Hence, a point $P \in \cC(\C)$ is algebraic if and only if $\pi(P)$ is, and in this case $\cA_{\pi(P)}$ is defined over $k(\pi(P))$ and
\begin{equation}\label{bound_deg_points}
\pq{k(P):k}\ll \pq{k(\pi(P)):k} \leq \pq{k(P):k}.
\end{equation}
Moreover, since $\cC$ is defined over $\Qal$ and complex abelian varieties with complex multiplication are defined over $\Qal$ (see Proposition 26 from Section 12.4 of \cite{Shi98}), it follows that $\pi(P) \in S(\Qal) \subseteq \A_{g, \mathbf{1},3}(\Qal)$ for every $P \in \cE_b(\cC)$. Hence, $\cE_b(\cC)$ is a subset of $\cC(\Qal)$.

From this point forward, $\t_1, \t_2, \ldots$ will be positive constants depending only on $g$, $h$, $S$, $\cA$, $\cC$, $B$ and $b$, unless otherwise specified.
\begin{lemma}\label{lemma_Falt_height}
Let $A$ be a CM abelian variety of dimension $g$ defined over a number field $K$. Then, there exist positive constants $\t_1, \t_2$ depending only on $g$ such that $h_F(A)\leq \t_1 \cdot \pq{K:\Q}^{\t_2}$.
\end{lemma}
\begin{proof}
By \cite[Theorem 4.1 and Corollary 3.3]{Sil92}, there exists a finite extension $K'/K$ of degree $\pq{K':K}< 2 \cdot (9g)^{2g}$ such that all endomorphisms of $A$ are defined over $K'$. Théorème~6.1 of \cite{Rem17} (see also the remarks following its proof) then guarantees the existence of abelian varieties $A_1,\ldots,A_t$ defined over $K'$ and positive integers $e_1,\ldots,e_t$  with the following properties: each $A_i$ is $\overline{K'}$-simple, the $A_i$ are pairwise non-isogenous over $\overline{K'}$, $\en_{K'}(A_i)=\en_{\overline{K'}}(A_i)$ is a maximal order in $\en^0_{\overline{K'}}(A_i)$, and $A$ is $\overline{K'}$-isogenous to $A'\coloneqq\prod_{i=1}^{t} A_i^{e_i}$. So, there exists an isogeny $\phi: A' \rightarrow A$ with
$$\deg \phi \leq \t_3 \cdot \max\!\pg{h_F(A'), [K':\Q]}^{\t_4},$$
where $\t_3, \t_4$ are positive constants depending only on $g$, by \cite[Théorème 1.4]{GR14}.

Since $A$ has CM, each $A_i$ has CM as well, and we may consider the corresponding primitive CM types $(E_i,\Phi_i)$. Note that $\en_{K'}(A_i)=\cO_{E_i}$ by construction. Then, by Corollary 3.3 of \cite{Tsi18}, there is a positive constant $\t_{5}$ depending only on $g$ such that $h_F(A_i)\leq \abs{\mathrm{Disc}(E_i)}^{\t_{5}}$. In addition, let $g_i=\dim A_i \leq g$, so that $\pq{E_i:\Q}=2g_i$. Since $A_i$ is projective over $K'$, it admits a polarization defined over $K'$, and therefore its field of moduli (as a polarized abelian variety, see the footnote to Theorem 1.2 of \cite{Tsi18}) is contained in $K'$. Hence, Theorem 4.2 of \cite{Tsi18} yields positive constants $\t_6,\t_7$, depending only on $g$, such that $\abs{\mathrm{Disc}(E_i)} \leq \t_6 \cdot \pq{K':\Q}^{\t_7}$. Here we take the maxima over $g_i \leq g$ of the constants given by \cite{Tsi18}, and we observe that they do not depend on the choice of the polarization. 
Combining these two estimates gives 
$$h_F(A_i) \leq \t_8 \cdot \pq{K':\Q}^{\t_{9}}$$
for some positive constants $\t_8, \t_9$. Since for abelian varieties $A$ and $B$ over a number field one has $h_F(A\times B)=h_F(A)+h_F(B)$, it follows that
$$h_F(A')=h_F\pt{\prod_{i=1}^{t} A_i^{e_i}}=\sum_{i=1}^{t} e_i \cdot h_F(A_i)\leq \t_{10} \cdot \pq{K':\Q}^{\t_{9}}.$$ 
Applying \cite[Lemma 5]{Fal83}, we deduce
\begin{align*}
h_F(A) 	&\leq h_F(A')+\dfrac{1}{2}\log(\deg \phi)\\
		&\leq h_F(A') + \dfrac{\t_4}{2}\log\max\!\pg{h_F(A'), [K':\Q]} + \t_{11}\\
		&\leq \t_{12} \cdot \pq{K':\Q}^{\t_{9}}
\end{align*}
where $\t_{12}$ is a positive constant depending only on $g$.

Finally, recalling that $[K':K] \leq 2 \cdot (9g)^{2g}$, we obtain
$$h_F(A) \leq \t_{12} \cdot \pq{K':\Q}^{\t_{9}} \leq \t_{1} \cdot \pq{K:\Q}^{\t_{2}},$$
with $\t_1=\t_{12}\pt{2(9g)^{2g}}^{\t_9}$ and $\t_2=\t_9$, since $[K':\Q]\leq 2(9g)^{2g}[K:\Q]$. Note that $\t_1$ and $\t_2$ depend only on $g$.
\end{proof}

\begin{lemma}\label{lemma:small_Ros_nm}
Let $A$ be an abelian variety of dimension $g$ and let $K$ be a number field such that $A$ and all endomorphisms of $A$ are defined over $K$. Let also $L$ be an ample symmetric line bundle on $A$, defined over $K$, and $P \in A(K)$ be a point contained in the kernel of a non-zero endomorphism of $A$. Then, there exist explicit positive constants $\t_{13}, \t_{14}, \t_{15}$ depending only on $g$, and a non-zero endomorphism $F \in \en(A)=\en_{K}(A)$ such that $F(P)=O$ and 
$$\nm{F}_{Ros}\leq \t_{13} \cdot \pq{K:\Q}^{\t_{14}} \cdot \max\!\pg{\widehat{h}_{A,L}(P),1}^{2g^2} \cdot \max \pg{1, h_F(A),\log \pq{K:\Q}}^{\t_{15}}.$$
\end{lemma}
\begin{proof}
If $P \in A(K)$ is a torsion point, then we can take $F=[\#A(K)_{\mathrm{tors}}]$, whose Rosati norm is $\sqrt{2g}\cdot\# A(K)_{\mathrm{tors}}$. We can bound this quantity using \cite[Théorème 1.2]{GR25}, which yields
$$\nm{F}_{Ros}\leq \sqrt{2g}\cdot(6g)^{8g} \cdot \pq{K:\Q}^g \cdot \max \pg{1, h_F(A),\log \pq{K:\Q}}^{g},$$
so that it suffices to choose $\t_{13}\geq \sqrt{2g}\cdot(6g)^{8g}$ and $\t_{14}, \t_{15}\geq g$. This will be true for the choice of the constants that appear at the end of this proof.

If $P$ is not torsion, we consider the evaluation map
\begin{equation*}
\setlength{\arraycolsep}{0pt}
\renewcommand{\arraystretch}{1.2}
  \begin{array}{ c c c c }
    \epsilon_P :& {} \en(A) & {} \longrightarrow {} & A(K)/A(K)_{\mathrm{tors}} \\
     		   &{} f      & {} \longmapsto {} & f(P)
  \end{array}.
\end{equation*}
We regard $\epsilon_P$ as a homomorphism between the Euclidean lattices $\Lambda_1=\en(A)$ endowed with the Rosati norm $\nm{f}_{Ros}=\sqrt{\tr(\r_r(\RI{f}f))}$, where the Rosati involution is the one defined by the line bundle $L$, and $\Lambda_2=A(K)/A(K)_{\mathrm{tors}}$ with the Néron--Tate norm $\nm{Q}_{NT}= \sqrt{\widehat{h}_{A,L}(Q)}$ (see Example \ref{example:euclidean_lattices}). Since $P$ lies in the kernel of a non-zero endomorphism of $A$ by hypothesis, the kernel of $\epsilon_P$ is non-trivial, and since $P$ is not torsion, $\epsilon_P$ is not the zero map. We therefore seek a non-zero element of $\ker(\epsilon_P)$ of small Rosati norm. 

By Lemma 3.1 of \cite{GR14}, there exists $f_0 \in \ker(\epsilon_P) \setminus \pg{0}$ such that
$$\nm{f_0}_{Ros} \leq \sqrt{\rk \ker(\epsilon_P)} \cdot \mathrm{Vol}(\ker(\epsilon_P))^{\frac{1}{\rk \ker(\epsilon_P)}}.$$
Moreover, $\sqrt{\rk \ker(\epsilon_P)} \leq \sqrt{\rk \en(A)} \leq 2g$, where the last inequality follows from \cite[Proposition 1.2.2]{BL04}, so that
$$\nm{f_0}_{Ros} \leq 2g \cdot \max\!\pg{1,\mathrm{Vol}(\ker(\epsilon_P))}$$
since, by assumption, $\ker(\epsilon_P)$ is non-trivial and thus $\rk \ker(\epsilon_P) \geq 1$. 

Next, consider the lattice $\Lambda_3=\Lambda_1/\ker(\epsilon_P)$ with the quotient norm $\nm{\cdot}_3$. The map $\epsilon_P$ then induces an injective homomorphism 
$$\overline{\epsilon_P}: \Lambda_3 \to \mathrm{im}(\epsilon_P) \subseteq \Lambda_2.$$
For every $\phi \in \en(A)$, we have
$$\nm{\epsilon_P(\phi)}_{NT}= \nm{\overline{\epsilon_P}(\overline{\phi})}_{NT}\leq \nm{\overline{\epsilon_P}}_{\mathrm{op}} \cdot \nm{\overline{\phi}}_3$$
where $\nm{\cdot}_{\mathrm{op}}$ denotes the operator norm.

Applying Lemma 3.1 of \cite{GR14} once again, there exists a non-zero element $\overline{\psi}\in\Lambda_3$ such that
$$\nm{\overline{\psi}}_{3} \leq \sqrt{\rk \Lambda_3} \cdot \mathrm{Vol}(\Lambda_3)^{\frac{1}{\rk \Lambda_3}}.$$
Since $\Lambda_3 \cong \mathrm{im}(\epsilon_P)\subseteq\Lambda_2$ is torsion-free, $\ker(\epsilon_P)$ is saturated in $\Lambda_1$, so Equation (\ref{eqn:vol_quotient}) gives $\mathrm{Vol}(\Lambda_3)=\mathrm{Vol}(\Lambda_1)/\mathrm{Vol}(\ker(\epsilon_P))$.

Since $\overline{\psi}\neq 0$ in $\Lambda_1/\ker(\epsilon_P)$, any lift $\psi\in\Lambda_1$ is not contained in $\ker(\epsilon_P)$, thus $\epsilon_P(\psi)=\psi(P)\neq 0$. Hence, recalling the notation $\l_1(\Lambda)$ for the first minimum of a Euclidean lattice $\Lambda$, we have
\begin{align*}
\l_1(\Lambda_2)
&\leq \nm{\epsilon_P(\psi)}_{NT}= \nm{\overline{\epsilon_P}(\overline{\psi})}_{NT}\\
&\leq \nm{\overline{\epsilon_P}}_{\mathrm{op}}\,\nm{\overline{\psi}}_3\\
&\leq \nm{\overline{\epsilon_P}}_{\mathrm{op}}\,\sqrt{\rk \Lambda_3}\,
\left(\frac{\mathrm{Vol}(\Lambda_1)}{\mathrm{Vol}(\ker(\epsilon_P))}\right)^{1/\rk \Lambda_3},
\end{align*}
which implies
$$\mathrm{Vol}(\ker (\epsilon_P))\leq \pt{\nm{\overline{\epsilon_P}}_{\mathrm{op}} \cdot \sqrt{\rk \Lambda_3}}^{\rk \Lambda_3}  \cdot \dfrac{\mathrm{Vol}(\Lambda_1)}{\l_1(\Lambda_2)^{\rk \Lambda_3}}.$$
Since $1 \leq \rk \Lambda_3\leq \rk \Lambda_1 \leq 4g^2$, this gives
$$\max\!\pg{1,\mathrm{Vol}(\ker (\epsilon_P))}\leq \max\!\pg{1,2g\nm{\overline{\epsilon_P}}_{\mathrm{op}}}^{4g^2}\cdot \max\!\pg{1,\mathrm{Vol}(\Lambda_1)} \cdot \max\!\pg{1,\l_1(\Lambda_2)^{-1}}^{4g^2}.$$

We shall estimate separately the operator norm $\nm{\overline{\epsilon_P}}_{\mathrm{op}}$, $\mathrm{Vol}(\Lambda_1)$, and the first minimum $\l_1(\Lambda_2)$.

By Théorème 1.9(3) of \cite{GR23}, we have
$$\mathrm{Vol}(\Lambda_1)=\mathrm{Vol}(\en_{\overline{K}}(A))\leq (7g)^{8g^5}\pt{[K:\Q]\max\!\pg{1, h_F(A),\log [K:\Q]}}^{g^3}.$$
Furthermore, by Théorème 1.3 of \cite{GR25},
\begin{align*}
\l_1(\Lambda_2)=\l_1(A(K)/A(K)_{\mathrm{tors}})&=\min\pg{\sqrt{\widehat{h}_{A,L}(Q)}: Q \in A(K)\setminus A(K)_{\mathrm{tors}}}\\
&\geq \dfrac{1}{(6g)^{8g}\cdot [K:\Q]^{g+1/2}\max\!\pg{1, h_F(A),\log [K:\Q]}^{g}}.
\end{align*}
Finally, by Lemma \ref{lemma:operator_norms_equal}, we have that $\nm{\overline{\epsilon_P}}_{\mathrm{op}}= \nm{\epsilon_P}_{\mathrm{op}}=\sup\limits_{\varphi \in (\Lambda_1)_{\R} \setminus \pg{0}}\pt{\dfrac{\nm{\epsilon_P(\varphi)}_{NT}}{\nm{\varphi}_{Ros}}}$. By Corollary \ref{cor_height_trace}, for every $\varphi \in \en(A)$
$$\nm{\epsilon_P(\varphi)}_{NT}^2=\widehat{h}_{A,L}(\varphi(P)) \leq \tr(\r_a(\RI{\f}\f)) \cdot \widehat{h}_{A,L}(P)=\dfrac{1}{2}\nm{\f}^2_{Ros}\cdot \widehat{h}_{A,L}(P)$$
since by Lemma 5.1.4 and Corollary 5.1.3 (b) of \cite{BL04}, we have $\mathrm{tr}(\rho_r(\RI{\f}\f)) = 2\,\mathrm{tr}(\rho_a(\RI{\f}\f))$. By homogeneity, this inequality extends to $\en^0(A)$ and then, by continuity, to $\en(A)\otimes \R$, so that
$$\nm{\overline{\epsilon_P}}_{\mathrm{op}}\leq \dfrac{\sqrt{\widehat{h}_{A,L}(P)}}{\sqrt{2}}.$$
Therefore, combining the above estimates, we obtain
\begin{align*}
\nm{f_0}_{Ros} &\leq 2g \cdot \max\!\pg{1,\mathrm{Vol}(\ker(\epsilon_P))}\\
&\leq 2g \cdot \max\!\pg{1,2g\nm{\overline{\epsilon_P}}_{\mathrm{op}}}^{4g^2}\cdot \max\!\pg{1,\mathrm{Vol}(\Lambda_1)} \cdot \max\!\pg{1,\l_1(\Lambda_2)^{-1}}^{4g^2}\\
&\leq \t_{16} \cdot \pq{K:\Q}^{2g^2+5g^3} \cdot \max\!\pg{\widehat{h}_{A,L}(P),1}^{2g^2} \cdot \max\!\pg{1, h_F(A),\log \pq{K:\Q}}^{5g^3}
\end{align*} 
where $\t_{16}=(2^{2g^2+1}\cdot 6^{32g^3} \cdot 7^{8g^5}) \cdot g^{1+4g^2+32g^3+8g^5}$ (note that 
$$\max\!\pg{1,2g\nm{\overline{\epsilon_P}}_{\mathrm{op}}}\leq \sqrt{2}\,g\max\!\pg{\widehat{h}_{A,L}(P),1}^{1/2}$$
and that the bounds above for $\mathrm{Vol}(\Lambda_1)$ and $\l_1(\Lambda_2)^{-1}$ are at least $1$). At this stage, $f_0\in\ker(\epsilon_P)$ only implies that $f_0(P) \in A(K)_{\mathrm{tors}}$. To obtain a non-zero $F \in \en(A)$ with $F(P)=O$, it is enough to compose $f_0$ with the multiplication-by-$\#A(K)_{\mathrm{tors}}$ map. Thus, using again \cite[Théorème 1.2]{GR25}, we get
\begin{align*}
\nm{F}_{Ros}&=\nm{[\# A(K)_{\mathrm{tors}}] \circ f_0}_{Ros}=\# A(K)_{\mathrm{tors}} \cdot \nm{f_0}_{Ros}\\
&\leq \t_{13} \cdot \pq{K:\Q}^{g+2g^2+5g^3} \cdot \max\!\pg{\widehat{h}_{A,L}(P),1}^{2g^2} \cdot \max\!\pg{1, h_F(A),\log \pq{K:\Q}}^{g+5g^3}
\end{align*}
where $\t_{13}=6^{8g} \cdot g^{8g} \cdot \t_{16}$. This concludes the proof upon taking $\t_{14}=g+2g^2+5g^3$ and $\t_{15}=g+5g^3$.
\end{proof}

\begin{remark}
Note that a similar result could be obtained by fixing a particular basis of the additive group $\en_{\overline{K}}(A)$ using \cite[Lemma 5.1]{MW94}, and then showing that there exists a $\Z$-linear combination of the basis elements with small coefficients that vanishes at $P$, by applying \cite[Proposition 6.1]{BC20}. For more details, see Lemma 4.27 of the author's PhD thesis \cite{Fer25}. Although this alternative approach is somewhat shorter, the argument presented above has the advantage of yielding completely explicit constants.
\end{remark}

\begin{corollary}\label{cor:small_endom}
If $P_0 \in \cE_b(\cC)$, then there exists $f_{_{P_0}}\! \in \en(\cA_{\pi(P_0)}\times B)$ whose restriction to $\cA_{\pi(P_0)}\times\pg{O_B}$ is non-zero, such that $f_{_{P_0}}\!\pt{P_0,b}=O$ and
$$\nm{f_{_{P_0}}}_{Ros} \leq \t_{17} \pq{k(P_0):\Q}^{\t_{18}},$$
where the Rosati involution is the one defined by $\mathcal{L}_{\pi(P_0)}\boxtimes L_B\coloneqq p_1^*\mathcal{L}_{\pi(P_0)}\otimes p_2^*L_B$ and $p_1, p_2$ are the two projections on $\cA_{\pi(P_0)}$ and $B$, respectively.
\end{corollary}
\begin{proof}
Let $A=\cA_{\pi(P_0)}\times B$, which has dimension $g+h$ and is defined over $k(\pi(P_0))$. By \cite[Theorem 4.1 and Corollary 3.3]{Sil92}, there exists a finite extension $K'/k(\pi(P_0))$ of degree $\pq{K':k(\pi(P_0))}<2(9(g+h))^{2(g+h)}$ over which all endomorphisms of $A$ are defined, and we let $K=k(P_0)K'$. Since $(P_0,b)\in A(K)$ lies in a proper algebraic subgroup of $A$, it lies in the kernel of a non-zero endomorphism. Lemma \ref{lemma:small_Ros_nm}, applied with $L=\mathcal{L}_{\pi(P_0)}\boxtimes L_B$ and $P=(P_0,b)$ (with $g+h$ in place of $g$ in the constants $\t_{13},\t_{14},\t_{15}$), then gives a non-zero $f_{_{P_0}}\!\in\en_K(A)$ with $f_{_{P_0}}\!(P_0,b)=O$ and
$$\nm{f_{_{P_0}}}_{Ros}\leq \t_{13}[K:\Q]^{\t_{14}}\max\!\pg{\widehat{h}_{A,L}(P_0,b),1}^{2(g+h)^2}\max\!\pg{1,h_F(A),\log[K:\Q]}^{\t_{15}}.$$
By Proposition \ref{prop:propr_can_heights}, $\widehat{h}_{A,L}(P_0,b)=\widehat{h}_{\cA_{\pi(P_0)},\mathcal{L}_{\pi(P_0)}}(P_0)+\widehat{h}_{B,L_B}(b)$, and $h_F(A)=h_F(\cA_{\pi(P_0)})+h_F(B)$, where $\widehat{h}_{B,L_B}(b)$ and $h_F(B)$ do not depend on $P_0$. Moreover, Proposition \ref{prop:can_height_points} and Lemma \ref{lemma_Falt_height} show that $\widehat{h}_{\cA_{\pi(P_0)},\mathcal{L}_{\pi(P_0)}}(P_0)$ and $h_F(\cA_{\pi(P_0)})$ are bounded above by a constant times a power of $[K:\Q]$. Since $[K:\Q]\leq 2(9(g+h))^{2(g+h)}[k(P_0):\Q]$, every factor in the bound above is polynomially bounded in terms of $[k(P_0):\Q]$. This gives the claimed bound for suitable $\t_{17},\t_{18}$.

Finally, if $f_{_{P_0}}\!$ vanished on $\cA_{\pi(P_0)}\times\pg{O_B}$, then $f_{_{P_0}}\!(x,y)=\f(y)$ for the non-zero homomorphism $\f:B\to A$, $y \mapsto f_{_{P_0}}\!(O,y)$. Then $\f(b)=O$, so $b$ would lie in the proper algebraic subgroup $\ker\f$ of $B$, a contradiction.
\end{proof}

Now, let $P_0 \in \cE_b(\cC)$ and let $\tau_{P_0}\in u_b^{-1}(\pi(P_0)) \cap \mathfrak{F}_{\Gamma}$, where $\Gamma=\Gamma(3)$, $\mathfrak{F}_{\Gamma}$ and the uniformization map $u_b:\HH_g \to \A_{g, \mathbf{1},3}(\C)$ were introduced in Section~\ref{sec_moduli}. The set $u_b^{-1}(\pi(P_0)) \cap \mathfrak{F}_{\Gamma}$ is finite, and all the results below hold for every choice of $\tau_{P_0}$ in it. Let $Z_{P_0} \in \mathfrak{F}_g$ be in the $\mathrm{Sp}_{2g}(\Z)$-orbit of $\tau_{P_0}$, and write $\tau_{P_0}=\s \cdot Z_{P_0}$ with $\s \in \pg{\s_1, \ldots, \s_m}$.

By construction of $u$, the map $z \mapsto u(\tau_{P_0},z)$ induces an isomorphism of principally polarized abelian varieties between $\C^g/(\Z^g+\tau_{P_0}\Z^g)$ and $\cA_{\pi(P_0)}$, i.e.\ there is a symplectic basis of the period lattice of $\cA_{\pi(P_0)}$ with period matrix $(\tau_{P_0},\mathbf{1})$. Together with the basis of $B$ fixed in Section \ref{sect:main_estim_CM}, this gives a symplectic basis of the period lattice of $\cA_{\pi(P_0)}\times B$ (with respect to $\l_{\pi(P_0)}\times\l_B$) with period matrix $(\widetilde{\tau}_{P_0},\mathbf{1}_{g+h})$, where $\widetilde{\tau}_{P_0}=\mathrm{diag}(\tau_{P_0},\tau_B)$. In the sequel, all analytic and rational representations of endomorphisms of $\cA_{\pi(P_0)}\times B$ are taken with respect to this basis. In particular, if $f \in \en(\cA_{\pi(P_0)}\times B)$ and $(\tau_{P_0},z) \in u^{-1}(P_0)\cap \mathcal{F}_g$, then $f(P_0,b)=O$ if and only if $\r_a(f)\widetilde{z} \in \Z^{g+h}+\widetilde{\tau}_{P_0}\Z^{g+h}$, and the first $g$ columns of $\r_a(f)$ form the analytic representation of the restriction of $f$ to $\cA_{\pi(P_0)}\times\pg{O_B}$.

Since $Z_{P_0}$ and $\tau_{P_0}$ are in the same $\mathrm{Sp}_{2g}(\Z)$-orbit, $(Z_{P_0},\mathbf{1})$ is also a period matrix of $\cA_{\pi(P_0)}$, with respect to another symplectic basis. As $\cA_{\pi(P_0)}$ has CM, Lemma \ref{lemma:degree_CM_periods} implies that all the entries of $\tau_{P_0}$ and of $Z_{P_0}$ are algebraic numbers of degree at most $\kappa_g=2^g$. Since $\cA_{\pi(P_0)}\times B$ has CM, Lemma \ref{lemma:degree_CM_periods} applied to it also shows that the entries of $\widetilde\tau_{P_0}$ have degree at most $\kappa=2^{g+h}$.

We now establish bounds for the heights of $\tau_{P_0}$ and $Z_{P_0}$.
\begin{lemma}\label{lemma:ht_tau_Ztau}
Let $P_0 \in \cE_b(\cC)$ and let $\tau_{P_0}$ and $Z_{P_0}$ be as above. Then, there are positive constants $\t_{19}$, $\t_{20}$, $\t_{21}$, $\t_{22}$ such that
$H_{\max}(Z_{P_0})\leq \t_{19} \cdot [k(P_0):\Q]^{\t_{20}}$ and $H_{\max}(\tau_{P_0}) \leq \t_{21} \cdot [k(P_0):\Q]^{\t_{22}}$, where $H_{\max}$ is the entry-wise height on $\mat_g(\Qal)$ defined in Section \ref{sect:prelim_heights}.
\end{lemma}
\begin{proof}
Since $\cA_{\pi(P_0)}$ has CM, $Z_{P_0}$ is a CM point of $\mathfrak{F}_g$. Let $R$ be the center of $\en(\cA_{\pi(P_0)})$ and let $x=\r(\pi(P_0)) \in \A_{g,\mathbf{1}}(\Qal)$, where $\r: \A_{g, \mathbf{1},3} \rightarrow \A_{g,\mathbf{1}}$ is the morphism which forgets the level structure (see the proof of Proposition \ref{prop:can_height_points}). By Theorem 3.1 of \cite{PT13}, $H_{\max}(Z_{P_0})$ is polynomially bounded in terms of $\abs{\mathrm{Disc}(R)}$, with constants depending only on $g$. Moreover, by Theorem 5.2 of \cite{Tsi18}, there exists $b_g>0$ depending only on $g$ such that $\abs{\mathrm{Disc}(R)} \ll_g \#\pt{\Gal(\Qal/\Q)\cdot x}^{b_g}$. Finally, since $\r$ is defined over $k$, we have $\#\pt{\Gal(\Qal/\Q)\cdot x}=\pq{\Q(x):\Q}\leq \pq{k(\pi(P_0)):\Q}\leq \pq{k(P_0):\Q}$. Hence, there exist positive constants $\t_{19}, \t_{20}$, depending only on $g$, such that
\begin{equation}\label{eqn_ht_ZP0}
H_{\max}(Z_{P_0}) \leq \t_{19} \cdot \#\pt{\Gal(\Qal/\Q)\cdot x}^{\t_{20}} \leq \t_{19} \cdot [k(P_0):\Q]^{\t_{20}}.
\end{equation}
Now, write $\s=\begin{psmallmatrix} A_{\s} & B_{\s}\\ C_{\s} & D_{\s} \end{psmallmatrix}$, so that $\tau_{P_0}= \s \cdot Z_{P_0}=(A_{\s}Z_{P_0}+B_{\s})(C_{\s}Z_{P_0}+D_{\s})^{-1}$, and recall that $\s \in \pg{\s_1, \ldots, \s_m}$.

Then, using Proposition \ref{prop:properties_mat_heights}, we get
\begin{align*}
H_{\max}(\tau_{P_0})&\leq g \cdot H_{\max}(A_{\s}Z_{P_0}+B_{\s})^g \cdot H_{\max}\!\pt{(C_{\s}Z_{P_0}+D_{\s})^{-1}}^g\\
&\ll_g H_{\max}(A_{\s}Z_{P_0})^g H_{\max}(B_{\s})^g \cdot H_{\max}(C_{\s}Z_{P_0})^{2g^4-g^3} H_{\max}(D_{\s})^{2g^4-g^3}\\
&\ll_g H_{\max}(A_{\s})^{g^2} H_{\max}(B_{\s})^g H_{\max}(C_{\s})^{2g^5-g^4} H_{\max}(D_{\s})^{2g^4-g^3} \cdot H_{\max}(Z_{P_0})^{2g^5-g^4+g^2}.
\end{align*}
This implies that there exists a constant $\t_{23}$, depending only on $g$ and $\s$, such that 
$$H_{\max}(\tau_{P_0}) \leq \t_{23} H_{\max}(Z_{P_0})^{2g^5-g^4+g^2}.$$
Taking the maximum of all such constants over all possible choices of $\s \in \pg{\s_1, \ldots, \s_m}$, we get a constant $\t_{24}$ that depends only on $g$ and the choice of $\s_1, \ldots, \s_m$, such that 
$$H_{\max}(\tau_{P_0}) \leq \t_{24} H_{\max}(Z_{P_0})^{2g^5-g^4+g^2}.$$
Finally, substituting the bound (\ref{eqn_ht_ZP0}) for $H_{\max}(Z_{P_0})$, gives the desired bound for $H_{\max}(\tau_{P_0})$.
\end{proof}

\begin{lemma}\label{lemma_2g-ht_anal_rep}
Let $P_0 \in \cE_b(\cC)$, let $\tau_{P_0}$ be as above and let $f_{_{P_0}}$ be the endomorphism given by Corollary \ref{cor:small_endom}. Then $\r_a(f_{_{P_0}}) \in \mat_{g+h}(\Q(\widetilde{\tau}_{P_0}))$ and $H_{\kappa}(\r_a(f_{_{P_0}}))\leq\t_{25}\pq{k(P_0):\Q}^{\t_{26}}$, for some positive constants $\t_{25}, \t_{26}$ independent of the choice of $\tau_{P_0}$.
\end{lemma}
\begin{proof}
Write 
$$\r_r(f_{_{P_0}})=\begin{pmatrix} M_1 & M_2\\ M_3 & M_4 \end{pmatrix},$$
with $M_{\ell} \in \mat_{g+h}(\Z)$ for $\ell=1,2,3,4$. 

Since $(\widetilde{\tau}_{P_0},\mathbf{1}_{g+h})$ is a period matrix of $\cA_{\pi(P_0)}\times B$, Equation (\ref{eqn:endom_matr_periods}) gives $\r_a(f_{_{P_0}})=\widetilde\tau_{P_0} M_2 + M_4$, so that all its entries lie in $\Q(\widetilde\tau_{P_0})$ and have degree at most $\kappa$.

Hence, Proposition \ref{prop:properties_mat_heights} implies
\begin{align*}
H_{\max}(\r_a(f_{_{P_0}}))&\leq 2H_{\max}(\widetilde\tau_{P_0}M_2)H_{\max}(M_4) \\
&\leq 2(g+h) H_{\max}(\widetilde\tau_{P_0})^{g+h} H_{\max}(M_2)^{g+h} H_{\max}(M_4)\\
&\leq 2(g+h) \nmi{\r_r(f_{_{P_0}})}^{g+h+1} H_{\max}(\widetilde\tau_{P_0})^{g+h}
\end{align*}
and, by Lemma \ref{lemma:bounds_H_hpoly},
$$H_{\kappa}(\r_a(f_{_{P_0}}))\leq 2^{\kappa} H_{\max}(\r_a(f_{_{P_0}}))^{\kappa} \leq (4g+4h)^{\kappa} \cdot \nmi{\r_r(f_{_{P_0}})}^{\kappa(g+h+1)} \cdot H_{\max}(\widetilde\tau_{P_0})^{(g+h)\kappa}.$$
Furthermore, since $\Phi_{\mathcal{L}_{\pi(P_0)}\boxtimes L_B}=6(\l_{\pi(P_0)}\times\l_B)$ and $[6]$ is central in $\en^0(\cA_{\pi(P_0)}\times B)$, the Rosati involution in Corollary \ref{cor:small_endom} coincides with the one defined by the principal polarization $\l_{\pi(P_0)}\times\l_B$. By Proposition \ref{prop:bound_rat_rep_general}, applied with $\mathbf{D}=\mathbf{1}_{g+h}$ and the period matrix $(\widetilde{\tau}_{P_0},\mathbf{1}_{g+h})$, we have 
$$\nmi{\r_r(f_{_{P_0}})}\ll N_{\widetilde{\tau}_{P_0}}^{3(g+h)+3}\nm{f_{_{P_0}}}_{Ros},$$
where $N_{\widetilde{\tau}_{P_0}}=\max\!\pg{N_{\tau_{P_0}},N_{\tau_B}}$. Moreover, parts (a), (b) and (d) of Proposition \ref{prop:prop_F_Gamma} give
$$N_{\tau_{P_0}}\leq\max\!\pg{\d_1,\d_2,\d_4}\cdot\max\!\pg{1,\nmi{\Im(Z_{P_0})}}^{2g^2-g+1}.$$
We then use Lemmas \ref{lemma:bound_mod_hpoly}, \ref{lemma:bounds_H_hpoly} and \ref{lemma:ht_tau_Ztau} to get
\begin{equation}\label{eqn:nmi_Im_ZP0}
\begin{aligned}
\nmi{\Im(Z_{P_0})} &\leq \nmi{Z_{P_0}} \leq \sqrt{\kappa_g+1} \cdot H_{\kappa_g}(Z_{P_0})\\
&\leq 2^{\kappa_g}\sqrt{\kappa_g+1} \cdot H_{\max}(Z_{P_0})^{\kappa_g}\\
&\leq \t_{27} \cdot \pq{k(P_0):\Q}^{\t_{28}}
\end{aligned}
\end{equation}
which, combined with the bound for $\nm{f_{_{P_0}}}_{Ros}$ in Corollary \ref{cor:small_endom}, implies that 
$$\nmi{\r_r(f_{_{P_0}})} \leq \t_{29} \cdot \pq{k(P_0):\Q}^{\t_{30}}.$$
Finally, since $H_{\max}(\widetilde{\tau}_{P_0})=\max\!\pg{H_{\max}(\tau_{P_0}),H_{\max}(\tau_B)}$, Lemma \ref{lemma:ht_tau_Ztau} gives $H_{\max}(\widetilde{\tau}_{P_0})\leq \t_{31} \cdot [k(P_0):\Q]^{\t_{22}}$, with $\t_{31}=\max\!\pg{\t_{21},H_{\max}(\tau_B)}$. Hence, we get
\begin{align*}
H_{\kappa}(\r_a(f_{_{P_0}}))&\leq (4g+4h)^{\kappa} \cdot \nmi{\r_r(f_{_{P_0}})}^{\kappa(g+h+1)} \cdot H_{\max}(\widetilde\tau_{P_0})^{(g+h)\kappa}\\
&\leq (4g+4h)^{\kappa} \cdot \t_{32} \cdot \pq{k(P_0):\Q}^{\t_{33}} \cdot \t_{31}^{(g+h)\kappa} \cdot [k(P_0):\Q]^{(g+h)\kappa \t_{22}}\\
&\leq \t_{25} \cdot \pq{k(P_0):\Q}^{\t_{26}}
\end{align*}
concluding the proof, since none of the constants depends on the choice of $\tau_{P_0}$.
\end{proof}

\begin{remark}
It is possible to deduce an upper bound for $\nmi{Z_{P_0}}$ similar to the one given in (\ref{eqn:nmi_Im_ZP0}) by using Proposition C.2 and Equation (C.8) of \cite{Bos96} and Lemma \ref{lemma_Falt_height}.
\end{remark}

\begin{lemma}\label{lemma:bound_det_Im_tauP0}
Let $P_0 \in \cE_b(\cC)$ and let $\tau_{P_0}$ be as above. Then, there are positive constants $\t_{34}$, $\t_{35}$ such that 
$$\det(\Im(\tau_{P_0}))\geq \dfrac{\t_{34}}{[k(P_0):\Q]^{\t_{35}}}.$$
\end{lemma}
\begin{proof}
By Proposition \ref{prop:prop_F_Gamma}, we have that $\det(\Im(\tau_{P_0}))\geq \d_3 \max\!\pg{1, \nmi{\Im(Z_{P_0})}}^{-2g}$. Hence, (\ref{eqn:nmi_Im_ZP0}) implies that
$$\det(\Im(\tau_{P_0}))\geq \dfrac{\d_3}{\max\!\pg{1, \nmi{\Im(Z_{P_0})}}^{2g}} \geq \dfrac{\d_3}{\t_{27}^{2g} \cdot [k(P_0):\Q]^{2g \t_{28}}}$$
which gives the desired bound.
\end{proof}

\section{Proof of Theorem \ref{thm:main_product}}
We need to establish the finiteness of the set $\cE_b(\cC)$, introduced in Section \ref{sec_red}, under the assumptions made at the end of that section. Let $P_0 \in \cE_b(\cC)$ and let $\s \in \Gal(\overline{k}/k)$. Since $\cA$, $\cC$, $B$ and $b$ are defined over $k$, we have $\s(P_0)\in\cC(\Qal)$ and $\s$ induces an isomorphism $\cA_{\pi(P_0)}\times B\to\cA_{\pi(\s(P_0))}\times B$ sending $(P_0,b)$ to $(\s(P_0),b)$, so $\s(P_0)\in\cE_b(\cC)$.

To simplify notation, we set $d_0\coloneqq\pq{k(P_0):\Q}=[k(\s(P_0)):\Q]$. 

Let $(\tau,z) \in u^{-1}(\s(P_0)) \cap \mathcal{F}_g$. Then $\tau \in u_b^{-1}(\pi(\s(P_0))) \cap \mathfrak{F}_{\Gamma}$, so we can take $\tau_{\s(P_0)}=\tau$ in Section \ref{sect:arithmetic_estims_CM}. Let $f_{\s(P_0)}$ be given by Corollary \ref{cor:small_endom} and let $M=\r_a(f_{\s(P_0)})$, computed with respect to the period matrix $(\widetilde{\tau},\mathbf{1}_{g+h})$. Since $f_{\s(P_0)}$ does not vanish on $\cA_{\pi(\s(P_0))}\times\pg{O_B}$, the first $g$ columns of $M$ are not all zero, and since $f_{\s(P_0)}(\s(P_0),b)=O$, we have $M\widetilde{z} \in \Z^{g+h} + \widetilde{\tau} \Z^{g+h}$. Moreover, Lemmas \ref{lemma_2g-ht_anal_rep}, \ref{lemma:bounds_H_hpoly}, \ref{lemma:ht_tau_Ztau} and \ref{lemma:bound_det_Im_tauP0}, together with (\ref{eqn:prop_hpoly}), give
$$H_{\kappa}(M) \leq \t_{25} d_0^{\t_{26}}, \qquad H_{\kappa}(\tau) \leq H_{\kappa_g}(\tau)\leq 2^{\kappa_g} H_{\max}(\tau)^{\kappa_g} \leq \t_{36} d_0^{\t_{37}}, \qquad \det(\Im(\tau)) \geq \frac{\t_{34}}{d_0^{\t_{35}}},$$
with $\t_{36}=2^{\kappa_g}\t_{21}^{\kappa_g}$ and $\t_{37}=\kappa_g\t_{22}$. Hence $(\tau,z) \in \mathcal{Z}(\t d_0^{\eta})$, where $\t=\max\!\pg{\t_{25},\t_{36}, \frac{1}{\t_{34}}}$ and $\eta=\max\!\pg{\t_{26}, \t_{37}, \t_{35}}$.

Since distinct conjugates of $P_0$ give distinct points $(\tau,z)$, we obtain at least $[k(P_0):k]=d_0/[k:\Q]$ points of $\mathcal{Z}(\t d_0^{\eta})$. Applying Proposition \ref{prop:main_estim_CM} with $\eps=\frac{1}{2\eta}$, we get 
$$d_0/[k:\Q]\leq \#\mathcal{Z}(\t d_0^{\eta}) \ll d_0^{1/2},$$
so that $d_0$ is uniformly bounded for all $P_0 \in \cE_b(\cC)$.

Hence, by Lemma \ref{lemma_Falt_height}, the Faltings height $h_F(\cA_{\pi(P_0)})$ is bounded above by a constant independent of $P_0 \in \cE_b(\cC)$. In view of (\ref{eqn:bound_hSbar_FalH}), it follows that the height $h_{\overline{S}, \mathcal{M}\vert_{\overline{S}}}$ is bounded on $\pi(\cE_b(\cC)) \subseteq S(\Qal)$. Consequently, $\pi(\cE_b(\cC)) \subseteq \overline{S}(\Qal)$ is a set of bounded height and bounded degree, as $[k(\pi(P_0)):\Q] \leq d_0$. Since $\mathcal{M}\vert_{\overline{S}}$ is ample, the Northcott property of the Weil height \cite[Theorem 2.4.9]{BG06} ensures that $\pi(\cE_b(\cC))$ is finite. 

Therefore, $\cE_b(\cC)$ is contained in the intersection of $\cC$ with the union of finitely many fibers of $\cA \to S$. As $\cC$ is irreducible and not contained in any fiber, we conclude that $\cE_b(\cC)$ itself is finite. This proves Theorem \ref{thm:main_product} and hence, by Proposition \ref{prop:reduction_transverse_curve}, Theorem \ref{main_thm}.

\section*{AI disclosure}
The author used Claude Opus 5.5 (Anthropic) to review earlier versions of this manuscript, which helped to identify and correct several errors, inaccuracies and typos, and to improve the exposition. The statements and proofs of Lemmas \ref{lemma:degree_CM_periods} and \ref{lemma:basis_in_F_Gamma} were drafted by Claude Opus 5.5. Moreover, Claude Opus 5.5 observed that the arguments now contained in Lemmas \ref{lemma:arc_to_disc} and \ref{lemma:heights_curves} were each needed more than once in the subsequent proofs, and devised these two lemmas so that each argument is given only once. It also helped to fill some gaps in the details of the proofs of Lemmas \ref{lemma:func_transc_CM} and \ref{lemma:product_decomposition} and suggested the averaging argument in the footnote to the proof of Proposition \ref{prop:can_height_points}. All AI-generated material and suggestions were checked and revised by the author. Apart from the assistance described above, the author developed the arguments and wrote the manuscript, and takes full responsibility for its contents.

\section*{Acknowledgements}
We are deeply grateful to Fabrizio Barroero, who proposed this problem, for his constant guidance during the author's PhD and for his advice and suggestions afterwards.

We thank Gabriel Dill for his comments and ideas, which led to the present form of Theorem \ref{main_thm} and, in particular, to Lemma \ref{lemma:product_decomposition}, Theorem \ref{thm:main_product} and Proposition \ref{prop:reduction_transverse_curve}.

We thank Éric Gaudron for explaining the argument used in the proof of Lemma \ref{lemma:small_Ros_nm}, and Fabien Pazuki for pointing out Corollary 1.3 of \cite{Paz12}, which allows one to make some of the constants in Section \ref{sect:height_ineq} effective.

We also thank Gabriel Dill, Davide Lombardo and Francesco Veneziano for their careful reading of earlier versions of this article and for many corrections and suggestions, which significantly improved the exposition.

Finally, we thank Laura Capuano, Amos Turchet, Valerio Talamanca, Francesco Tropeano, Guido Lido, Roberto Vacca, Nelson Alvarado and Nicola Ottolini for many stimulating discussions and helpful comments, especially on the material of Section \ref{sec_can_bounds}.

The author was supported by the PRIN 2022 project \emph{2022HPSNCR: Semiabelian varieties, Galois representations and related Diophantine problems}, the \emph{National Group for Algebraic and Geometric Structures, and their Applications} (GNSAGA INdAM) and by the French National Research Agency (ANR) under the \emph{ANR-23-CE40-0006 GAEC project}. An earlier version of this work appears in the author's PhD thesis \cite{Fer25}.

\bibliographystyle{alpha}
\bibliography{CM_relations+canonical_height_bounds_biblio}

\end{document}